\documentclass[11pt,a4]{article}

\usepackage{latexsym}
\usepackage{graphics}
\usepackage{graphicx}
\usepackage{amsthm}
\usepackage{amsmath}
\usepackage{amsfonts}
\usepackage{amssymb}
\usepackage{amsbsy}
\usepackage{mathrsfs}
\usepackage{ifpdf}
\usepackage{amstext}

\usepackage{framed}
\usepackage{tikz}
\usepackage{enumerate}

\newtheorem{theorem}{Theorem}[section]
\newtheorem{lemma}[theorem]{Lemma}
\newtheorem{proposition}[theorem]{Proposition}

\newtheorem*{thm-P7C4C5-free-decomp}{Theorem~\ref{thm-P7C4C5-free-decomp}}
\newtheorem*{thm-decomp-P7C4C5-free-with-C7}{Theorem~\ref{thm-decomp-P7C4C5-free-with-C7}}
\newtheorem*{thm-P7C4C5C7-free-contains-Theta-decomp}{Theorem~\ref{thm-P7C4C5C7-free-contains-Theta-decomp}}
\newtheorem*{thm-P7C4C5C7Theta33-free-decomp}{Theorem~\ref{thm-P7C4C5C7Theta33-free-decomp}}

\usepackage{bbm}

\usepackage{microtype}
\numberwithin{figure}{section}

\begin{document}

\title{The class of $(P_7,C_4,C_5)$-free graphs: decomposition, algorithms, and $\chi$-boundedness}

\author{Kathie Cameron \thanks{Department of Mathematics, Wilfrid Laurier University,
Waterloo, ON, Canada, N2L 3C5. Email: \texttt{kcameron@wlu.ca}. Research supported by the Natural Sciences and
Engineering Research Council of Canada (NSERC) grant RGPIN-2016-06517.}  \and Shenwei Huang \thanks{Department of Mathematics, Wilfrid Laurier University,
Waterloo, ON, Canada, N2L 3C5. Email: \texttt{dynamichuang@gmail.com}. Research supported by the Natural Sciences and
Engineering Research Council of Canada (NSERC) grant RGPIN-2016-06517.} \and Irena Penev\thanks{School of Computing, University of Leeds, Leeds LS2 9JT, UK. Email: \texttt{i.penev@leeds.ac.uk}. Partially supported by EPSRC grant EP/N0196660/1.} \and Vaidy Sivaraman \thanks{Department of Mathematics, University of Central Florida, Orlando, FL 32816, USA. Email: \texttt{vaidysivaraman@gmail.com}. Partially supported by the European Research Council under the European Union's Seventh Framework Programme (FP7/2007-2013) / ERC grant agreement 339109.}  }

\maketitle

\begin{abstract}
As usual, $P_n$ ($n \geq 1$) denotes the path on $n$ vertices, and $C_n$ ($n \geq 3$) denotes the cycle on $n$ vertices. For a family $\mathcal{H}$ of graphs, we say that a graph $G$ is {\em $\mathcal{H}$-free} if no induced subgraph of $G$ is isomorphic to any graph in $\mathcal{H}$. We present a decomposition theorem for the class of $(P_7,C_4,C_5)$-free graphs; in fact, we give a complete structural characterization of $(P_7,C_4,C_5)$-free graphs that do not admit a clique-cutset. We use this decomposition theorem to show that the class of $(P_7,C_4,C_5)$-free graphs is $\chi$-bounded by a linear function (more precisely, every $(P_7,C_4,C_5)$-free graph $G$ satisfies $\chi(G) \leq \frac{3}{2} \omega(G)$). We also use the decomposition theorem to construct an $O(n^3)$ algorithm for the minimum coloring problem, an $O(n^2m)$ algorithm for the maximum weight stable set problem, and an $O(n^3)$ algorithm for the maximum weight clique problem for this class, where $n$ denotes the number of vertices and $m$ the number of edges of the input graph.
\end{abstract}

\section{Introduction} \label{sec:intro}

In this paper, all graphs are finite and simple. Furthermore, unless stated otherwise, all graphs are nonnull. 

Given graphs $G$ and $H$, we say that $G$ is {\em $H$-free} if no induced subgraph of $G$ is isomoprhic to $H$. Given a graph $G$ and a family $\mathcal{H}$ of graphs, we say that $G$ is {\em $\mathcal{H}$-free} if $G$ is $H$-free for all $H \in \mathcal{H}$.

As usual, given a positive integer $n$, we denote the path on $n$ vertices by $P_n$, and we denote the complete graph on $n$ vertices by $K_n$. For an integer $n \geq 3$, $C_n$ is the cycle on $n$ vertices.

A {\em clique} in a graph $G$ is a (possibly empty) set of pairwise adjacent vertices of $G$, and a {\em stable set} in $G$ is a (possibly empty) set of pairwise nonadjacent vertices of $G$. The {\em clique number} of $G$, denoted by $\omega(G)$, is the maximum size of a clique in $G$, and the {\em stability number} of $G$, denoted by $\alpha(G)$, is the maximum size of a stable set in $G$.
A {\em $q$-coloring} of $G$ is a function $c:V(G)\longrightarrow \{ 1, \ldots ,q\}$, such that $c(u)\neq c(v)$ for every edge $uv$ of $G$.
The {\em chromatic number} of a graph $G$, denoted by $\chi (G)$, is the minimum number $q$ for which there exists a $q$-coloring of $G$.

In this paper, we give a decomposition theorem for $(P_7,C_4,C_5)$-free graphs (see Theorem \ref{thm-P7C4C5-free-decomp}). In fact, we give a full structural description of $(P_7,C_4,C_5)$-free graphs that do not admit a clique-cutset. (We remark that this is not quite a full structure theorem for the class of $(P_7,C_4,C_5)$-free graphs. This is because $P_7$ admits a clique-cutset, and so the operation of ``gluing along a clique,'' the operation that ``reverses'' the clique-cutset decomposition, is not class-preserving.) We use this decomposition theorem to construct an $O(n^3)$ algorithm for the minimum coloring problem, an $O(n^2m)$ algorithm for the maximum weight stable set problem, and an $O(n^3)$ algorithm for the maximum weight clique problem for this class, where $n$ denotes the number of vertices and $m$ the number of edges of the input graph. We also use it to prove that every $(P_7,C_4,C_5)$-free graph $G$ satisfies $\chi(G) \leq \lfloor \frac{3}{2} \omega(G) \rfloor$. 

Minimum coloring is NP-hard for $(C_4,C_5)$-free graphs, and even 3-coloring is NP-complete on this class~\cite{KKTW}.  Huang~\cite{Shenwei} proved that 4-coloring $P_7$-free graphs is NP-complete. In~\cite{4K1C4C5FreeColoring}, the authors show that there is a polynomial-time algorithm for coloring $(4K_1,C_4,C_5)$-free graphs, a subclass of the class of $(P_7,C_4,C_5)$-free graphs. We remark that the algorithm from~\cite{4K1C4C5FreeColoring} relies 
on algorithms for coloring graphs of bounded clique-width (\cite{CouMak2000, KobRot2003}).

The maximum weight stable set problem is NP-hard for $(C_4,C_5)$-free graphs; its complexity is unknown for $P_7$-free graphs but it can be solved in polynomial-time for $P_5$-free graphs~\cite{LVV}. 

Any $C_4$-free graph has $O(n^2)$ maximal cliques~\cite{Alekseev-C4free, Farber}. Furthermore, if a graph $G$ has $K$ maximal cliques, they can all be found in $O(Kn^3)$ time~\cite{MakinoUno, TIAS}. Thus, all maximal cliques of a $C_4$-free graph can be found in $O(n^5)$ time, and it follows that a maximum weight clique of a $C_4$-free graph can be found in $O(n^5)$ time. As mentioned above, we show that a maximum weight clique of a $(P_7,C_4,C_5)$-free graph can be found in $O(n^3)$ time. 

A class of graphs is {\em hereditary} if it is closed under isomorphism and induced subgraphs; it is not hard to see that a class $\mathcal{G}$ is hereditary if and only if there exists a family $\mathcal{H}$ such that $\mathcal{G}$ is precisely the class of $\mathcal{H}$-free graphs. A hereditary class $\mathcal{G}$ is {\em $\chi$-bounded} if there exists a function $f:\mathbb{N} \rightarrow \mathbb{N}$ such that every graph $G \in \mathcal{G}$ satisfies $\chi(G) \leq f(\omega(G))$. $\chi$-Bounded classes were introduced by Gy\'arf\'as~\cite{Gyarfas} in the 1980s as a generalization of perfection (a graph $G$ is {\em perfect} if all its induced subgraphs $H$ satisfy $\chi(H) = \omega(H)$; clearly, the class of perfect graphs is the maximal hereditary class $\chi$-bounded by the identity function). As mentioned above, we proved that all $(P_7,C_4,C_5)$-free graphs $G$ satisfy $\chi(G) \leq \lfloor \frac{3}{2} \omega(G) \rfloor$; thus, the class of $(P_7,C_4,C_5)$-free graphs is $\chi$-bounded by the function $f(n) = \lfloor \frac{3}{2}n \rfloor$. Gy\'arf\'as~\cite{Gyarfas} showed that for all positive integers $n$, the class of $P_n$-free graphs is $\chi$-bounded. It is well known that $P_4$-free graphs are perfect. On the other hand, for $n \geq 5$, the best $\chi$-bounding function known for the class of $P_n$-free graphs is exponential: it was shown in~\cite{PnFreeChiBound} that every $P_n$-free graph $G$ satisfies $\chi(G) \leq (n-2)^{\omega(G)-1}$. On the other hand, since there exist graphs of arbitrarily large girth and chromatic number~\cite{ErdosGirthChi}, the class of $(C_4,C_5)$-free graphs is not $\chi$-bounded. Finally, Gaspers and Huang~\cite{GH} showed that every $(P_6, C_4)$-free graph $G$ satisfies $\chi(G) \leq \lfloor \frac{3}{2} \omega(G) \rfloor$, and  Chudnovsky and Sivaraman~\cite{CS} proved that for another incomparable class, $(P_5,C_5)$-free graphs, $\chi(G) \leq 2^{\omega(G)-1}$.



\section{Terminology and notation} \label{sec:terminology}

A graph is {\em bipartite} if its vertex set can be partitioned into two stable sets. A graph is {\em cobipartite} if its vertex set can be partitioned into two cliques. Thus, a graph is cobipartite if it is the complement of a bipartite graph.

Let $G$ be a graph. Given a vertex $x \in V(G)$ and a set $Y \subseteq V(G) \setminus \{x\}$, we say that $x$ is {\em complete} (resp.\ {\em anticomplete}) to $Y$ in $G$ if $x$ is adjacent (resp.\ nonadjacent) to every vertex in $Y$; $x$ is {\em mixed} on $Y$ if $x$ is neither complete nor anticomplete to $Y$, that is, if $x$ has both a neighbor and a nonneighbor in $Y$. Given disjoint sets $X,Y \subseteq V(G)$, we say that $X$ is {\em complete} (resp.\ {\em anticomplete}) to $Y$ if every vertex in $X$ is complete (resp.\ anticomplete) to $Y$.

A {\em homogeneous set} in a graph $G$ is a nonempty set $X \subseteq V(G)$ such that no vertex in $V(G) \setminus X$ is mixed on $X$. A homogeneous set $X$ of $G$ is {\em proper} if $2 \leq |X| \leq |V(G)|-1$.

For a graph $G$ and a vertex $v \in V(G)$, the set of all neighbors of $v$ in $G$ is denoted by $N_G(v)$, and we set $N_G[v] = \{v\} \cup N_G(v)$. Given a set $S \subseteq V(G)$, we denote by $N_G(S)$ the set of all vertices in $V(G) \setminus S$ that have a neighbor in $S$, and we set $N_G[S] = S \cup N_G(S)$. If $H$ is an induced subgraph of $G$, we sometimes write $N_G(H)$ and $N_G[H]$ instead of $N_G(V(H))$ and $N_G[V(H)]$, respectively. We say that a vertex $x \in V(G)$ is {\em dominating} in $G$ provided that $N_G[x] = V(G)$, and we say that a set $S \subseteq V(G)$ is {\em dominating} in $G$ provided that $N_G[S] = V(G)$. An induced subgraph $H$ of a graph $G$ is said to be {\em dominating} in $G$ provided that $V(H)$ is dominating in $G$. Given distinct vertices $u,v \in V(G)$, we say that $u$ {\em dominates} $v$ in $G$, or that $v$ is {\em dominated} by $u$ in $G$, provided that $N_G[v] \subseteq N_G[u]$.

Given a graph $G$ and a nonempty nonempty set $S \subseteq V(G)$, we denote by $G[S]$ the subgraph of $G$ induced by $S$; for vertices $v_1,\dots,v_t \in V(G)$, we sometimes write $G[v_1,\dots,v_t]$ instead of $G[\{v_1,\dots,v_t\}]$.

The complement of a graph $G$ is denoted by $\overline{G}$. A graph is {\em anticonnected} if its complement is connected. An {\em anticomponent} of a graph $G$ is a maximal anticonnected induced subgraph of $H$ of $G$. (Equivalently, $H$ is an {\em anticomponent} of $G$ provided that $\overline{H}$ is a component of $\overline{G}$.) Clearly, the vertex sets of the anticomponents of $G$ are pairwise disjoint and complete to each other. A {\em trivial} anticomponent of a graph $G$ is one that has just one vertex; a {\em nontrivial} anticomponent of $G$ is one that has at least two vertices. Note that every trivial anticomponent of a graph $G$ is dominating in $G$.

A {\em cutset} in a graph $G$ is a (possibly empty) set $C \subsetneqq V(G)$ such that $G \setminus C$ is disconnected. A {\em clique-cutset} of $G$ is a (possibly empty) clique $C$ such that $C$ is a cutset of $G$. A {\em cut-partition} of $G$ is a partition $(A,B,C)$ of $V(G)$ such that $A$ and $B$ are nonempty and anticomplete to each other ($C$ may possibly be empty). A {\em clique-cut-partition} of $G$ is a cut-partition $(A,B,C)$ of $G$ such that $C$ is a (possibly empty) clique. Note that if $(A,B,C)$ is a cut-partition (resp.\ clique-cut-partition) of $G$, then $C$ is a cutset (resp.\ clique-cutset) of $G$; conversely, every cutset (resp.\ clique-cutset) of $G$ gives rise to at least one cut-partition (resp.\ clique-cut-partition) of $G$.

For an integer $k \geq 4$, a {\em $k$-hole} is an induced cycle of length $k$. A {\em hole} is a $k$-hole for some $k \geq 4$. A hole is {\em odd} (resp.\ {\em even}) if its length is odd (resp.\ even).

A graph is {\em chordal} if it contains no holes. It is well known that if a graph is chordal, then either it is complete or it admits a clique-cutset~\cite{Dirac61}. A vertex $v$ in a graph $G$ is {\em simplicial} if $N_G(v)$ is a (possibly empty) clique. A {\em simplicial elimination ordering} of a graph $G$ is an ordering $v_1,\dots,v_n$ of its vertices such that for all $i \in \{1,\dots,n\}$, $v_i$ is simplicial in $G[v_i,\dots,v_n]$. It is well known that a graph is chordal if and only if there is a simplicial elimination ordering for it~\cite{FulkersonGrossSimplicialElimOrd}. To simplify notation, for a graph $G$ and sets $Z_1,\dots,Z_t \subseteq V(G)$, we say that $Z_1,\dots,Z_t$ is a {\em simplicial elimination ordering} of $G$ provided that for any ordering $Z_1 = \{z_1^1,\dots,z_{|Z_1|}^1\},\dots,Z_t = \{z_1^t,\dots,z_{|Z_t|}^t\}$ of the sets $Z_1,\dots,Z_t$, we have that $z_1^1,\dots,z_{|Z_1|}^1,\dots,z_1^t,\dots,z_{|Z_t|}^t$ is a simplicial elimination ordering of $G$ (note that this implies that the sets $Z_1,\dots,Z_t$ are pairwise disjoint, and that their union is $V(G)$).

Given graphs $G$ and $H$, we say that $G$ is {\em obtained by blowing up each vertex of $H$ to a nonempty clique} provided that there exists a partition $\{X_v\}_{v \in V(H)}$ of $V(G)$ into nonempty cliques such that for all distinct $u,v \in V(H)$, if $uv \in E(H)$, then $X_u$ is complete to $X_v$ in $G$, and if $uv \notin E(H)$, then $X_u$ is anticomplete to $X_v$ in $G$.

\section{A decomposition theorem for $(P_7,C_4,C_5)$-free graphs} \label{sec:decomp}

The main result of this section is a theorem that characterizes $(P_7,C_4,C_5)$-free graphs that do not admit a clique-cutset. We state this theorem below, but we note that we have not yet defined all the terms that appear it.

\begin{theorem} \label{thm-P7C4C5-free-decomp} Let $G$ be a graph. Then the following are equivalent:
\begin{itemize}
\item $G$ is a $(P_7,C_4,C_5)$-free graph that does not admit a clique-cutset;
\item either $G$ is a complete graph, or $G$ contains exactly one nontrivial anticomponent, and this anticomponent is either a 7-bracelet, a thickened emerald, a lantern, a 6-wreath, or a 6-crown.
\end{itemize}
\end{theorem}

To make sense of Theorem~\ref{thm-P7C4C5-free-decomp}, we must define ``7-bracelets,'' ``thickened emeralds,'' ``lanterns,'' ``6-wreaths,'' and ``6-crowns.'' We first define these terms, and after that, we turn to the proof of Theorem~\ref{thm-P7C4C5-free-decomp}.

Let $B$ be a graph, let $\{A_i\}_{i \in \mathbb{Z}_7}$ be a partition of $V(B)$, and let $i^* \in \mathbb{Z}_7$. We say that $B$ is a {\em 7-bracelet} with {\em good pair} $(\{A_i\}_{i \in \mathbb{Z}_7},i^*)$ provided that the following hold (see Figure~\ref{fig:7bracelet}):
\begin{enumerate}[(I)]
\item for all $i \in \mathbb{Z}_7$, $A_i$ is a nonempty clique, complete to $A_{i-1} \cup A_{i+1}$ and anticomplete to $A_{i-3} \cup A_{i+3}$;
\item for all $i \in \mathbb{Z}_7$, there exists a partition of $A_i$ into three (possibly empty) sets, call them $A_i^*,A_i^+,A_i^-$, such that all the following hold:
\begin{enumerate}[(a)]
\item $A_i^*$ is anticomplete to $A_{i-2} \cup A_{i+2}$,
\item $A_i^+$ is anticomplete to $A_{i-2}$, and every vertex in $A_i^+$ has a neighbor in $A_{i+2}$,
\item $A_i^-$ is anticomplete to $A_{i+2}$, and every vertex in $A_i^-$ has a neighbor in $A_{i-2}$,\footnote{Note that (II.a), (II.b), and (II.c) together imply that the partition $(A_i^*,A_i^+,A_i^-)$ of $A_i$ is unique, and furthermore, that every vertex in $A_i$ is anticomplete to at least one of $A_{i-2},A_{i+2}$.}
\item if $A_i^+ \neq \emptyset$, then $A_i^+$ can be ordered as $A_i^+ = \{a_1^{i^+},\dots,a_{|A_i^+|}^{i^+}\}$ so that $N_B[a_{|A_i^+|}^{i^+}] \subseteq \dots \subseteq N_B[a_1^{i^+}]$,
\item if $A_i^- \neq \emptyset$, then $A_i^-$ can be ordered as $A_i^- = \{a_1^{i^-},\dots,a_{|A_i^-|}^{i^-}\}$ so that $N_B[a_{|A_i^-|}^{i^-}] \subseteq \dots \subseteq N_B[a_1^{i^-}]$,
\item either $A_i^* \neq \emptyset$, or $A_i^+$ is not complete to $A_{i+2}$, or $A_i^-$ is not complete to $A_{i-2}$,\footnote{Equivalently: some vertex in $A_i$ has a nonneighbor both in $A_{i-2}$ and in $A_{i+2}$.}
\end{enumerate}
\item $A_{i^*-3} = A_{i^*-3}^*$ and $A_{i^*+3} = A_{i^*+3}^*$;\footnote{Equivalently: $A_{i^*-3}^+ = A_{i^*-3}^- = A_{i^*+3}^+ = A_{i^*+3}^- = \emptyset$. In other words, for all $i \in \{i^*-3,i^*+3\}$, $A_i$ is anticomplete to $A_{i-2} \cup A_{i+2}$.}
\item $A_{i^*-2} = A_{i^*-2}^* \cup A_{i^*-2}^+$ and $A_{i^*+2} = A_{i^*+2}^* \cup A_{i^*+2}^-$;\footnote{Equivalently: $A_{i^*-2}^- = A_{i^*+2}^+ = \emptyset$. In other words, $A_{i^*-2}$ is anticomplete to $A_{i^*+3}$, and $A_{i^*+2}$ is anticomplete to $A_{i^*-3}$.}
\item $A_{i^*-1} = A_{i^*-1}^* \cup A_{i^*-1}^+$ and $A_{i^*+1} = A_{i^*+1}^* \cup A_{i^*+1}^-$.\footnote{Equivalently: $A_{i^*-1}^- = A_{i^*+1}^+ = \emptyset$. In other words, $A_{i^*-1}$ is anticomplete to $A_{i^*-3}$, and $A_{i^*+1}$ is anticomplete to $A_{i^*+3}$.}
\end{enumerate}

A graph $B$ is said to be a {\em 7-bracelet} provided that there exists a partition $\{A_i\}_{i \in \mathbb{Z}_7}$ of $V(B)$ and an index $i^* \in \mathbb{Z}_7$ such that $B$ is a 7-bracelet with good pair $(\{A_i\}_{i \in \mathbb{Z}_7},i^*)$. Furthermore, given a 7-bracelet $B$, we say that $\{A_i\}_{i \in \mathbb{Z}_7}$ is a {\em good partition} of $B$ provided that there exists some index $i^* \in \mathbb{Z}_7$ such that $B$ is a 7-bracelet with good pair $(\{A_i\}_{i \in \mathbb{Z}_7},i^*)$.\footnote{Note that the index $i^*$ need not be unique. In fact, $i^*$ is unique if any only if there exists some $i \in \mathbb{Z}_7$ such that $A_i^+,A_i^-$ are both nonempty (in this case, $i^* = i$). Note that if $A_i^+ = A_i^- = \emptyset$ for all $i \in \mathbb{Z}_7$, then $i^* \in \mathbb{Z}_7$ can be chosen arbitrarily.} Note that $C_7$ is a 7-bracelet.

\begin{figure}
\begin{center}
\includegraphics[scale=0.6]{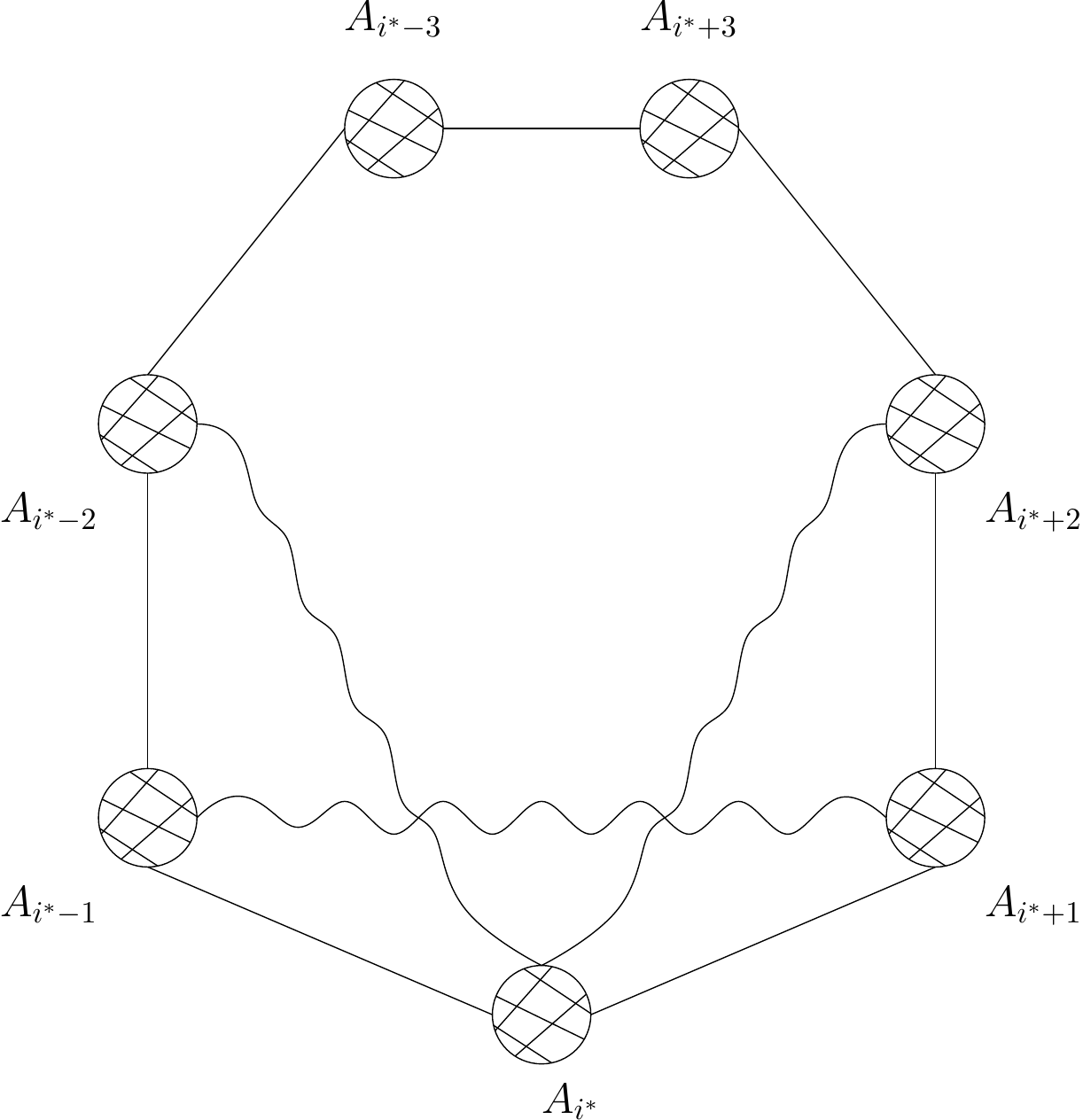}
\end{center}
\caption{7-bracelet $B$ with good pair $(\{A_i\}_{i \in \mathbb{Z}_7},i^*)$. A shaded disk represents a nonempty clique. A straight line between two cliques indicates that the two cliques are complete to each other. A wavy line between two cliques indicates that there may be edges between the two cliques (furthermore, such edges must obey the axioms from the definition of a 7-bracelet). The absence of a line (straight or wavy) between two cliques indicates that the two cliques are anticomplete to each other.} \label{fig:7bracelet}
\end{figure}

\begin{figure}
\begin{center}
\includegraphics[scale=0.6]{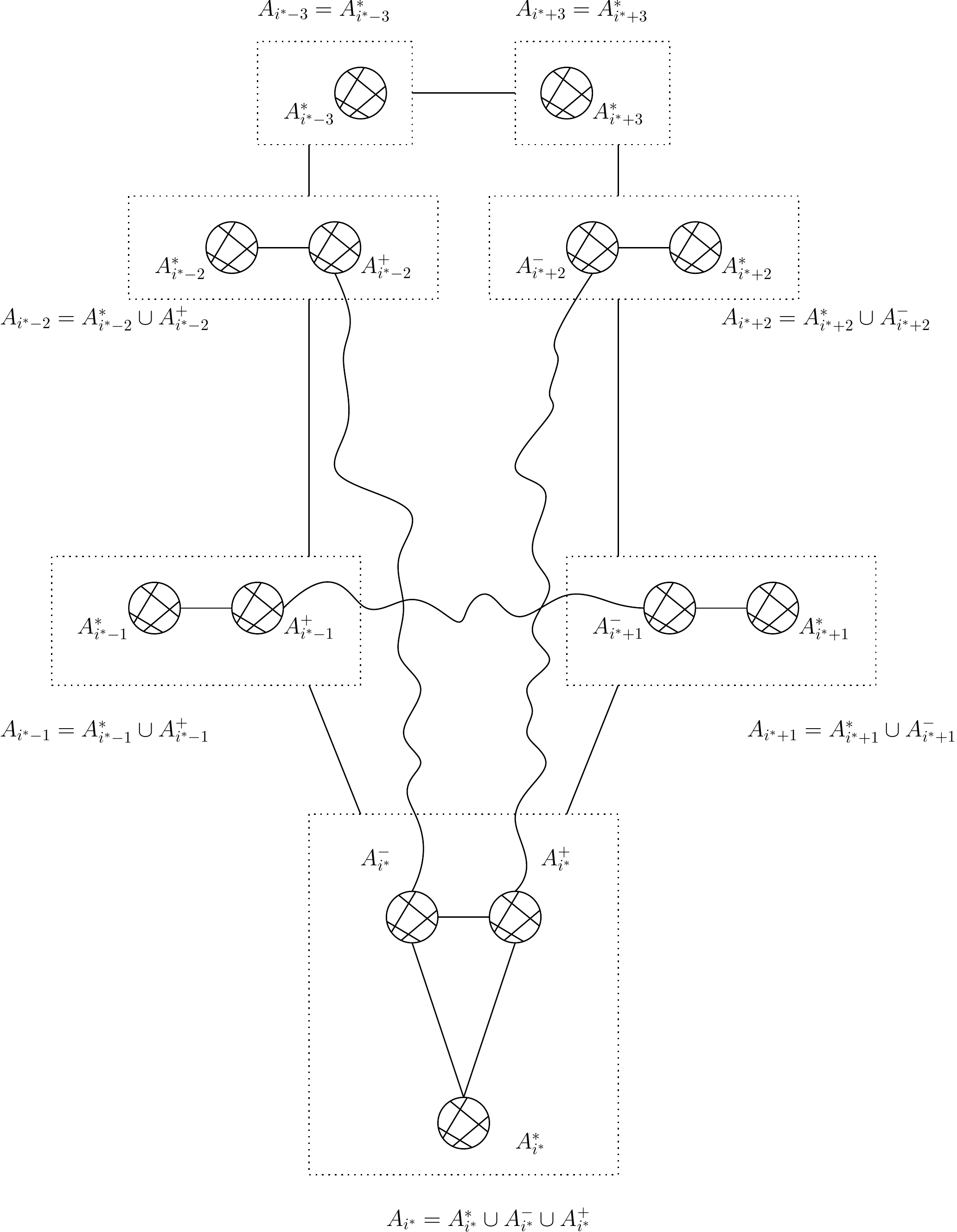}
\end{center}
\caption{7-Bracelet $B$ with good pair $(\{A_i\}_{i \in \mathbb{Z}_7},i^*)$.} \label{fig:7bracelet-details}
\end{figure}

We give a more detailed representation of a 7-bracelet in Figure~\ref{fig:7bracelet-details}. Let us explain this figure. Dotted rectangles in Figure~\ref{fig:7bracelet-details} represent nonempty cliques, and shaded circles represent (possibly empty) cliques; each dotted rectangle represents the union of the cliques represented by the shaded circles inside that rectangle. For each $i \in \mathbb{Z}_7$, $A_i = A_i^* \cup A_i^+ \cup A_i^-$; those $A_i^+$'s and $A_i^-$'s that are not represented in the picture are empty. A straight line between two rectangles (resp. two shaded circles) indicates that the cliques represented by the two rectangles (resp. shaded circles) are complete to each other. If there is no straight line between two rectangles, but there is a wavy line between two shaded circles in those rectangles, that indicates that all edges between the cliques represented by the two rectangles are in fact between the cliques represented by those two shaded circles (furthermore, such edges must obey the axioms from the definition of a 7-bracelet). If there is no straight line between two rectangles, and there is no wavy line between any two shaded circles in those rectangles, then the cliques represented by the two rectangles are anticomplete to each other.

\begin{figure}
\begin{center}
\includegraphics[scale=0.6]{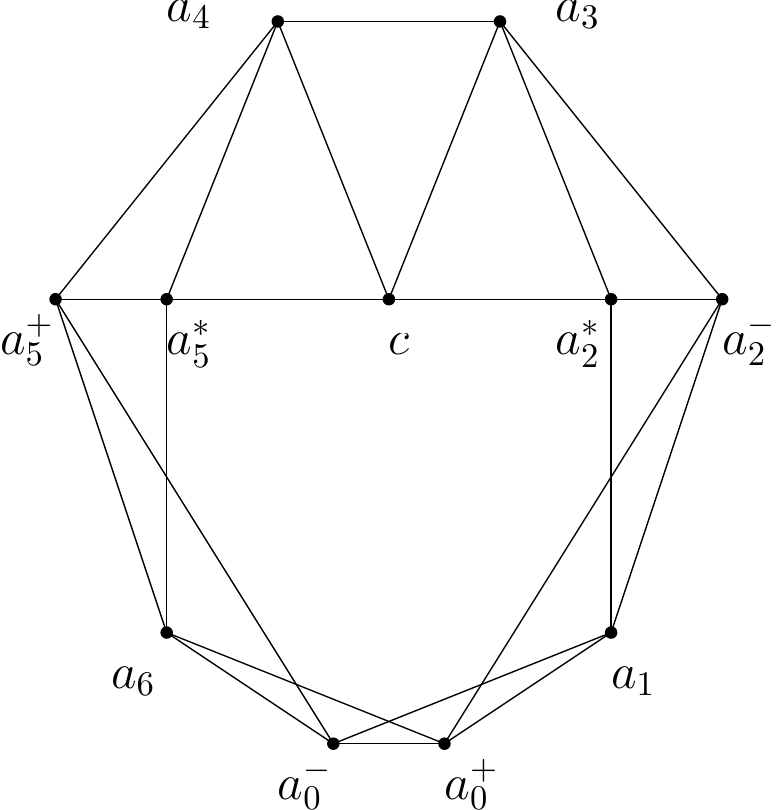}
\end{center}
\caption{The emerald.} \label{fig:emerald}
\end{figure}

The {\em emerald} is the graph represented in Figure~\ref{fig:emerald}.

Let $B$ be a graph, let $C,A_0,\dots,A_6$ (with indices in $\mathbb{Z}_7$) be a partition of $V(B)$ into nonempty cliques, and let $i^* \in \mathbb{Z}_7$. We say that $B$ is a {\em thickened emerald} with {\em good triple} $(\{A_i\}_{i \in \mathbb{Z}_7},C,i^*)$ provided that all the following hold (see Figure~\ref{fig:thickened-emerald}):
\begin{itemize}
\item for all $i \in \mathbb{Z}_7$, $A_i$ is complete to $A_{i-1} \cup A_{i+1}$ and anticomplete to $A_{i-3} \cup A_{i+3}$;
\item for all $i \in \{i^*-3,i^*-1,i^*+1,i^*+3\}$, $A_i$ is anticomplete to $A_{i-2} \cup A_{i+2}$;
\item there exist nonempty, pairwise disjoint cliques $A_{i^*}^-,A_{i^*}^+,A_{i^*+2}^*,A_{i^*+2}^-,A_{i^*-2}^*,A_{i^*-2}^+$ such that all the following hold:
\begin{itemize}
\item $A_{i^*} = A_{i^*}^- \cup A_{i^*}^+$,
\item $A_{i^*+2} = A_{i^*+2}^* \cup A_{i^*+2}^-$,
\item $A_{i^*-2} = A_{i^*-2}^* \cup A_{i^*-2}^+$,
\item $A_{i^*}^-$ is complete to $A_{i^*-2}^+$ and anticomplete to $A_{i^*-2}^* \cup A_{i^*+2}$,
\item $A_{i^*}^+$ is complete to $A_{i^*+2}^-$ and anticomplete to $A_{i^*+2}^* \cup A_{i^*-2}$,
\item $C$ is complete to $A_{i^*+2}^* \cup A_{i^*+3} \cup A_{i^*-3} \cup A_{i^*-2}^*$ and anticomplete to $A_{i^*-2}^+ \cup A_{i^*-1} \cup A_{i^*} \cup A_{i^*+1} \cup A_{i^*+2}^-$.
\end{itemize}
\end{itemize}
Furthermore, for notational purposes, we set $A_{i^*}^* = A_{i^*+2}^+ = A_{i^*-2}^- = \emptyset$, and for all $i \in \{i^*-3,i^*-1,i^*+1,i^*+3\}$, we set $A_i^* = A_i$ and $A_i^- = A_i^+ = \emptyset$.

A {\em thickened emerald} is any graph $B$ for which there exists a partition $C,A_0,A_1,\dots,A_6$ (with indices in $\mathbb{Z}_7$) of $V(B)$ and an index $i^* \in \mathbb{Z}_7$ such that $B$ is a thickened emerald with good triple $(\{A_i\}_{i \in \mathbb{Z}_7},C,i^*)$. Note that a graph is a thickened emerald if and only if it can be obtained from the emerald by blowing up each vertex to a nonempty clique.

\begin{figure}
\begin{center}
\includegraphics[scale=0.6]{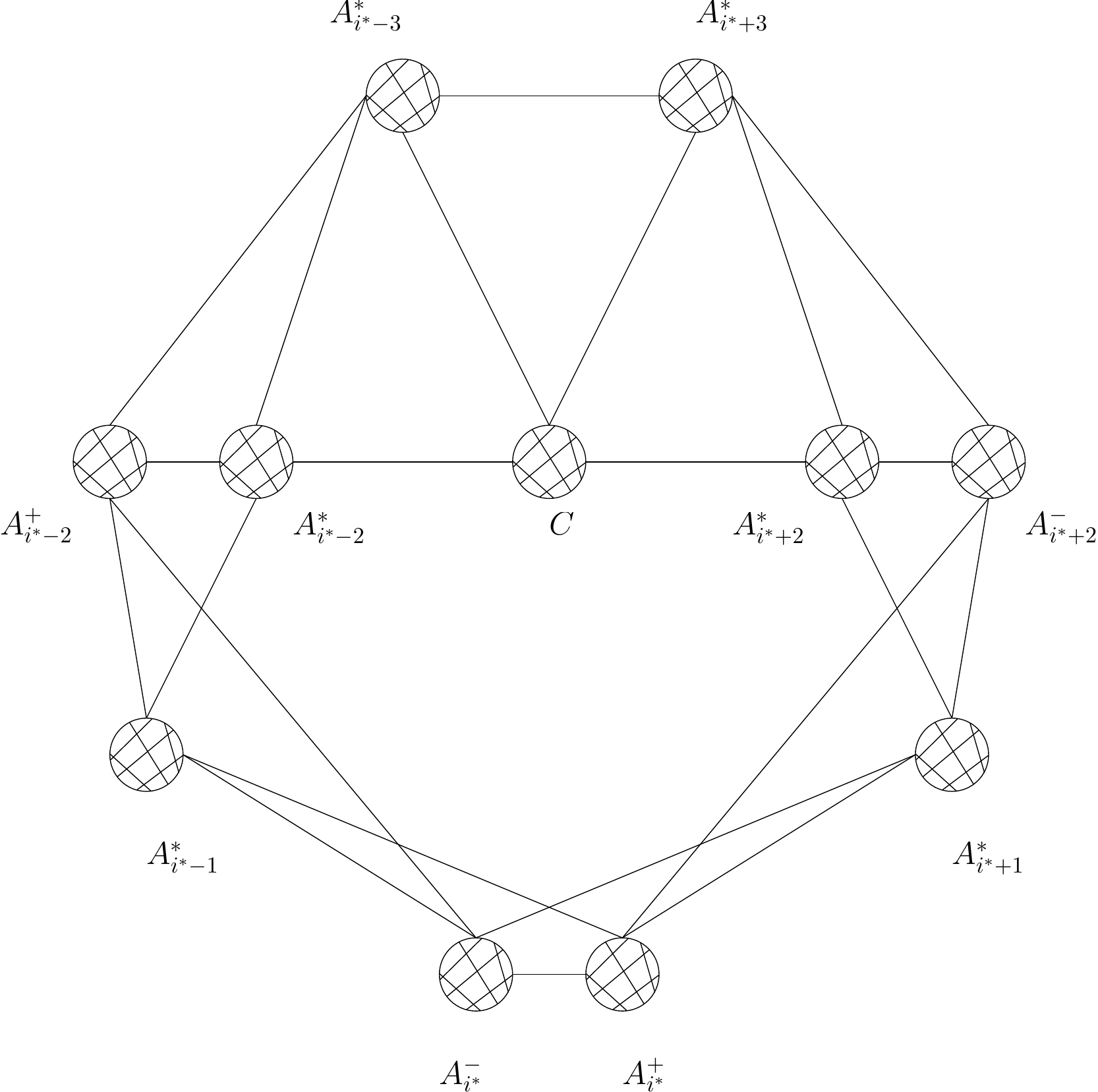}
\end{center}
\caption{Thickened emerald $B$ with good triple $(\{A_i\}_{i \in \mathbb{Z}_7},C,i^*)$. For all $i \in \mathbb{Z}_7$, $A_i = A_i^* \cup A_i^- \cup A_i^+$ (those $A_i^-$'s and $A_i^+$'s that are not represented in the figure are empty). A shaded disk represents a nonempty clique. A straight line between two cliques indicates that the two cliques are complete to each other. The absence of a line between two cliques indicates that the two cliques are anticomplete to each other.} \label{fig:thickened-emerald}
\end{figure}

\begin{lemma} \label{lemma-7-bracelet-in-thickened-emerald} Let $B$ be a thickened emerald with good triple $(\{A_i\}_{i \in \mathbb{Z}_7},C,i^*)$. Then $B \setminus C$ is a 7-bracelet with good pair $(\{A_i\}_{i \in \mathbb{Z}_7},i^*)$.
\end{lemma}
\begin{proof}
This is immediate from the appropriate definitions.
\end{proof}

For an integer $r \geq 3$, an {\em $r$-lantern} is a graph $R$ whose vertex set can be partitioned into nonempty cliques $A,B_1,\dots,B_r,C_1,\dots,C_r,D$ such that all the following hold:
\begin{itemize}
\item $A$ is anticomplete to $D$;
\item $A$ is complete to $\bigcup_{i=1}^r B_i$ and anticomplete to $\bigcup_{i=1}^r C_i$;
\item $D$ is complete to $\bigcup_{i=1}^r C_i$ and anticomplete to $\bigcup_{i=1}^r B_i$;
\item $B_1$ and $C_1$ can be ordered as $B_1 = \{b_1^1,\dots,b_{|B_1|}^1\}$ and $C_1 = \{c_1^1,\dots,c_{|C_1|}^1\}$ so that $N_R[b_{|B_1|}^1] \cap C_1 \subseteq \dots \subseteq N_R[b_1^1] \cap C_1 = C_1$ and $N_R[c_{|C_1|}^1] \cap B_1 \subseteq \dots \subseteq N_R[c_1^1] \cap B_1 = B_1$;\footnote{Note that this implies that $b_1^1$ is complete to $C_1$, and that $c_1^1$ is complete to $B_1$ (in particular, $b_1^1$ and $c_1^1$ are adjacent). Consequently, every vertex in $B_1$ has a neighbor (namely, $c_1^1$) in $C_1$, and every vertex in $C_1$ has a neighbor (namely, $b_1^1$) in $B_1$.}
\item for all $i \in \{2,\dots,r\}$, $B_i$ is complete to $C_i$;
\item for all distinct $i,j \in \{1,\dots,r\}$, $B_i \cup C_i$ is anticomplete to $B_j \cup C_j$.
\end{itemize}
Under these circumstances, we also say that $(A,B_1,\dots,B_r,C_1,\dots,C_r,D)$ is a {\em good partition} for the $r$-lantern $R$ (see Figure~\ref{fig:Lantern}).

A graph $R$ is a {\em lantern} if there exists some integer $r \geq 3$ such that $R$ is an $r$-lantern.

\begin{figure}
\begin{center}
\includegraphics[scale=0.6]{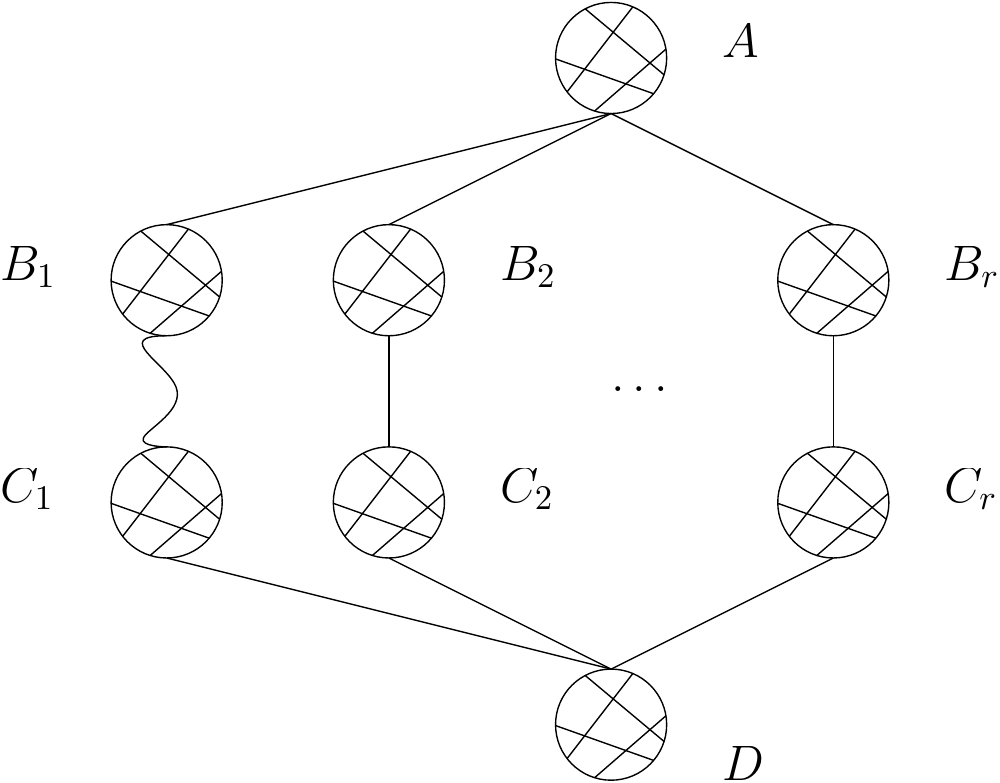}
\end{center}
\caption{$r$-Lantern, $r \geq 3$, with good partition $(A,B_1,\dots,B_r,C_1,\dots,C_r,D)$. A shaded disk represents a nonempty clique. A straight line between two cliques indicates that the two cliques are complete to each other. A wavy line between two cliques indicates that there may be edges between the two cliques (furthermore, such edges must obey the axioms from the definition of an $r$-lantern). The absence of a line between two cliques indicates that the two cliques are anticomplete to each other.} \label{fig:Lantern}
\end{figure}

\begin{figure}
\begin{center}
\includegraphics[scale=0.6]{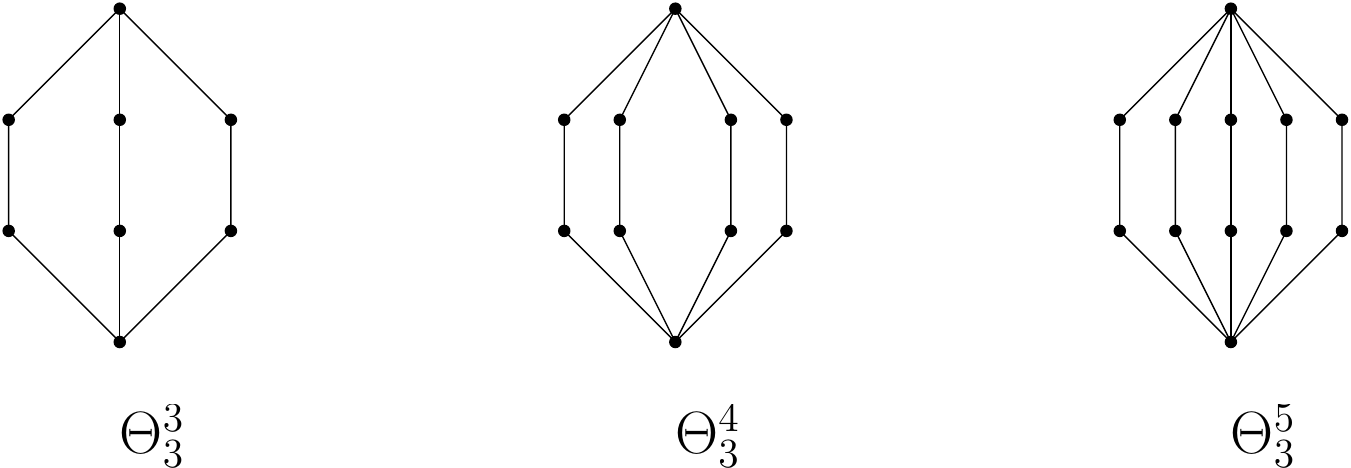}
\end{center}
\caption{Graph $\Theta_3^r$ for $r = 3,4,5$.} \label{fig:Theta3r}
\end{figure}

For an integer $r \geq 3$, $\Theta_3^r$ is the graph that consists of $r$ internally disjoint induced three-edge paths that meet at their endpoints (see Figure~\ref{fig:Theta3r}).

Note that for all integers $r \geq 3$, every $r$-lantern contains an induced $\Theta_3^r$ (and consequently, an induced $\Theta_3^3$ as well), and $\Theta_3^r$ is an $r$-lantern.

A {\em 6-wreath} is a graph $R$ whose vertex set can be partitioned into six nonempty sets, say $X_0,\dots,X_5$ (with indices understood to be in $\mathbb{Z}_6$), that can be ordered as $X_0 = \{u_1^0,\dots,u_{|X_0|}^0\},\dots,X_5 = \{u_1^5,\dots,u_{|X_5|}^5\}$ so that both the following hold:
\begin{itemize}
\item for all $i \in \mathbb{Z}_6$, $X_i \subseteq N_R[u_{|X_i|}^i] \subseteq \dots \subseteq N_R[u_1^i] = X_{i-1} \cup X_i \cup X_{i+1}$;\footnote{Note that this implies that $X_0,\dots,X_5$ are nonempty cliques, and that for all $i \in \mathbb{Z}_6$, $X_i$ is anticomplete to $X_{i+2} \cup X_{i+3} \cup X_{i+4}$.}
\item $X_0$ is complete to $X_1$, $X_2$ is complete to $X_3$, and $X_4$ is complete to $X_5$.
\end{itemize}
Under these circumstances, we say that $(X_0,\dots,X_5)$ is a {\em good partition} of the 6-wreath $R$ (see Figure~\ref{fig:6Wreath}).

\begin{figure}
\begin{center}
\includegraphics[scale=0.6]{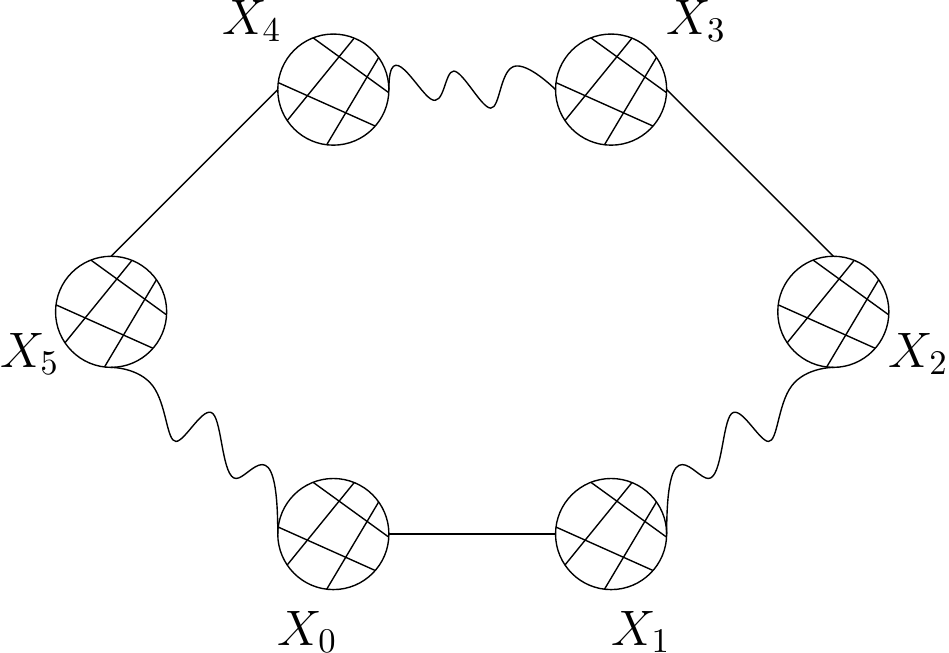}
\end{center}
\caption{6-Wreath with good partition $(X_0,\dots,X_5)$. A shaded disk represents a nonempty clique. A straight line between two cliques indicates that the two cliques are complete to each other. A wavy line between two cliques indicates that there are edges between the two cliques (furthermore, such edges must obey the axioms from the definition of a 6-wreath). The absence of a line (straight or wavy) between two cliques indicates that the two cliques are anticomplete to each other.} \label{fig:6Wreath}
\end{figure}

\begin{figure}
\begin{center}
\includegraphics[scale=0.6]{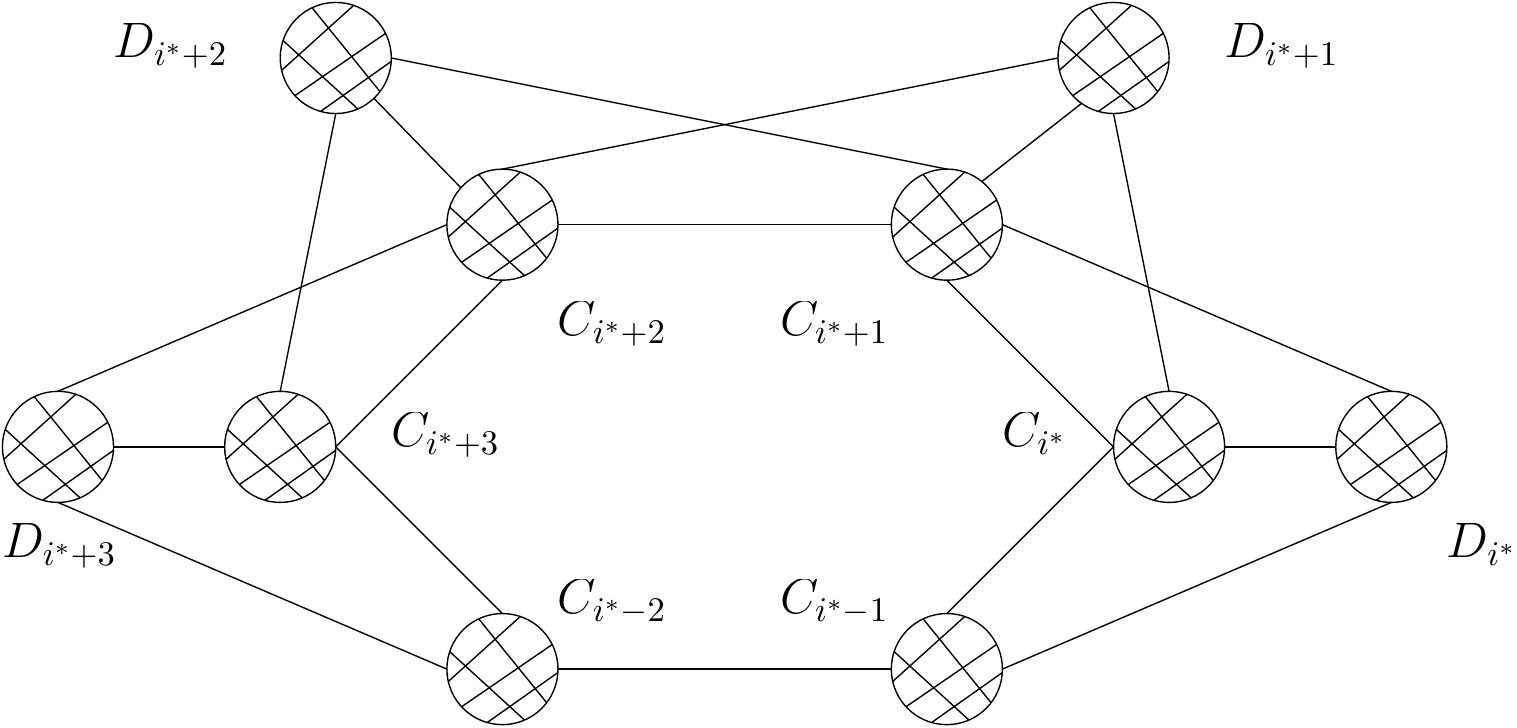}
\end{center}
\caption{6-Crown with good triple $(\{C_i\}_{i \in \mathbb{Z}_6},\{D_i\}_{i \in \mathbb{Z}_6},i^*)$. A shaded disk represents a (possibly empty) clique. Clique $D_{i^*}$ may be empty or nonempty, and the other cliques represented by shaded disks are all nonempty. Cliques $D_{i^*-1},D_{i^*-2}$ are empty (and not represented in the figure). A straight line between two cliques indicates that the two cliques are complete to each other. The absence of a line between two cliques indicates that the two cliques are anticomplete to each other.} \label{fig:6Crown}
\end{figure}

\begin{figure}
\begin{center}
\includegraphics[scale=0.6]{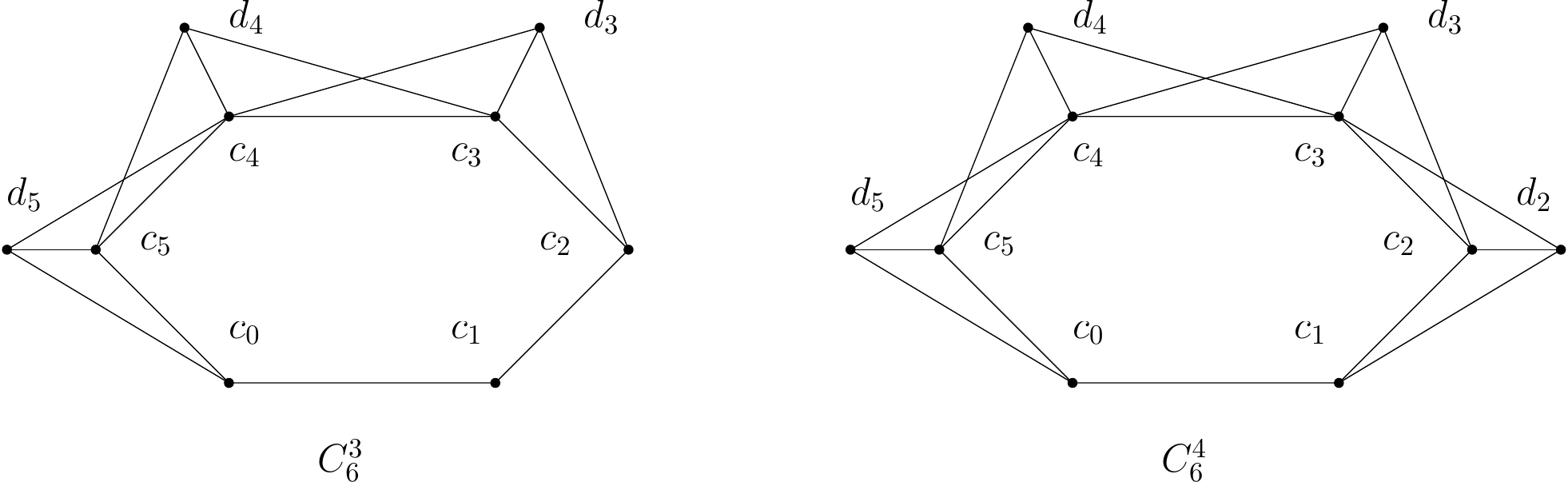}
\end{center}
\caption{Graphs $C_6^3$ and $C_6^4$.} \label{fig:C63C64}
\end{figure}

Let $R$ be a graph, let $C_1,\dots,C_6,D_1,\dots,D_6$ (with indices in $\mathbb{Z}_6$) be a partition of $V(R)$ into (possibly empty) cliques, and let $i^* \in \mathbb{Z}_6$. We say that $R$ is a {\em 6-crown} with {\em good triple} $(\{C_i\}_{i \in \mathbb{Z}_6},\{D_i\}_{i \in \mathbb{Z}_6},i^*)$ provided that all the following hold (see Figure~\ref{fig:6Crown}):
\begin{itemize}
\item $C_0,\dots,C_5$ are nonempty;
\item $D_{i^*-2},D_{i^*-1}$ are empty, and $D_{i^*+1},D_{i^*+2},D_{i^*+3}$ are nonempty ($D_{i^*}$ may be empty or nonempty);
\item for all $i \in \mathbb{Z}_6$, $C_i$ is complete to $C_{i-1} \cup C_{i+1}$ and anticomplete to $C_{i+2} \cup C_{i+3} \cup C_{i+4}$;
\item for all $i \in \mathbb{Z}_6$, $D_i$ is complete to $C_{i-1} \cup C_i \cup C_{i+1}$ and anticomplete to $C_{i+2} \cup C_{i+3} \cup C_{i+4}$;
\item $D_0,\dots,D_5$ are pairwise anticomplete to each other.
\end{itemize}

A {\em 6-crown} is any graph $R$ for which there exists a partition $C_1,\dots,C_6,D_1,\dots,D_6$ of $V(R)$ and an index $i^* \in \mathbb{Z}_6$ such that $R$ is a 6-crown with good triple $(\{C_i\}_{i \in \mathbb{Z}_6},\{D_i\}_{i \in \mathbb{Z}_6},i^*)$.

Let $C_6^3$ and $C_6^4$ be the graphs represented in Figure~\ref{fig:C63C64}. Note that a graph $R$ is a 6-crown if and only if it can be obtained from one of $C_6^3,C_6^4$ by blowing up each vertex to a nonempty clique.

\begin{lemma} Every 7-bracelent, thickened emerald, lantern, 6-wreath, and 6-crown is anticonnected.
\end{lemma}
\begin{proof}
This can be seen by routine checking.
\end{proof}

Theorem~\ref{thm-P7C4C5-free-decomp} (stated at the beginning of this section) is an immediate corollary of Theorems~\ref{thm-decomp-P7C4C5-free-with-C7},~\ref{thm-P7C4C5C7-free-contains-Theta-decomp}, and~\ref{thm-P7C4C5C7Theta33-free-decomp}, stated below (and proven later in this section).

\begin{theorem} \label{thm-decomp-P7C4C5-free-with-C7} Let $G$ be a graph. Then the following two statements are equivalent:
\begin{enumerate}[(i)]
\item $G$ is a $(P_7,C_4,C_5)$-free graph that contains a 7-hole and does not admit a clique-cutset;
\item $G$ contains exactly one nontrivial anticomponent, and this anticomponent is either a 7-bracelet or a thickened emerald.
\end{enumerate}
\end{theorem}

\begin{theorem} \label{thm-P7C4C5C7-free-contains-Theta-decomp} Let $G$ be a graph. Then the following two statements are equivalent:
\begin{enumerate}[(i)]
\item $G$ is a $(P_7,C_4,C_5,C_7)$-free graph that contains an induced $\Theta_3^3$ and does not admit a clique-cutset;
\item $G$ contains exactly one nontrivial anticomponent, and this anticomponent is a lantern.
\end{enumerate}
\end{theorem}

\begin{theorem} \label{thm-P7C4C5C7Theta33-free-decomp} Let $G$ be a graph. Then the following two statements are equivalent:
\begin{enumerate}[(i)]
\item $G$ is a $(P_7,C_4,C_5,C_7,\Theta_3^3)$-free graph that does not admit a clique-cutset;
\item either $G$ is a complete graph, or $G$ contains exactly one nontrivial anticomponent, and this anticomponent is either a 6-wreath or a 6-crown.
\end{enumerate}
\end{theorem}

Our goal in the remainder of this section is to prove Theorems~\ref{thm-decomp-P7C4C5-free-with-C7},~\ref{thm-P7C4C5C7-free-contains-Theta-decomp}, and~\ref{thm-P7C4C5C7Theta33-free-decomp}. We first prove a simple lemma (Lemma~\ref{lemma-one-anticomp}), which we will use several times in this section. In subsection~\ref{subsec:thm-decomp-P7C4C5-free-with-C7}, we prove Theorem~\ref{thm-decomp-P7C4C5-free-with-C7}, in subsection~\ref{subsec:thm-P7C4C5C7-free-contains-Theta-decomp}, we prove Theorem~\ref{thm-P7C4C5C7-free-contains-Theta-decomp}, and in subsection~\ref{subsec:thm-P7C4C5C7Theta33-free-decomp}, we prove Theorem~\ref{thm-P7C4C5C7Theta33-free-decomp}. As we already pointed out, Theorem~\ref{thm-P7C4C5-free-decomp} follows immediately from Theorems~\ref{thm-decomp-P7C4C5-free-with-C7},~\ref{thm-P7C4C5C7-free-contains-Theta-decomp}, and~\ref{thm-P7C4C5C7Theta33-free-decomp}.

\begin{lemma} \label{lemma-one-anticomp} Let $G$ be a graph that has exactly one nontrivial anticomponent, call it $B$. Then all the following hold:
\begin{itemize}
\item $\alpha(G) = \alpha(B)$;
\item for every graph $H$ such that $H$ does not contain a dominating vertex, $G$ is $H$-free if and only if $B$ is $H$-free;
\item $G$ admits a clique-cutset if and only if $B$ admits a clique-cutset.
\end{itemize}
\end{lemma}
\begin{proof}
Let $U = V(G) \setminus V(B)$; then $U$ is a (possibly empty) clique, complete to $V(B)$ in $G$. If $U = \emptyset$, then $G = B$, and the result is immediate. So assume that $U \neq \emptyset$.

First, since $V(B)$ and $U$ are complete to each other in $G$, it is clear that $\alpha(G) = \max\{\alpha(B),\alpha(G[U])\}$. Since $B$ is a nontrivial anticomponent of $G$, we see that $\alpha(B) \geq 2$, and since $U$ is a nonempty clique in $G$, we have that $\alpha(G[U]) = 1$. Thus, $\alpha(G) = \alpha(B)$.

Next, let $H$ be a graph that does not contain a dominating vertex. Clearly, if $G$ is $H$-free, then so is $B$. For the converse, suppose that $G$ is not $H$-free, and let $S \subseteq V(G)$ be such that $G[S] \cong H$. Clearly, every vertex in $S \cap U$ is a dominating vertex of $G[S]$; since $H$ does not contain a dominating vertex, it follows that $S \cap U = \emptyset$. Thus, $S \subseteq V(B)$, and it follows that $B$ is not $H$-free.

It remains to show that $G$ admits a clique-cutset if and only if $B$ does. If $S$ is a clique-cutset of $B$, then clearly, $S \cup U$ is a clique-cutset of $G$. Assume now that $S$ is a clique-cutset of $G$. Clearly, every vertex in $U \setminus S$ is a dominating vertex of $G \setminus S$; since $G \setminus S$ is disconnected, it follows that $U \subseteq S$. But then $G \setminus S = B \setminus (S \setminus U)$, and we deduce that $S \setminus U$ is a clique-cutset of $B$. This completes the argument.
\end{proof}

\subsection{Proof of Theorem~\ref{thm-decomp-P7C4C5-free-with-C7}} \label{subsec:thm-decomp-P7C4C5-free-with-C7}

\begin{lemma} \label{lemma-7-bracelet-no-parallel-2-skips} Let $B$ be a 7-bracelet with good partition $\{A_i\}_{i \in \mathbb{Z}_7}$. Then all the following hold:
\begin{itemize}
\item for all $i \in \mathbb{Z}_7$, $A_{i-1}^+ \neq \emptyset$ if and only if $A_{i+1}^- \neq \emptyset$;
\item for all $i \in \mathbb{Z}_7$, if $A_i^+ \neq \emptyset$, then $A_{i+3}^+ = A_{i-3}^+ = A_{i-2}^- = A_{i-1}^- = \emptyset$;
\item for all $i \in \mathbb{Z}_7$, if $A_i^- \neq \emptyset$, then $A_{i+1}^+ = A_{i+2}^+ = A_{i+3}^- = A_{i-3}^- = \emptyset$.
\end{itemize}
\end{lemma}
\begin{proof}
This follows from the definition of a 7-bracelet.
\end{proof}

\begin{lemma} \label{lemma-7BTE-P7C4C5-free} Let $B$ be either a 7-bracelet of a thickened emerald. Then $\alpha(B) = 3$, $B$ is $(P_7,C_4,C_5)$-free, and $B$ does not admit a clique-cutset.
\end{lemma}
\begin{proof}
If $B$ is a 7-bracelet, then let $(\{A_i\}_{i \in \mathbb{Z}_7},i^*)$ be a good pair for it, and set $C = \emptyset$. If $B$ is a thickened emerald, then let $(\{A_i\}_{i \in \mathbb{Z}_7},C,i^*)$ be a good triple for it. By Lemma~\ref{lemma-7-bracelet-in-thickened-emerald}, $B \setminus C$ is a 7-bracelet with good pair $(\{A_i\}_{i \in \mathbb{Z}_7},i^*)$.

\begin{quote}
{\em (1) $\alpha(B) = 3$.}
\end{quote}
\begin{proof}[Proof of (1)]
By symmetry, we may assume that $i^* = 0$.

Clearly, for any $a_1 \in A_1$, $a_3 \in A_3$, and $a_5 \in A_5$, $\{a_1,a_3,a_5\}$ is a stable set in $B$; since $A_i \neq \emptyset$ for all $i \in \mathbb{Z}_7$, it follows that $\alpha(B) \geq 3$.

It remains to show that $\alpha(B) \leq 3$. Suppose otherwise, and let $S$ be a stable set of size four in $B$. Suppose first that $S \cap C \neq \emptyset$. Then $C \neq \emptyset$, and it follows that $B$ is a thickened emerald with good triple $(\{A_i\}_{i \in \mathbb{Z}_7},C,0)$. Further, since $C$ is a clique and $S$ a stable set, we see that $|C \cap S| = 1$. Since $C$ is complete to $A_2^* \cup A_3 \cup A_4 \cup A_5^*$, we see that $S \cap (A_2^* \cup A_3 \cup A_4 \cup A_5^*) = \emptyset$; consequently, $S \subseteq C \cup A_5^+ \cup A_6 \cup A_0 \cup A_1 \cup A_2^-$. But this is impossible because $S$ is a stable set of size four, and $C \cup A_5^+ \cup A_6 \cup A_0 \cup A_1 \cup A_2^-$ is the union of three cliques (namely, $C$, $A_5^+ \cup A_6 \cup A_0^-$, and $A_0^+ \cup A_1 \cup A_2^-$).

From now on, we assume that $S \cap C = \emptyset$. Since $B \setminus C$ is a 7-bracelet with good pair $(\{A_i\}_{i \in \mathbb{Z}_7},0)$, we see that $V(B) \setminus C$ can be partitioned into four cliques, namely $A_0$, $A_1 \cup A_2$, $A_3 \cup A_4$, and $A_5 \cup A_6$. Since $S$ is a stable set of size four, we see that $S$ intersects each of these four cliques in exactly one vertex. Let $a_0 \in S \cap A_0$ and $a_3 \in S \cap (A_3 \cup A_4)$; by symmetry, we may assume that $a_3 \in S \cap A_3$. But now $a_0$ is complete to $A_1$, and $a_3$ is complete to $A_2$, and so since $S$ is a stable set, we deduce that $S \cap (A_1 \cup A_2) = \emptyset$, a contradiction. This proves (1).
\end{proof}

\begin{quote}
{\em (2) $B$ is $(P_7,C_4,C_5)$-free.}
\end{quote}
\begin{proof}[Proof of (2)]
Since $\alpha(P_7) = 4$, (1) implies that $B$ is $P_7$-free. It remains to show that $B$ is $(C_4,C_5)$-free. Let $k \in \{4,5\}$, and let $H = x_0,x_1,\dots,x_{k-1},x_0$ be a $k$-hole in $B$.

Suppose first that $V(H) \cap C \neq \emptyset$. Then $C \neq \emptyset$, and $B$ is a thickened emerald with good triple $(\{A_i\}_{i \in \mathbb{Z}_7},C,i^*)$. Note that $A_{i^*-1} \cup A_{i^*+1},A_{i^*-2}^+ \cup A_{i^*+2}^-,A_{i^*-2}^* \cup A_{i^*+2}^*,A_{i^*-3} \cup A_{i^*+3} \cup C$ is a simplicial elimination ordering of $B \setminus A_{i^*}$, and so (by~\cite{FulkersonGrossSimplicialElimOrd}) $B \setminus A_{i^*}$ is chordal. It follows that $V(H) \cap A_{i^*} \neq \emptyset$. Fix $c \in V(H) \cap C$ and $a \in V(H) \cap A_{i^*}$; since $C$ is anticomplete to $A_{i^*}$, we see that $a$ is nonadjacent to $c$. Since $H$ is a hole of length four or five, we see that $c$ and $a$ have a common neighbor in $H$, and therefore in $B$ as well. But no vertex in $B$ has a neighbor both in $A_{i^*}$ and in $C$, a contradiction.

From now on, we assume that $V(H) \cap C = \emptyset$. Thus, $H$ is a hole in $B \setminus C$, and we know that $B \setminus C$ is a 7-bracelet with good pair $(\{A_i\}_{i \in \mathbb{Z}_7},i^*)$.

Let us first show that $|V(H) \cap A_i| \leq 1$ for all $i \in \mathbb{Z}_7$. Suppose otherwise; by symmetry, we may assume that $|V(H) \cap A_0| \geq 2$. Since $A_6 \cup A_0$ and $A_0 \cup A_1$ are cliques, and since $H$ is triangle-free, this implies that $|V(H) \cap A_0| = 2$, that the two vertices in $V(H) \cap A_0$ are adjacent (by symmetry, we may assume that $V(H) \cap A_0 = \{x_0,x_1\}$), and that $V(H) \cap A_6 = V(H) \cap A_1 = \emptyset$. Since $A_0$ is anticomplete to $A_3 \cup A_4$, and since $x_0x_{k-1},x_1x_2 \in E(B)$, we now deduce that $x_2,x_{k-1} \in A_2 \cup A_5$. By symmetry, we may assume that either $x_2,x_{k-1} \in A_2$, or that $x_2 \in A_2$ and $x_{k-1} \in A_5$.

Suppose first that $x_2,x_{k-1} \in A_2$. Since $A_2$ is a clique, we have that $x_2x_{k-1} \in E(B)$, and it follows that $k = 4$ (and so $x_3 \in A_2$). Now, since $x_0 \in A_0$ is adjacent to $x_3 \in A_2$, we have that $x_0 \in A_0^+$. Similarly, since $x_1 \in A_0$ is adjacent to $x_2 \in A_2$, we have that $x_1 \in A_0^+$. It now follows from the definition of a 7-bracelet that one of $x_0,x_1$ dominates the other in $B \setminus C$, and consequently, in $H$ as well. But this is impossible since $H$ is a hole, and no vertex in a hole dominates any other vertex in that hole.

Suppose now that $x_2 \in A_2$ and $x_{k-1} \in A_5$. Then $x_0 \in A_0^-$ and $x_1 \in A_0^+$; in particular, $A_0^-$ and $A_0^+$ are both nonempty, and it follows that $i^* = 0$. Since $A_2$ is anticomplete to $A_5$, we see that $x_2x_{k-1} \notin E(B)$, and we deduce that $k = 5$. It now follows that $a_3$ is adjacent both to $x_2 \in A_2$ and to $x_4 \in A_5$. But this is impossible since $i^* = 0$, and so no vertex in $V(B) \setminus C$ has a neighbor both in $A_2$ and in $A_5$.

We have now shown that $|V(H) \cap A_i| \leq 1$ for all $i \in \mathbb{Z}_7$. Since $4 \leq |V(H)| \leq 5$, we deduce that there exists some $i \in \mathbb{Z}_7$ such that $|V(H) \cap A_i| = 1$ and $V(H) \cap A_{i+1} = \emptyset$; by symmetry, we may assume that $|V(H) \cap A_0| = 1$ and $V(H) \cap A_1 = \emptyset$. Again by symmetry, we may assume that $V(H) \cap A_0 = \{x_0\}$. It then follows that $x_1,x_{k-1} \in A_2 \cup A_5 \cup A_6$. Furthermore, we know that $x_0$ is anticomplete to at least one of $A_2,A_5$, and consequently, either $x_1,x_{k-1} \in A_2 \cup A_6$ or $x_1,x_{k-1} \in A_5 \cup A_6$. However, the latter is impossible because $x_1x_{k-1} \notin E(H)$, and $A_5 \cup A_6$ is a clique. Thus, $x_1,x_{k-1} \in A_2 \cup A_6$. Since $|V(H) \cap A_i| \leq 1$ for all $i \in \mathbb{Z}_7$, we may now assume by symmetry that $x_1 \in A_2$ and $x_{k-1} \in A_6$. (Thus, $V(H) \cap A_2 = \{x_1\}$ and $V(H) \cap A_6 = \{x_{k-1}\}$.)

Suppose that $k = 4$ (thus, $x_3 \in A_6$). Then $x_2$ is adjacent both to $x_1 \in A_2$ and to $x_3 \in A_6$. By the definition of a 7-bracelet, we know that for all $i \in \mathbb{Z}_7$, every vertex in $A_i$ is anticomplete to $A_{i-3} \cup A_{i+3}$ and to at least one of $A_{i-2}$ and $A_{i+2}$. Since $x_2$ has a neighbor both in $A_2$ and in $A_6$, we deduce that $x_2 \in A_0 \cup A_1$. But this is impossible because $V(H) \cap A_0 = \{x_0\}$ and $V(H) \cap A_1 = \emptyset$.

Thus, $k = 5$. Furthermore, we have that $V(H) \cap A_0 = \{x_0\}$, $V(H) \cap A_1 = \emptyset$, $V(H) \cap A_2 = \{x_1\}$, and $V(H) \cap A_6 = \{x_4\}$. Consequently, $x_2,x_3 \in A_3 \cup A_4 \cup A_5$. By the definition of a 7-bracelet, for all $i \in \mathbb{Z}_7$, every vertex in $A_i$ is anticomplete to $A_{i-3} \cup A_{i+3}$ and to at least one of $A_{i-2},A_{i+2}$. Since $x_1 \in A_2$ is adjacent to $x_0 \in A_0$, we deduce that $x_1$ is anticomplete to $A_4 \cup A_5 \cup A_6$. Since $x_1$ is adjacent to $x_2$, and $x_2 \in A_3 \cup A_4 \cup A_5$, we deduce that $x_2 \in A_3$. Since $|V(H) \cap A_i| \leq 1$ for all $i \in \mathbb{Z}_7$, it follows that $V(H) \cap A_3 = \{x_2\}$; consequently, $x_3 \in A_4 \cup A_5$. Since $x_3$ is adjacent to $x_2 \in A_3$ and to $x_4 \in A_6$, we deduce that $x_4 \in A_4^+ \cup A_5^-$, and in particular, $A_4^+ \cup A_5^- \neq \emptyset$. Lemma~\ref{lemma-7-bracelet-no-parallel-2-skips} now implies that $A_0^+ = \emptyset$. But this is impossible since $x_0 \in A_0$ is adjacent to $x_1 \in A_2$, and so $x_0 \in A_0^+$. This proves (2).
\end{proof}

\begin{quote}
{\em (3) $B$ does not admit a clique-cutset.}
\end{quote}
\begin{proof}[Proof of (3)]
Let $S$ be a clique in $V(B)$; we must show that $B \setminus S$ is connected. To simplify notation, set $A = \bigcup_{i \in \mathbb{Z}_7} A_i$. Since $S$ is a clique, we see that there exists some $j \in \mathbb{Z}_7$ such that $S \subseteq C \cup A_{j-1} \cup A_j \cup A_{j+1}$. (In particular, $A \setminus S \neq \emptyset$.) Now, $C$ is either empty or a homogeneous set in $B$, and  the set of vertices in $A$ that are complete to $C$ is not a clique; since $S$ is a clique, we see that some vertex in $A \setminus S$ is complete to $C \setminus S$. It remains to show that $B[A \setminus S]$ is connected. Now, we know that for all $i \in \mathbb{Z}_7$, $A_i$ is a nonempty clique, complete to $A_{i-1} \cup A_{i+1}$. Since $S \subseteq C \cup A_{j-1} \cup A_j \cup A_{j+1}$, this implies that $B[A \setminus (S \cup A_j)]$ is connected. It remains to show that every vertex in $A_j \setminus S$ has a neighbor in $A \setminus (A_j \cup S)$. We know that $A_j$ is complete to $A_{j-1} \cup A_{j+1}$, and so if $A_{j-1} \setminus S \neq \emptyset$ or $A_{j+1} \setminus S \neq \emptyset$, then we are done. Thus, we may assume that $A_{j-1} \cup A_{j+1} \subseteq S$. But this implies that $A_{j-1} \cup A_{j+1}$ is a clique, contrary to the fact that for all $i \in \mathbb{Z}_7$, some vertex in $A_i$ has a nonneighbor both in $A_{i-2}$ and in $A_{i+2}$. This proves that $B \setminus S$ is indeed connected, and it follows that $B$ does not admit a clique-cutset. This proves (3).
\end{proof}

This completes the argument.
\end{proof}

\begin{lemma} \label{lemma-7-hole-attachment-in-P7C4C5-free} Let $G$ be a $(P_7,C_4,C_5)$-free graph, and let $H = x_0,\dots,x_6,x_0$ (with indices in $\mathbb{Z}_7$) be a 7-hole in $G$. Then for all $x \in V(G) \setminus V(H)$, exactly one of the following holds:
\begin{enumerate}[(a)]
\item \label{ref-attachemt-complete} $x$ is complete to $V(H)$;
\item \label{ref-attachemt-anticomplete} $x$ is anticomplete to $V(H)$;
\item \label{ref-attachemt-twin} there exists some $i \in \mathbb{Z}_7$ such that $N_G(x) \cap V(H) = \{x_{i-1},x_i,x_{i+1}\}$;
\item \label{ref-attachemt-four} there exists some $i \in \mathbb{Z}_7$ such that $N_G(x) \cap V(H) = \{x_{i+2},x_{i+3},x_{i-3},x_{i-2}\}$.
\end{enumerate}
\end{lemma}
\begin{proof}
Fix $x \in V(G) \setminus V(H)$. We may assume that $x$ is mixed on $V(H)$, for otherwise (\ref{ref-attachemt-complete}) or (\ref{ref-attachemt-anticomplete}) holds, and we are done. Then there exists some $i \in \mathbb{Z}_7$ such that $x$ is adjacent to $x_i$ and nonadjacent to $x_{i+1}$; by symmetry, we may assume that $xx_0 \in E(G)$ and $xx_1 \notin E(G)$. Now, if $N_G(x) \cap V(H) \subseteq \{x_0,x_6\}$, then since $xx_0 \in E(G)$, we have that $x,x_0,x_1,\dots,x_5$ is an induced $P_7$ in $G$, a contradiction. Thus, $N_G(x) \cap V(H) \not\subseteq \{x_0,x_6\}$. Since $xx_1 \notin E(G)$, we deduce that $x$ has a neighbor in $\{x_2,x_3,x_4,x_5\}$. Note that $xx_2 \notin E(G)$, for otherwise, $x,x_0,x_1,x_2,x$ would be a 4-hole in $G$, a contradiction. Further, $xx_3 \notin E(G)$, for otherwise, $x,x_0,x_1,x_2,x_3,x$ would be a 5-hole in $G$, a contradiction.

We have now shown that $x$ is adjacent to $x_0$, is anticomplete to $\{x_1,x_2,x_3\}$, and has a neighbor in $\{x_4,x_5\}$. Suppose that $xx_5 \notin E(G)$. Then $xx_4 \in E(G)$. But now if $xx_6 \in E(G)$, then $x,x_4,x_5,x_6,x$ is a 4-hole in $G$, and if $xx_6 \notin E(G)$, then $x,x_4,x_5,x_6,x_0,x$ is a 5-hole in $G$, a contradiction in either case. This proves that $xx_5 \in E(G)$. Then $xx_6 \in E(G)$, for otherwise, $x,x_5,x_6,x_0,x$ would be a 4-hole in $G$, a contradiction. We now have that $x$ is complete to $\{x_5,x_6,x_0\}$ and anticomplete to $\{x_1,x_2,x_3\}$ in $G$, and so either $N_G(x) \cap V(H) = \{x_4,x_5,x_6,x_0\}$ or $N_G(x) \cap V(H) = \{x_5,x_6,x_0\}$. In the former case, (\ref{ref-attachemt-four}) holds with $i = 2$, and in the latter case, (\ref{ref-attachemt-twin}) holds with $i = 6$. This completes the argument.
\end{proof}

\begin{lemma} \label{lemma-7-bracelet-7-hole} Let $B$ be a 7-bracelet with good pair $(\{A_i\}_{i \in \mathbb{Z}_7},i^*)$. Then $B$ contains a 7-hole. Furthermore, every 7-hole $H$ in $B$ satisfies the following:
\begin{itemize}
\item $|V(H) \cap A_i| = 1$ for all $i \in \mathbb{Z}_7$,
\item $H$ is of the form $H = x_0,x_1,\dots,x_6,x_0$ (with indices in $\mathbb{Z}_7$), where for all $i \in \mathbb{Z}_7$, $x_i$ is the unique vertex of $V(H) \cap A_i$.
\end{itemize}
\end{lemma}
\begin{proof}
Let us first show that $B$ contains a $7$-hole. Let $a_{i^*} \in A_{i^*}$ be such that $a_{i^*}$ has a nonneighbor both in $A_{i^*-2}$ and in $A_{i^*+2}$; fix $a_{i^*-2} \in A_{i^*-2}$ and $a_{i^*+2} \in A_{i^*+2}$ such that $a_{i^*}a_{i^*-2},a_{i^*}a_{i^*+2} \notin E(B)$. Further, fix nonadjacent vertices $a_{i^*-1} \in A_{i^*-1}$ and $a_{i^*+1} \in A_{i^*+1}$. Finally, fix arbitrary $a_{i^*-3} \in A_{i^*-3}$ and $a_{i^*+3} \in A_{i^*+3}$. Then $a_0,a_1,\dots,a_6,a_0$ is a 7-hole in $B$.

Now, let $H$ be a 7-hole in $B$. We must show that $H$ satisfies the following:
\begin{enumerate}[(i)]
\item $|V(H) \cap A_i| = 1$ for all $i \in \mathbb{Z}_7$;
\item $H$ is of the form $H = x_0,x_1,\dots,x_6,x_0$ (with indices in $\mathbb{Z}_7$), where for all $i \in \mathbb{Z}_7$, $x_i$ is the unique vertex of $V(H) \cap A_i$.
\end{enumerate}
It suffices to show that $|V(H) \cap A_i| \leq 1$ for all $i \in \mathbb{Z}_7$, for the fact that $|V(H)| = 7$ will then immediately imply (i), and this, together with the fact that $A_i$ is complete to $A_{i+1}$ for all $i \in \mathbb{Z}_7$, will immediately imply (ii).

Suppose that there exists some $i \in \mathbb{Z}_7$ such that $|V(H) \cap A_i| \geq 2$; by symmetry, we may assume that $|V(H) \cap A_0| \geq 2$. Since $A_6 \cup A_0$ and $A_0 \cup A_1$ are cliques, and since $H$ is triangle-free, it follows that $|V(H) \cap A_0| = 2$ and $V(H) \cap A_1 = V(H) \cap A_6 = \emptyset$. Consequently, $|V(H) \cap (A_2 \cup A_3 \cup A_4 \cup A_5)| = 5$. But this implies that either $|V(H) \cap (A_2 \cup A_3)| \geq 3$ or $|V(H) \cap (A_4 \cup A_5)| \geq 3$, which is impossible because $A_2 \cup A_3$ and $A_4 \cup A_5$ are cliques, and $H$ is triangle-free. This completes the argument.
\end{proof}

Let $B$ be a graph, and let $\{A_i\}_{i \in \mathbb{Z}_7}$ be a partition of $V(B)$. We say that $B$ is a {\em quasi-7-bracelet} with {\em good partition} $\{A_i\}_{i \in \mathbb{Z}_7}$ provided the following hold:
\begin{itemize}
\item for all $i \in \mathbb{Z}_7$, $A_i$ is a nonempty clique, complete to $A_{i-1} \cup A_{i+1}$ and anticomplete to $A_{i-3} \cup A_{i+3}$;
\item for all $i \in \mathbb{Z}_7$, some vertex in $A_i$ has a nonneighbor both in $A_{i-2}$ and in $A_{i+2}$.
\end{itemize}

A graph $B$ is said to be a {\em quasi-7-bracelet} provided that there exists a partition $\{A_i\}_{i \in \mathbb{Z}_7}$ of $V(B)$ such that $B$ is a quasi-7-bracelet with good partition $\{A_i\}_{i \in \mathbb{Z}_7}$.

\begin{lemma} Let $B$ be a 7-bracelet with good pair $(\{A_i\}_{i \in \mathbb{Z}_7},i^*)$. Then $B$ is a quasi-7-bracelet with good partition $\{A_i\}_{i \in \mathbb{Z}_7}$.
\end{lemma}
\begin{proof}
This is immediate from the appropriate definitions.
\end{proof}

\begin{lemma} \label{lemma-quasi-7-bracelet-P7C4C5-free-is-7-bracelet} Let $B$ be a quasi-7-bracelet with good partition $\{A_i\}_{i \in \mathbb{Z}_7}$, and assume that $B$ is $(P_7,C_4,C_5)$-free. Then there exists some index $i^* \in \mathbb{Z}_7$ such that $B$ is a 7-bracelet with good pair $(\{A_i\}_{i \in \mathbb{Z}_7},i^*)$.
\end{lemma}
\begin{proof}
We must show that there exists some $i^* \in \mathbb{Z}_7$ such that $B$, $\{A_i\}_{i \in \mathbb{Z}_7}$, and $i^*$ satisfy axioms (I)-(V) from the definition of a 7-bracelet. The fact that (I) holds follows from the definition of a quasi-7-bracelet. Our next goal is to prove (II).

\begin{quote}
\emph{(1) For all $i \in \mathbb{Z}_7$, every vertex in $A_i$ is anticomplete to at least one of $A_{i-2},A_{i+2}$.}
\end{quote}
\begin{proof}[Proof of (1)]
Suppose otherwise. By symmetry, we may assume that there exist some $a_0 \in A_0$, $a_2 \in A_2$, and $a_5 \in A_5$ such that $a_0a_2,a_0a_5 \in E(B)$. Suppose first that $a_2$ is adjacent to some $a_4 \in A_4$. Then $a_0,a_2,a_4,a_5,a_0$ is a 4-hole in $B$, a contradiction. Thus, $a_2$ is anticomplete to $A_4$, and similarly, $a_5$ is anticomplete to $A_3$. Now, fix arbitrary $a_3 \in A_3$ and $a_4 \in A_4$. Then $a_0,a_2,a_3,a_4,a_5,a_0$ is a 5-hole in $B$, a contradiction. This proves (1).
\end{proof}

For all $i \in \mathbb{Z}_7$:
\begin{itemize}
\item let $A_i^*$ be the set of all vertices in $A_i$ that are anticomplete to $A_{i-2} \cup A_{i+2}$;
\item let $A_i^+$ be the set of all vertices in $A_i$ that have a neighbor in $A_{i+2}$ and are anticomplete to $A_{i-2}$;
\item let $A_i^-$ be the set of all vertices in $A_i$ that have a neighbor in $A_{i-2}$ and are anticomplete to $A_{i+2}$.
\end{itemize}
For all $i \in \mathbb{Z}_7$, $A_i^*,A_i^+,A_i^-$ are disjoint (by construction), and by (1), their union is $A_i$. By construction, (II.a), (II.b), and (II.c) hold. Furthermore, the definition of a quasi-7-bracelet implies that (II.f) holds.

\begin{quote}
\emph{(2) For all $i \in \mathbb{Z}_7$, the following hold:
\begin{itemize}
\item if $A_i^+ \neq \emptyset$, then $A_i^+$ can be ordered as $A_i^+ = \{a_1^{i^+},\dots,a_{|A_i^+|}^{i^+}\}$ so that $N_B[a_{|A_i^+|}^{i^+}] \subseteq \dots \subseteq N_B[a_1^{i^+}]$;
\item if $A_i^- \neq \emptyset$, then $A_i^-$ can be ordered as $A_i^- = \{a_1^{i^-},\dots,a_{|A_i^-|}^{i^-}\}$ so that $N_B[a_{|A_i^-|}^{i^-}] \subseteq \dots \subseteq N_B[a_1^{i^-}]$.
\end{itemize}}
\end{quote}
\begin{proof}[Proof of (2)]
Fix $i \in \mathbb{Z}_7$. By symmetry, it suffices to prove the first statement. If $|A_i^+| \leq 1$, then this is immediate. So assume that $|A_i^+| \geq 2$. Clearly, it suffices to show that for any two distinct vertices in $A_i^+$, one of the two vertices dominates the other in $B$. Fix distinct $a_i,a_i' \in A_i^+$; we must show that one of $a_i,a_i'$ dominates the other. By the definition of a quasi-7-bracelet, we have that $A_{i-1} \cup A_i \cup A_{i+1} \subseteq N_B[a_i] \cap N_B[a_i']$, and that both $a_i,a_i'$ are anticomplete to $A_{i-3} \cup A_{i+3}$; furthermore, by the construction of $A_i^+$, we have that $a_i,a_i'$ are anticomplete to $A_{i-2}$. Thus, it suffices to show that one of $N_B[a_i] \cap A_{i+2}$ and $N_B[a_i'] \cap A_{i+2}$ is included in the other. Suppose otherwise, and fix $a_{i+2},a_{i+2}' \in A_{i+2}$ such that $a_ia_{i+2},a_i'a_{i+2}' \in E(B)$ and $a_ia_{i+2}',a_i'a_{i+2} \notin E(B)$. Since $A_i,A_{i+2}$ are cliques, it now follows that $a_i,a_{i+2},a_{i+2}',a_i',a_i$ is a 4-hole in $B$, a contradiction. This proves (2).
\end{proof}

By (2), (II.d) and (II.e) hold. We now deduce that (II) holds.

\begin{quote}
\emph{(3) All the following hold:
\begin{enumerate}[(i)]
\item for all $i \in \mathbb{Z}_7$, $A_{i-1}^+ \neq \emptyset$ if and only if $A_{i+1}^- \neq \emptyset$;
\item for all $i \in \mathbb{Z}_7$, if $A_i^+ \neq \emptyset$, then $A_{i+3}^+ = A_{i-3}^+ = A_{i-2}^- = A_{i-1}^- = \emptyset$;
\item for all $i \in \mathbb{Z}_7$, if $A_i^- \neq \emptyset$, then $A_{i+1}^+ = A_{i+2}^+ = A_{i+3}^- = A_{i-3}^- = \emptyset$.
\end{enumerate}}
\end{quote}
\begin{proof}[Proof of (3)]
The fact that (i) holds is immediate from the construction. It remains to prove (ii) and (iii). We prove (ii); the proof of (iii) is completely analogous. Fix $i \in \mathbb{Z}_7$, and assume that $A_i^+ \neq \emptyset$ (and consequently, $A_{i+2}^- \neq \emptyset$); we must show that $A_{i+3}^+ = A_{i-3}^+ = A_{i-2}^- = A_{i-1}^- = \emptyset$. Using the fact that $A_i^+ \neq \emptyset$, we fix adjacent vertices $a_i \in A_i^+$ and $a_{i+2} \in A_{i+2}^-$.

Suppose that $A_{i+3}^+ \neq \emptyset$, and fix adjacent $a_{i+3} \in A_{i+3}^+$ and $a_{i-2} \in A_{i-2}^-$. Fix an arbitrary vertex $a_{i-1} \in A_{i-1}$. Then $a_i,a_{i+2},a_{i+3},a_{i-2},a_{i-1},a_i$ is a 5-hole in $B$ a contradiction. Thus, $A_{i+3}^+ = \emptyset$, and consequently, $A_{i-2}^- = \emptyset$.

Suppose now that $A_{i-3}^+ \neq \emptyset$, and fix adjacent $a_{i-3} \in A_{i-3}^+$ and $a_{i-1} \in A_{i-1}^-$. Fix an arbitrary vertex $a_{i+3} \in A_{i+3}$. Then $a_i,a_{i+2},a_{i+3},a_{i-3},a_{i-1},a_i$ is a 5-hole in $B$, a contradiction. Thus, $A_{i-3}^+ = \emptyset$, and consequently, $A_{i-1}^- = \emptyset$. This proves (3).
\end{proof}

It remains to show that there exists some $i^* \in \mathbb{Z}_7$ such that (III), (IV), and (V) hold.

Suppose first that there exists some $i^* \in \mathbb{Z}_7$ such that $A_{i^*}^+,A_{i^*}^- \neq \emptyset$. The fact that (III), (IV), and (V) hold now readily follows from (3).

From now on, we assume that for all $i \in \mathbb{Z}_7$, at least one of $A_i^+,A_i^-$ is empty. If $A_i^+ = A_i^- = \emptyset$ for all $i \in \mathbb{Z}_7$, then we pick an arbitrary $i^* \in \mathbb{Z}_7$, and we observe that (III), (IV), and (V) hold. So from now on, we may assume by symmetry that $A_6^+ \cup A_6^- \neq \emptyset$. Now, we know that at least one of $A_6^+,A_6^-$ is empty, and so by symmetry, we may assume that $A_6^+ \neq \emptyset$ and $A_6^- = \emptyset$; consequently, $A_1^- \neq \emptyset$, $A_1^+ = \emptyset$, and $A_4^+ = \emptyset$. By (3), with $i = 6$, we now have that $A_2^+ = A_3^+ = A_4^- = A_5^- = \emptyset$. Furthermore, since $A_1^+ = \emptyset$, we have that $A_3^- = \emptyset$. We now set $i^* = 0$, and we observe that (III), (IV), and (V) hold. This completes the argument.
\end{proof}

\begin{lemma} \label{lemma-7-hole-twins-7-bracelet} Let $G$ be a $(P_7,C_4,C_5)$-free graph, and let $H = x_0,x_1,\dots,x_6,x_0$ (with indices in $\mathbb{Z}_7$) be a 7-hole in $G$. For all $i \in \mathbb{Z}_7$, set $A_i = \{x \in V(G) \mid N_G[x] \cap V(H) = \{x_{i-1},x_i,x_{i+1}\}\}$. Set $B = G[\bigcup_{i \in \mathbb{Z}_7} A_i]$. Then $B$ is a 7-bracelet, and $\{A_i\}_{i \in \mathbb{Z}_7}$ is a good partition for it.
\end{lemma}
\begin{proof}
In view of Lemma~\ref{lemma-quasi-7-bracelet-P7C4C5-free-is-7-bracelet}, it suffices to show that $B$ is a quasi-7-bracelet with good partition $\{A_i\}_{i \in \mathbb{Z}_7}$.

\begin{quote}
\emph{(1) For all $i \in \mathbb{Z}_7$, $x_i \in A_i$ and $A_i$ is a nonempty clique.}
\end{quote}
\begin{proof}[Proof of (1)]
Fix $i \in \mathbb{Z}_7$. By construction, $x_i \in A_i$, and consequently, $A_i \neq \emptyset$. Suppose that $A_i$ is not a clique, and fix distinct, nonadjacent vertices $a_i,a_i' \in A_i$. Then $a_i,x_{i+1},a_i',x_{i-1},a_i$ is a 4-hole in $G$, a contradiction. This proves (1).
\end{proof}

\begin{quote}
\emph{(2) For all $i \in \mathbb{Z}_7$, $A_i$ is complete to $A_{i-1} \cup A_{i+1}$ and anticomplete to $A_{i-3} \cup A_{i+3}$.}
\end{quote}
\begin{proof}[Proof of (2)]
By symmetry, it suffices to prove that $A_0$ is complete to $A_6 \cup A_1$ and anticomplete to $A_3 \cup A_4$.

Suppose that $A_0$ is not complete to $A_6 \cup A_1$. By symmetry, we may assume that some $a_0 \in A_0$ is nonadjacent to some $a_1 \in A_1$. But then $a_1,x_2,x_3,x_4,x_5,x_6,a_0$ is an induced $P_7$ in $G$, a contradiction. Thus, $A_0$ is complete to $A_6 \cup A_1$.

Suppose that $A_0$ is not anticomplete to $A_3 \cup A_4$. By symmetry, we may assume that some $a_0 \in A_0$ and $a_3 \in A_3$ are adjacent. But then $a_0,x_1,x_2,a_3,a_0$ is a 4-hole in $G$, a contradiction. Thus, $A_0$ is anticomplete to $A_3 \cup A_4$. This proves (2).
\end{proof}

Finally, note that for all $i \in \mathbb{Z}_7$, the vertex $x_i \in A_i$ is nonadjacent both to $x_{i-2} \in A_{i-2}$ and to $x_{i+2} \in A_{i+2}$. It now follows that $B$ is a quasi-7-bracelet with good partition $\{A_i\}_{i \in \mathbb{Z}_7}$, and we are done.
\end{proof}

\begin{lemma} \label{lemma-add-vertex-to-7-bracelet} Let $G$ be a $(P_7,C_4,C_5)$-free graph, let $B$ be an induced 7-bracelet in $G$, let $\{A_i\}_{i \in \mathbb{Z}_7}$ be a good partition of the 7-bracelet $B$, and let $H = x_0,x_1,\dots,x_6,x_0$ be a 7-hole in $B$ such that $x_i \in A_i$ for all $i \in \mathbb{Z}_7$. Let $c \in V(G) \setminus V(B)$ and $\ell \in \mathbb{Z}_7$, and assume that $N_G(c) \cap V(H) = \{x_{\ell+2},x_{\ell+3},x_{\ell-3},x_{\ell-2}\}$. Set $B_c = G[V(B) \cup \{c\}]$. Then either $B_c$ is a 7-bracelet, or $B_c$ is a thickened emerald with good triple $(\{A_i\}_{i \in \mathbb{Z}_7},\{c\},\ell)$.
\end{lemma}
\begin{proof}
For all $i \in \mathbb{Z}_7$, let $A_i^*,A_i^+,A_i^-$ be as in the definition of a 7-bracelet. By symmetry, we may assume that $\ell = 0$, so that $N_G(c) \cap V(H) = \{x_2,x_3,x_4,x_5\}$. We must show that either $B_c$ is a 7-bracelet, or $B_c$ is a thickened emerald with good triple $(\{A_i\}_{i \in \mathbb{Z}_7},\{c\},0)$.

\begin{quote}
\emph{(1) Vertex $c$ is complete to $A_3 \cup A_4$ and anticomplete to $A_6 \cup A_0 \cup A_1$.}
\end{quote}
\begin{proof}[Proof of (1)]
If $c$ has a nonneighbor $a_3 \in A_3$, then $c,x_2,a_3,x_4,c$ is a 4-hole in $G$, a contradiction. Thus, $c$ is complete to $A_3$, and similarly, $c$ is complete to $A_4$.

Suppose that $c$ has a neighbor $a_0 \in A_0$. Then $a_0x_2 \in E(G)$, for otherwise, $c,a_0,x_1,x_2,c$ is a 4-hole in $G$, a contradiction. Similarly, $a_0x_5 \in E(G)$, for otherwise, $c,a_0,x_6,x_5,c$ is a 4-hole in $G$, a contradiction. But now $a_0 \in A_0$ is adjacent both to $x_2 \in A_2$ and to $x_5 \in A_5$, contrary to the fact that (by the definition of a 7-bracelet) every vertex in $A_0$ is anticomplete to at least one of $A_2,A_5$. This proves that $c$ is anticomplete to $A_0$.

Suppose that $c$ has a neighbor $a_1 \in A_1$. But then if $a_1x_6 \in E(G)$, then $c,a_1,x_6,x_5,c$ is a 4-hole in $G$, and if $a_1x_6 \notin E(G)$, then $c,a_1,x_0,x_6,x_5,c$ is a 5-hole in $G$, a contradiction in either case. Thus, $c$ is anticomplete to $A_1$, and similarly, $c$ is anticomplete to $A_2$. This proves (1).
\end{proof}

\begin{quote}
\emph{(2) $A_0$ is anticomplete to $N_G(c) \cap (A_2 \cup A_5)$.}
\end{quote}
\begin{proof}[Proof of (2)]
Suppose otherwise. By symmetry, we may assume that some $a_0 \in A_0$ and $a_2 \in N_G(c) \cap A_2$ are adjacent. By the definition of a 7-bracelet, every vertex in $A_0$ is anticomplete to at least one of $A_2,A_5$. Since $a_0 \in A_0$ is adjacent to $a_2 \in A_2$, we deduce that $a_0$ is anticomplete to $A_5$, and in particular, $a_0x_5 \notin E(G)$. But now $c,x_5,x_6,a_0,a_2,c$ is a 5-hole in $G$, a contradiction. This proves (2).
\end{proof}

\begin{quote}
\emph{(3) Every vertex in $A_0$ is complete to at least one of the sets $A_2 \setminus N_G(c)$ and $A_5 \setminus N_G(c)$.}
\end{quote}
\begin{proof}[Proof of (3)]
Suppose otherwise, and fix $a_0 \in A_0$, $a_2 \in A_2 \setminus N_G(c)$, and $a_5 \in A_5 \setminus N_G(c)$ such that $a_0a_2,a_0a_5 \notin E(G)$. By (1), $ca_0 \notin E(G)$. But now $c,x_3,a_2,x_1,a_0,x_6,a_5$ is an induced $P_7$ in $G$, a contradiction. This proves (3).
\end{proof}

\begin{quote}
\emph{(4) $A_1$ is anticomplete to $A_6$.}
\end{quote}
\begin{proof}[Proof of (4)]
Suppose otherwise, and fix adjacent vertices $a_1 \in A_1$ and $a_6 \in A_6$. By (1), $c$ is anticomplete to $\{a_1,a_6\}$. But now $c,x_5,a_6,a_1,x_2,c$ is a 5-hole in $G$, a contradiction. This proves (4).
\end{proof}

\begin{quote}
\emph{(5) If $A_0$ is complete to at least one of the sets $A_2 \setminus N_G(c)$ and $A_5 \setminus N_G(c)$, then $B_c$ is a 7-bracelet.}
\end{quote}
\begin{proof}[Proof of (5)]
Suppose that $A_0$ is complete to at least one of the sets $A_2 \setminus N_G(c)$ and $A_5 \setminus N_G(c)$; by symmetry, we may assume that $A_0$ is complete to $A_2 \setminus N_G(c)$. Now, set $A_1' = A_1 \cup (A_2 \setminus N_G(c))$, $A_2' = A_2 \cap N_G(c)$, $A_3' = A_3 \cup \{c\}$, and $A_i' = A_i$ for all $i \in \{4,5,6,0\}$. Note that $A_i' \neq \emptyset$ for all $i \in \mathbb{Z}_7$. (Indeed, since $cx_2 \in E(G)$, we see that $x_2 \in A_2'$, and for all $i \in \mathbb{Z}_7 \setminus \{2\}$, we have that $A_i' \supseteq A_i \neq \emptyset$.) Furthermore, by construction, $\{A_i\}_{i \in \mathbb{Z}_7}$ is a partition of $V(B_c)$.

Now, our goal is to show that $B_c$ is a 7-bracelet with good partition $\{A_i'\}_{i \in \mathbb{Z}_7}$. By Lemma~\ref{lemma-quasi-7-bracelet-P7C4C5-free-is-7-bracelet}, it suffices to show that $B_c$ is a quasi-7-bracelet with good partition $\{A_i'\}_{i \in \mathbb{Z}_7}$. For this, we must prove the following:
\begin{enumerate}[(i)]
\item for all $i \in \mathbb{Z}_7$, $A_i'$ is a nonempty clique, complete to $A_{i-1}' \cup A_{i+1}'$ and anticomplete to $A_{i-3}' \cup A_{i+3}'$;
\item for all $i \in \mathbb{Z}_7$, some vertex in $A_i'$ has a nonneighbor both in $A_{i-2}'$ and in $A_{i+2}'$.
\end{enumerate}

We have already shown that $A_i' \neq \emptyset$ for all $i \in \mathbb{Z}_7$. Next, since $A_i$ is a clique for all $i \in \mathbb{Z}_7$, since (by (1)) $c$ is complete to $A_3$, and since $A_1$ is complete to $A_2$, we see that $A_i'$ is a clique for all $i \in \mathbb{Z}_7$.

Next, we claim that for all $i \in \mathbb{Z}_7$, $A_i'$ is complete to $A_{i-1}' \cup A_{i+1}'$; clearly, it suffices to show that for all $i \in \mathbb{Z}_7$, $A_i'$ is complete to $A_{i+1}'$. Since $A_i$ is complete to $A_{i+1}$ for all $i \in \mathbb{Z}_7$, and since (by (1)) $c$ is complete to $A_3 \cup A_4$, it follows immediately from the construction that $A_i'$ is complete to $A_{i+1}'$ for all $i \in \mathbb{Z}_7 \setminus \{0,1\}$. Next, since $A_1$ is complete to $A_2$, and since $A_2$ is a clique, we see that $A_1'$ is complete to $A_2'$. Finally, since $A_0$ is complete to $A_1$ and to $A_2 \setminus N_G(c)$, we see that $A_0'$ is complete to $A_1'$. This proves that $A_i'$ is complete to $A_{i+1}'$ for all $i \in \mathbb{Z}_7$; consequently, $A_i'$ is complete to $A_{i-1}' \cup A_{i+1}'$ for all $i \in \mathbb{Z}_7$.

Further, we must show that $A_i'$ is anticomplete to $A_{i-3}' \cup A_{i+3}'$ for all $i \in \mathbb{Z}_7$; clearly, it suffices to show that for all $i \in \mathbb{Z}_7$, $A_i'$ is anticomplete to $A_{i+3}'$. Since $A_i$ is anticomplete to $A_{i-3} \cup A_{i+3}$ for all $i \in \mathbb{Z}_7$, and since (by (1)) $c$ is anticomplete to $A_6 \cup A_0 \cup A_1$, we readily see that $A_i'$ is anticomplete to $A_{i+3}'$ for all $i \in \mathbb{Z}_7 \setminus \{1\}$. It remains to show that $A_1'$ is anticomplete to $A_4'$. By construction, $A_1' = A_1 \cup (A_2 \setminus N_G(c))$ and $A_4' = A_4$; since $A_1$ is anticomplete to $A_4$, we need only show that $A_2 \setminus N_G(c)$ is anticomplete to $A_4$. By the definition of a 7-bracelet, we know that every vertex in $A_2$ is anticomplete to at least one of $A_0,A_4$; since $A_0 \neq \emptyset$ is complete to $A_2 \setminus N_G(c)$, it follows that $A_2 \setminus N_G(c)$ is anticomplete to $A_4$. Thus, $A_1'$ is anticomplete to $A_4'$. It follows that $A_i'$ is anticomplete to $A_{i-3}' \cup A_{i+3}'$ for all $i \in \mathbb{Z}_7$. This proves (i).

It remains to prove (ii). By construction, for all $i \in \mathbb{Z}_7$, we have that $x_i \in A_i'$ (for $i = 2$, we use the fact that $cx_2 \in E(G)$), and $x_i$ is anticomplete to $\{x_{i-2},x_{i+2}\}$. This proves (ii).

Thus, $B_c$ is a quasi-7-bracelet, and so by Lemma~\ref{lemma-quasi-7-bracelet-P7C4C5-free-is-7-bracelet}, $B_c$ is a 7-bracelet. This proves (5).
\end{proof}

\begin{quote}
\emph{(6) If $A_0$ is complete neither to $A_2 \setminus N_G(c)$ nor to $A_5 \setminus N_G(c)$, then $B_c$ is a thickened emerald with good triple $(\{A_i\}_{i \in \mathbb{Z}_7},\{c\},0)$.}
\end{quote}
\begin{proof}[Proof of (6)]
Assume that $A_0$ is complete neither to $A_2 \setminus N_G(c)$ nor to $A_5 \setminus N_G(c)$. In particular, $A_2 \setminus N_G(c)$ and $A_5 \setminus N_G(c)$ are both nonempty. Now, we know that $\{c\},A_0,\dots,A_6$ are nonempty cliques, and that they form a partition of $V(B_c)$. Thus, to show that $B_c$ is a thickened emerald with good triple $(\{A_i\}_{i \in \mathbb{Z}_7},\{c\},0)$, we need only prove the following:
\begin{enumerate}[(i)]
\item \label{ref-emerald-from-7-bracelet-Ai-comp-anticomp} for all $i \in \mathbb{Z}_7$, $A_i$ is complete to $A_{i-1} \cup A_{i+1}$ and anticomplete to $A_{i-3} \cup A_{i+3}$;
\item \label{ref-emerald-from-7-bracelet-i-*-only} for all $i \in \{1,3,4,6\}$, $A_i$ is anticomplete to $A_{i-2} \cup A_{i+2}$;
\item \label{ref-emerald-from-7-bracelet-nonempty} $A_0^-,A_0^+,A_2^*,A_2^-,A_5^*,A_5^+ \neq \emptyset$;
\item \label{ref-emerald-from-7-bracelet-i-0-pm} $A_0 = A_0^- \cup A_0^+$;
\item \label{ref-emerald-from-7-bracelet-i-2} $A_2 = A_2^* \cup A_2^-$;
\item \label{ref-emerald-from-7-bracelet-i-5} $A_5 = A_5^* \cup A_5^+$;
\item \label{ref-emerald-from-7-bracelet-i-0m-comp-anticomp} $A_0^-$ is complete to $A_5^+$ and anticomplete to $A_5^* \cup A_2$;
\item \label{ref-emerald-from-7-bracelet-i-0p-comp-anticomp} $A_0^+$ is complete to $A_2^-$ and anticomplete to $A_2^* \cup A_5$;
\item \label{ref-emerald-from-7-bracelet-C} $c$ is complete to $A_2^* \cup A_3 \cup A_4 \cup A_5^*$ and anticomplete to $A_5^+ \cup A_6 \cup A_0 \cup A_1 \cup A_2^-$.
\end{enumerate}

Statement (\ref{ref-emerald-from-7-bracelet-Ai-comp-anticomp}) follows from the fact that $\{A_i\}_{i \in \mathbb{Z}_7}$ is a good partition of the 7-bracelet $B$.

Next, we prove (\ref{ref-emerald-from-7-bracelet-i-0-pm}). By (3), every vertex in $A_0$ is complete to at least one of the sets $A_2 \setminus N_G(c)$ and $A_5 \setminus N_G(c)$. Since $A_2 \setminus N_G(c)$ and $A_5 \setminus N_G(c)$ are both nonempty, it follows that $A_0^* = \emptyset$, and consequently, $A_0 = A_0^- \cup A_0^+$. This proves (\ref{ref-emerald-from-7-bracelet-i-0-pm}).

Next, we prove that $A_0^-,A_0^+ \neq \emptyset$. By hypothesis, $A_0$ is complete neither to $A_2 \setminus N_G(c)$ nor to $A_5 \setminus N_G(c)$. Fix $a_0^-,a_0^+ \in A_0$ such that $a_0^-$ is not complete to $A_2 \setminus N_G(c)$, and $a_0^+$ is not complete to $A_5 \setminus N_G(c)$. Then (by (3)) $a_0^-$ is complete to $A_5 \setminus N_G(c)$, and $a_0^+$ is complete to $A_2 \setminus N_G(c)$. Since $A_2 \setminus N_G(c)$ and $A_5 \setminus N_G(c)$ are both nonempty, it follows that $a_0^- \in A_0^-$ and $a_0^+ \in A_0^+$, and in particular, $A_0^-,A_0^+ \neq \emptyset$.

Since $A_0^-,A_0^+ \neq \emptyset$, Lemma~\ref{lemma-7-bracelet-no-parallel-2-skips} implies that $A_2^-,A_5^+ \neq \emptyset$, that $A_3^+ = A_4^+ = A_5^- = A_6^- = \emptyset$, and that $A_1^+ = A_2^+ = A_3^- = A_4^- = \emptyset$. Furthermore, by (4), we have that $A_1^- = A_6^+ = \emptyset$. It now follows that $A_1 = A_1^*$, $A_2 = A_2^* \cup A_2^-$, $A_3 = A_3^*$, $A_4 = A_4^*$, $A_5 = A_5^* \cup A_5^+$, and $A_6 = A_6^*$. Thus proves (\ref{ref-emerald-from-7-bracelet-i-*-only}), (\ref{ref-emerald-from-7-bracelet-i-2}), and (\ref{ref-emerald-from-7-bracelet-i-5}).

We now prove (\ref{ref-emerald-from-7-bracelet-nonempty}). We have already shown that $A_0^-,A_0^+,A_2^-,A_5^+ \neq \emptyset$; it remains to show that $A_2^*,A_5^* \neq \emptyset$. We claim that $x_2 \in A_2^*$ and $x_5 \in A_5^*$. By hypothesis, $x_2 \in N_G(c) \cap A_2$ and $x_5 \in N_G(c) \cap A_5$, and so by (2), $x_2,x_5$ are anticomplete to $A_0$. This implies that $x_2 \in A_2^* \cup A_2^+$ and $x_5 \in A_5^* \cup A_5^-$. But we already showed that $A_2^+ = A_5^- = \emptyset$. Thus, $x_2 \in A_2^*$ and $x_5 \in A_5^*$, and in particular, $A_2^*,A_5^* \neq \emptyset$. This proves (\ref{ref-emerald-from-7-bracelet-nonempty}).

Next, we prove (\ref{ref-emerald-from-7-bracelet-C}). By (1), $c$ is complete to $A_3 \cup A_4$ and anticomplete to $A_6 \cup A_0 \cup A_1$. Further, it follows from (2) that $c$ is anticomplete to $A_5^+ \cup A_2^-$. It remains to show that $c$ is complete to $A_2^* \cup A_5^*$. Suppose otherwise; by symmetry, we may assume that $c$ has a nonneighbor $a_2^* \in A_2^*$. Then $a_2^* \in A_2 \setminus N_G(c)$. By hypothesis, there exists some $a_0 \in A_0$ such that $a_0$ is not complete to $A_5 \setminus N_G(c)$; by (3), it follows that $a_0$ is complete to $A_2 \setminus N_G(c)$, and in particular, $a_0a_2^* \in E(G)$. But this is impossible since $a_2^* \in A_2^*$, and $A_2^*$ is anticomplete to $A_0$. This proves (\ref{ref-emerald-from-7-bracelet-C}).

It remains to prove (\ref{ref-emerald-from-7-bracelet-i-0m-comp-anticomp}) and (\ref{ref-emerald-from-7-bracelet-i-0p-comp-anticomp}). Since $\{A_i\}_{i \in \mathbb{Z}_7}$ is a good partition of the 7-bracelet $B$, we have that $A_0^-$ is anticomplete to $A_5^* \cup A_2$, and that $A_0^+$ is anticomplete to $A_2^* \cup A_5$. It remains to show that $A_0^-$ is complete to $A_5^+$, and that $A_0^+$ is complete to $A_2^-$; by symmetry, it suffices to show that $A_0^+$ is complete to $A_2^-$. Suppose otherwise, and fix nonadjacent $a_0^+ \in A_0^+$ and $a_2^- \in A_2^-$. By (\ref{ref-emerald-from-7-bracelet-C}), $a_2^- \in A_2 \setminus N_G(c)$, and so since $a_0^+$ is not complete to $A_2 \setminus N_G(c)$, (3) implies that $a_0^+$ is complete to $A_5 \setminus N_G(c)$. Since $A_0^+$ is anticomplete to $A_5$, it follows that $A_5 \setminus N_G(c) = \emptyset$. But this contradicts our assumption that $A_5 \setminus N_G(c) \neq \emptyset$. Thus, (\ref{ref-emerald-from-7-bracelet-i-0m-comp-anticomp}) and (\ref{ref-emerald-from-7-bracelet-i-0p-comp-anticomp}) hold.

This proves (6).
\end{proof}

By (5) and (6), we have that either $B_c$ is a 7-bracelet, or $B_c$ is a thickened emerald with good triple $(\{A_i\}_{i \in \mathbb{Z}_7},\{c\},0)$. This completes the argument.
\end{proof}

\begin{lemma} \label{lemma-dominating-7-hole-anticomp-7-bracelet-or-emerald} Let $G$ be a $(P_7,C_4,C_5)$-free graph, and assume that $G$ contains a dominating 7-hole. Then $G$ contains exactly one nontrivial anticomponent, and this anticomponent is either a 7-bacelet or a thickened emerald.
\end{lemma}
\begin{proof}
Let $H = x_0,x_1,\dots,x_6,x_0$ be a dominating 7-hole in $G$. For all $i \in \mathbb{Z}_7$, let $T_i = \{x \in V(G) \mid N_G[x] \cap V(H) = \{x_{i-1},x_i,x_{i+1}\}\}$. (Note that $x_i \in T_i$ for all $i \in \mathbb{Z}_7$.) By Lemma~\ref{lemma-7-hole-twins-7-bracelet}, $G[\bigcup_{i \in \mathbb{Z}_7} T_i]$ is a 7-bracelet, and $\{T_i\}_{i \in \mathbb{Z}_7}$ is a good partition for it. Let $B$ be a maximal induced 7-bracelet in $G$ such that $\bigcup_{i \in \mathbb{Z}_7} T_i \subseteq V(B)$, and let $\{A_i\}_{i \in \mathbb{Z}_7}$ be good partition of the 7-bracelet $B$. Now, $H$ is a 7-hole in the 7-bracelet $B$. By Lemma~\ref{lemma-7-bracelet-7-hole}, and by symmetry, we may assume that $x_i \in A_i$ for all $i \in \mathbb{Z}_7$. Let $U$ be the set of all vertices in $V(G) \setminus V(B)$ that are complete to $V(H)$.

\begin{quote}
\emph{(1) $U$ is a (possibly empty) clique, complete to $V(B)$.}
\end{quote}
\begin{proof}[Proof of (1)]
Suppose that $U$ is not a clique, and fix distinct, nonadjacent vertices $u_1,u_2 \in U$. But now $x_0,u_1,x_2,u_2,x_0$ is a 4-hole in $G$, a contradiction. Thus, $U$ is a clique.

It remains to show that $U$ is complete to $V(B)$. Suppose otherwise. By symmetry, we may assume that some $u \in U$ and $a_0 \in A_0$ are nonadjacent. But now $u,x_6,a_0,x_1,u$ is a 4-hole in $G$, a contradiction. Thus, $U$ is complete to $V(B)$. This proves (1).
\end{proof}

Set $K = G[V(B) \cup U]$. Clearly, the 7-bracelet $B$ is anticonnected, and so (1) implies that $K$ has exactly one nontrivial anticomponent, and that this anticomponent is the 7-bracelet $B$. If $K = G$, then we are done. So assume that $V(K) \subsetneqq V(G)$. Set $C = V(G) \setminus V(K)$, and note that $C \neq \emptyset$. Set $B_C = G \setminus U$. (Thus, $B_C = [V(B) \cup C]$.) Our goal is to show that $B_C$ is a thickened emerald.

\begin{quote}
\emph{(2) For all $c \in C$, there exists some $i \in \mathbb{Z}_7$ such that $N_G(c) \cap V(H) = \{x_{i+2},x_{i+3},x_{i-3},x_{i-2}\}$.}
\end{quote}
\begin{proof}[Proof of (2)]
Fix $c \in C$. Since $H$ is a dominating hole in $G$, we know that $c$ is not anticomplete to $V(H)$. Since $c \notin U$, we know that $c$ is not complete to $V(H)$. Since $c \notin V(B)$, and since $\bigcup_{i \in \mathbb{Z}_7} T_i \subseteq V(B)$, we see that $c \notin \bigcup_{i \in \mathbb{Z}_7} T_i$. Lemma~\ref{lemma-7-hole-attachment-in-P7C4C5-free} now implies that there exists some $i \in \mathbb{Z}_7$ such that $N_G(c) \cap V(H) = \{x_{i+2},x_{i+3},x_{i-3},x_{i-2}\}$. This proves (2).
\end{proof}

\begin{quote}
\emph{(3) For all $c \in C$, there exists some $\ell \in \mathbb{Z}_7$ such that the following hold:
\begin{itemize}
\item $c$ is complete to $\{x_{\ell+2},x_{\ell+3},x_{\ell-3},x_{\ell-2}\}$ and anticomplete to $\{x_{\ell-1},x_{\ell},x_{\ell+1}\}$;
\item $G[V(B) \cup \{c\}]$ is a thickened emerald with good triple $(\{A_i\}_{i \in \mathbb{Z}_7},\{c\},\ell)$.
\end{itemize} }
\end{quote}
\begin{proof}[Proof of (3)]
Fix $c \in C$. By (2), and by symmetry, we may assume that $N_G(c) \cap V(H) = \{x_2,x_3,x_4,x_5\}$. Since $B$ is a maximal 7-bracelet in $G$, we know that $G[V(B) \cup \{c\}]$ is not a 7-bracelet. Lemma~\ref{lemma-add-vertex-to-7-bracelet} now implies that $G[V(B) \cup \{c\}]$ is a thickened emerald with good triple $(\{A_i\}_{i \in \mathbb{Z}_7},\{c\},0)$. This proves (3).
\end{proof}

\begin{quote}
\emph{(4) There exists some $\ell \in \mathbb{Z}_7$ such that $C$ is complete to $\{x_{\ell+2},x_{\ell+3},x_{\ell-3},x_{\ell-2}\}$ and anticomplete to $\{x_{\ell-1},x_{\ell},x_{\ell+1}\}$.}
\end{quote}
\begin{proof}[Proof of (4)]
Suppose otherwise. In view of (3), this implies that there exist distinct $c_1,c_2 \in C$ and distinct $\ell_1,\ell_2 \in \mathbb{Z}_7$ such that for each $j \in \{1,2\}$, $G[V(B) \cup \{c_j\}]$ is a thickened emerald with good triple $(\{A_i\}_{i \in \mathbb{Z}_7},\{c_j\},\ell_j)$. By the definition of a thickened emerald, this implies that $A_{\ell_1}^-,A_{\ell_1}^+,A_{\ell_2}^-,A_{\ell_2}^+ \neq \emptyset$. But this is impossible since by the definition of a 7-bracelet, there exists at most one index $i \in \mathbb{Z}_7$ such that $A_i^-,A_i^+ \neq \emptyset$. This proves (4).
\end{proof}

By (4), and by symmetry, we may assume that $C$ is complete to $\{x_2,x_3,x_4,x_5\}$ and anticomplete to $\{x_6,x_0,x_1\}$.

\begin{quote}
\emph{(5) $C$ is a nonempty clique.}
\end{quote}
\begin{proof}[Proof of (5)]
By assumption, $C$ is nonempty. Suppose that $C$ is not a clique, and fix distinct, nonadjacent vertices $c_1,c_2 \in C$. Then $c_1,x_2,c_2,x_4,c_1$ is a 4-hole in $G$, a contradiction. This proves (5).
\end{proof}

\begin{quote}
\emph{(6) $B_C$ is a thickened emerald with good triple $(\{A_i\}_{i \in \mathbb{Z}_7},C,0)$.}
\end{quote}
\begin{proof}[Proof (6)]
Since $C$ is complete to $\{x_2,x_3,x_4,x_5\}$ and anticomplete to $\{x_6,x_0,x_1\}$, (3) implies that for all $c \in C$, $G[V(B) \cup \{c\}]$ is a thickened emerald with good triple $(\{A_i\}_{i \in \mathbb{Z}_7},\{c\},0)$. By (5), $C$ is a nonempty clique, and it readily follows that $B_C$ is a thickened emerald with good triple $(\{A_i\}_{i \in \mathbb{Z}_7},C,0)$. This proves (6).
\end{proof}

\begin{quote}
\emph{(7) $U$ is complete to $V(B_C)$.}
\end{quote}
\begin{proof}[Proof of (7)]
By construction, $V(B_C) = V(B) \cup C$, and by (1), $U$ is complete to $V(B)$. It remains to show that $U$ is complete to $C$. Suppose otherwise, and fix nonadjacent vertices $u \in U$ and $c \in C$. But now $u,x_2,c,x_4,u$ is a 4-hole in $G$, a contradiction. This proves (7).
\end{proof}

Clearly, every thickened emerald is anticonnected. By construction, $V(G) = V(B_C) \cup U$ and $V(B_C) \cap U = \emptyset$; (1), (6), and (7) now imply that $B_C$ is the only nontrivial anticomponent of $G$, and that $B_C$ is a thickened emerald. This completes the argument.
\end{proof}

We are now ready to prove Theorem~\ref{thm-decomp-P7C4C5-free-with-C7}, restated below for the reader's convenience.

\begin{thm-decomp-P7C4C5-free-with-C7} Let $G$ be a graph. Then the following two statements are equivalent:
\begin{enumerate}[(i)]
\item \label{ref-contains-C7-no-clique-cut-P7C4C5-free} $G$ is a $(P_7,C_4,C_5)$-free graph that contains a 7-hole and does not admit a clique-cutset;
\item \label{ref-contains-C7-no-clique-cut-7-bracelet-emerald} $G$ contains exactly one nontrivial anticomponent, and this anticomponent is either a 7-bracelet or a thickened emerald.
\end{enumerate}
\end{thm-decomp-P7C4C5-free-with-C7}
\begin{proof}
Suppose first that $G$ satisfies (\ref{ref-contains-C7-no-clique-cut-7-bracelet-emerald}). By Lemma~\ref{lemma-7-bracelet-in-thickened-emerald}, every thickened emerald contains an induced 7-bracelet, and by Lemma~\ref{lemma-7-bracelet-7-hole}, every 7-bracelet contains a 7-hole; it follows that $G$ contains a 7-hole. Further, by Lemmas~\ref{lemma-one-anticomp} and~\ref{lemma-7BTE-P7C4C5-free}, $G$ is $(P_7,C_4,C_5)$-free and does not admit a clique-cutset. Thus, $G$ satisfies (\ref{ref-contains-C7-no-clique-cut-P7C4C5-free}).

Suppose now that $G$ satisfies (\ref{ref-contains-C7-no-clique-cut-P7C4C5-free}). Let $H^*$ be a 7-hole in $G$, chosen so that $|N_G[H^*]|$ is maximum. Set $K = G[N_G[H^*]]$. Clearly, $H^*$ is a dominating 7-hole in $K$, and so Lemma~\ref{lemma-dominating-7-hole-anticomp-7-bracelet-or-emerald} implies that $K$ contains exactly one nontrivial anticomponent (call it $B$), and this anticomponent is either a 7-bracelet or a thickened emerald. Set $U = V(G) \setminus V(B)$; then $U$ is a (possibly empty) clique, complete to $V(B)$ in $G$. If $K = G$, then we are done. So assume that $V(K) \subsetneqq V(G)$, and set $R = V(G) \setminus V(K)$. (Thus, $R \neq \emptyset$.)

If $B$ is a 7-bracelet, then let $(\{A_i\}_{i \in \mathbb{Z}_7},i^*)$ be a good pair for it, for all $i \in \mathbb{Z}_7$, let $A_i^*,A_i^+,A_i^-$ be as in the definition of a 7-bracelet, and set $C = \emptyset$. On the other hand, if $B$ is a thickened emerald, then let $(\{A_i\},i^*,C)$ be a good triple for it, and for all $i \in \mathbb{Z}_7$, let $A_i^*,A_i^+,A_i^-$ be as in the definition of a thickened emerald. By symmetry, we may assume that $i^* = 0$. Note that $C$ is a (possibly empty) clique, and that (by Lemma~\ref{lemma-7-bracelet-in-thickened-emerald}) $B \setminus C$ is a 7-bracelet with good pair $(\{A_i\}_{i \in \mathbb{Z}_7},i^*)$. By symmetry, we may assume that $i^* = 0$.

By the definition of a 7-bracelet, for all $i \in \mathbb{Z}_7$, some vertex in $A_i$ has a nonneighbor both in $A_{i-2}$ and in $A_{i+2}$. Let $x_0 \in A_0$ be such that $x_0$ has a nonneighbor both in $A_2$ and in $A_5$. If $A_2^* \neq \emptyset$, then let $x_2 \in A_2^*$, and otherwise, let $x_2 \in A_2$ be any nonneighbor of $x_0$. Similarly, if $A_5^* \neq \emptyset$, then let $x_5 \in A_5^*$, and otherwise, let $x_5 \in A_5$ be any nonneighbor of $x_0$. Fix nonadjacent $x_1 \in A_1$ and $x_6 \in A_6$. Finally, fix any $x_3 \in A_3$ and $x_4 \in A_4$. Set $H = G[x_0,x_1,\dots,x_6]$. Since $B \setminus C$ is a 7-bracelet with good pair $(\{A_i\}_{i \in \mathbb{Z}_7},0)$, we see that $H = x_0,x_1,\dots,x_6$ is a 7-hole. Clearly, $H$ is a dominating hole of $B$, and since $U$ is complete to $V(B)$, we see that $V(H)$ is complete to $U$. Thus, $V(K) \subseteq N_G[H]$. Now, by construction, $K = G[N_G[H^*]]$, and so by the choice of $H^*$, we have that $N_G[H] = V(K)$; consequently, $R$ is anticomplete to $H$.

\begin{quote}
\emph{(1) $C$ is complete to $\{x_2,x_3,x_4,x_5\}$ and anticomplete to $\{x_6,x_0,x_1\}$.}
\end{quote}
\begin{proof}[Proof of (1)]
If $C = \emptyset$, then this is immediate. So assume that $C \neq \emptyset$. Then $B$ is a thickened emerald with good triple $(\{A_i\}_{i \in \mathbb{Z}_7},C,0)$. In particular, $A_2^*,A_5^* \neq \emptyset$, and so $x_2 \in A_2^*$ and $x_5 \in A_5^*$. Furthermore, $C$ is complete to $A_2^* \cup A_3 \cup A_4 \cup A_5^*$ and anticomplete to $A_6 \cup A_0 \cup A_1$. Since $x_i \in A_i$ for all $i \in \mathbb{Z}_7$, and since $x_2 \in A_2^*$ and $x_5 \in A_5^*$, we see that $C$ is complete to $\{x_2,x_3,x_4,x_5\}$ and anticomplete to $\{x_6,x_0,x_1\}$. This proves (1).
\end{proof}

Let $D$ be the vertex set of a component of $G[R]$. Our goal is to show that $N_G(R) \cap V(K)$ is a clique; this is enough because (since $K$ is not a complete graph) it implies that $N_G(R) \cap V(K)$ is a clique-cutset of $G$, contrary to the fact that $G$ satisfies (i) and therefore does not admit a clique-cutset. If $D$ is anticomplete to $V(K)$, then we are done. So assume that $D$ is not anticomplete to $V(K)$.

\begin{quote}
\emph{(2) All vertices in $D$ have exactly the same neighbors in $C$.}
\end{quote}
\begin{proof}[Proof of (2)]
Suppose otherwise. Since $G[D]$ is connected, it follows that there exist adjacent vertices $d,d' \in D$ such that $N_G(d) \cap C \neq N_G(d') \cap C$. By symmetry, we may assume that there exists a vertex $c \in C$ such that $cd \in E(G)$ and $cd' \notin E(G)$. Now, recall that $R$ is anticomplete to $V(H)$; consequently, $d,d'$ are anticomplete to $V(H)$. But now by (1), we have that $d',d,c,x_5,x_6,x_0,x_1$ is an induced $P_7$ in $G$, a contradiction. This proves (2).
\end{proof}

\begin{quote}
\emph{(3) All vertices in $D$ have exactly the same neighbors in $V(B)$.}
\end{quote}
\begin{proof}[Proof of (3)]
In view of (2), it suffices to show that all vertices in $D$ have exactly the same neighbors in $\bigcup_{i \in \mathbb{Z}_7} A_i$. Suppose otherwise. Since $G[D]$ is connected, there exist adjacent vertices $d,d' \in D$ and an index $i \in \mathbb{Z}_7$ such that $N_G(d) \cap A_i \neq N_G(d') \cap A_i$; by symmetry, we may assume that there exists some $a_i \in A_i$ such that $da_i \in E(G)$ and $d'a_i \notin E(G)$. Recall that $R$ is anticomplete to $V(H)$; consequently, $d,d'$ are anticomplete to $V(H)$. Now, we know that $a_i \in A_i$ is anticomplete to at least one of $A_{i-2},A_{i+2}$, and in particular, $a_i$ is nonadjacent to at least one of $x_{i-2},x_{i+2}$. But if $a_ix_{i+2} \notin E(G)$, then $d',d,a_i,x_{i+1},x_{i+2},x_{i+3},x_{i-3}$ is an induced $P_7$ in $G$, and if $a_ix_{i-2} \notin E(G)$, then $d',d,a_i,x_{i-1},x_{i-2},x_{i-3},x_{i+3}$ is an induced $P_7$ in $G$, a contradiction in either case. This proves (3).
\end{proof}

\begin{quote}
\emph{(4) For all $d \in D$, $N_G(d) \cap V(B)$ is a clique.}
\end{quote}
\begin{proof}[Proof of (4)]
Suppose otherwise, and fix some $d \in D$ such that $N_G(d) \cap V(B)$ is not a clique. Fix nonadjacent $y_1,y_2 \in N_G(d) \cap V(B)$. Since $R$ is anticomplete to $V(H)$, we see that $y_1,y_2 \notin E(H)$. Now, set $Y = V(H) \cup \{y_1,y_2\}$. Note that every vertex in $V(B) \setminus V(H)$ has at least three neighbors in $V(H)$, and so we readily deduce that the distance between $y_1,y_2$ in $G[Y]$ is at most three. Let $P$ be a minimum-length induced path between $y_1,y_2$ in $G[Y]$. Then $P$ is of length at most three, and since $y_1y_2 \notin E(G)$, $P$ is of length at least two. Furthermore, every interior vertex of $P$ belongs to $V(H)$; consequently, $d$ is anticomplete to the interior of $P$. But now $V(P) \cup \{d\}$ induces a hole of length four or five in $G$, a contradiction. This proves (4).
\end{proof}

We are now ready to show that $N_G(D) \cap V(K)$ is a clique. First, $V(K) = V(B) \cup U$, and $U$ is a clique, complete to $V(B)$. Thus, it suffices to show that $N_G(D) \cap V(B)$ is a clique. Fix $d_0 \in D$. By (4), $N_G(d_0) \cap V(G_B)$ is a clique. By (3), $N_G(d_0) \cap V(B) = N_G(D) \cap V(B)$. Thus, $N_G(D) \cap V(B)$ is a clique. This completes the argument.
\end{proof}

\subsection{Proof of Theorem~\ref{thm-P7C4C5C7-free-contains-Theta-decomp}} \label{subsec:thm-P7C4C5C7-free-contains-Theta-decomp}

\begin{lemma} \label{lemma-lantern-P7C4C5-free} Let $R$ be a lantern. Then $R$ is $P_7$-free, every hole in $R$ is of length six, and $R$ does not admit a clique-cutset.
\end{lemma}
\begin{proof}
Let $(A,B_1,\dots,B_r,C_1,\dots,C_r,D)$, $r \geq 3$, be a good partition of the lantern $R$.

\begin{quote}
{(1) $R$ does not admit a clique-cutset.}
\end{quote}
\begin{proof}[Proof of (1)]
Let $S$ be a clique of $R$; we must show that $R \setminus S$ is connected. Since $S$ is a clique, and since $A$ is anticomplete to $D$, we know that $S$ intersects at most one of $A$ and $D$; by symmetry, we may assume that $S \cap A = \emptyset$. Then there exists some $i \in \{1,\dots,r\}$ such that $S \subseteq B_i \cup C_i \cup D$ (in fact, either $S \subseteq B_i \cup C_i$ or $S \subseteq C_i \cup D$). By symmetry, we may assume that $i \in \{1,2\}$. Now, clearly, $R \setminus (B_i \cup C_i \cup D)$ is connected. Further, $B_i \setminus S$ is complete to $A$, and $D \setminus S$ is complete to $C_3$, and so $R \setminus (S \cup C_i)$ is connected. But $S$ intersects at most one of $B_i$ and $D$, and we know that every vertex in $C_i$ has a neighbor in $B_i$, and that $C_i$ is complete to $D$; thus, every vertex in $C_i$ has a neighbor in $(B_i \cup D) \setminus S$, and we deduce that $R \setminus S$ is connected. This proves (1).
\end{proof}

\begin{quote}
{\em (2) $R$ is $P_7$-free.}
\end{quote}
\begin{proof}[Proof of (2)]
Suppose otherwise, and let $P = x_1,x_2,\dots,x_7$ be an induced $P_7$ in $R$.

Note that $R \setminus (A \cup D)$ is the disjoint union of cobipartite graphs; since $P_7$ is connected and not cobipartite, we see that $V(P)$ intersects at least one of $A$ and $D$. By symmetry, we may assume that $V(P) \cap A \neq \emptyset$. Since $A$ is complete to $\bigcup_{i=1}^r B_i$, and since $P$ is a path (and therefore every vertex in $P$ is of degree one or two), we deduce that there exist at most two indices $i \in \{1,\dots,r\}$ such that $V(P)$ intersects $B_i$. Furthermore, there exist at most two indices $i \in \{1,\dots,r\}$ such that $V(P)$ intersects $C_i$, for otherwise, either $P$ would be disconnected (if $V(P) \cap D = \emptyset$), or some vertex in $P$ would be of degree at least three (if $V(P) \cap D \neq \emptyset$), a contradiction in either case.

Now, our goal is to show that $|V(P) \cap (A \cup \bigcup_{i=1}^r B_i)| \leq 3$ and $|V(P) \cap (D \cup \bigcup_{i=1}^r C_i)| \leq 3$; this will contradict the fact that $|V(P)| = 7$.

First, since $A$ is a homogeneous set in $R$, and since $P$ does not admit a proper homogeneous set, we see that $|V(P) \cap A| \leq 1$; since $V(P) \cap A \neq \emptyset$, we deduce that $|V(P) \cap A| = 1$. Next, since $B_1,\dots,B_r$ are cliques, complete to $A$, and since $P$ is triangle-free, we see that $|V(P) \cap B_i| \leq 1$ for all $i \in \{1,\dots,r\}$. Since there are at most two indices $i \in \{1,\dots,r\}$ such that $V(P) \cap B_i \neq \emptyset$, we deduce that $|V(P) \cap (A \cup \bigcup_{i=1}^r B_i))| \leq 3$.

It remains to show that $|V(P) \cap (D \cup \bigcup_{i=1}^r C_i)| \leq 3$. If $V(P) \cap D \neq \emptyset$, then the argument is completely analogous to the one establishing that $|V(P) \cap (A \cup \bigcup_{i=1}^r B_i))| \leq 3$. So assume that $V(P) \cap D = \emptyset$; it now suffices to show that $|V(P) \cap (\bigcup_{i=1}^r C_i)| \leq 3$. Since $C_1,\dots,C_r$ are cliques, and since $P$ is triangle-free, we see that $|V(P) \cap C_i| \leq 2$ for all $i \in \{1,\dots,r\}$. Since there are at most two indices $i \in \{1,\dots,r\}$ such that $V(P) \cap C_i \neq \emptyset$, it now suffices to show that there is at most one index $i \in \{1,\dots,r\}$ such that $|V(P) \cap C_i| = 2$. Suppose otherwise. Then there exists some index $i \in \{2,\dots,r\}$ such that $|V(P) \cap C_i| = 2$; by symmetry, we may assume that $|V(P) \cap C_2| = 2$. But now if $V(P) \cap B_2 \neq \emptyset$, then $P$ contains a triangle, and otherwise, $P$ is disconnected, a contradiction in either case. This proves (2).
\end{proof}

\begin{quote}
{\em (3) $R \setminus A$ and $R \setminus D$ are both chordal.}
\end{quote}
\begin{proof}[Proof of (3)]
By symmetry, it suffices to show that $R \setminus A$ is chordal. Using the definition of a lantern, we set $B_1 = \{b_1^1,\dots,b_{|B_1|}^1\}$ and $C_1 = \{c_1^1,\dots,c_{|C_1|}^1\}$ so that $N_R[b_{|B_1|}^1] \cap C_1 \subseteq \dots \subseteq N_R[b_1^1] \cap C_1 = C_1$ and $N_R[c_{|C_1|}^1] \cap B_1 \subseteq \dots \subseteq N_R[c_1^1] \cap B_1 = B_1$. Set $B = \bigcup_{i=2}^r B_i$ and $C = \bigcup_{i=2}^r C_i$. Set $D = \{d_1,\dots,d_{|D|}\}$, and $B = \{b_1,\dots,b_{|B|}\}$ and $C = \{c_1,\dots,c_{|C|}\}$. Then the following is a simplicial elimination ordering for $R \setminus A$: $b_{|B_1|}^1,\dots,b_1^1,b_1,\dots,b_{|B|},c_1^1,\dots,c_{|C_1|}^1,c_1,\dots,c_{|C|},d_1,\dots,d_{|D|}$. Thus (by~\cite{FulkersonGrossSimplicialElimOrd}) $R \setminus A$ is chordal. This proves (3).
\end{proof}

\begin{quote}
{\em (4) All holes in $R$ are of length six.}
\end{quote}
\begin{proof}[Proof of (4)]
Let $k \geq 4$, and let $H = x_0,x_1,\dots,x_{k-1},x_0$ (with indices in $\mathbb{Z}_k$) be a hole in $R$. We claim that $k = 6$. First, by (3), $H$ intersects both $A$ and $D$. Clearly, the distance in $R$ between any vertex in $A$ and any vertex in $D$ is three, and in $C_4$ and $C_5$, the distance between any two vertices is at most two. Thus, $k \geq 6$. Now, $A$ and $D$ are homogeneous sets in $R$, and holes of length greater than four do not admit a proper homogeneous set; it follows that $|V(H) \cap A| = |V(H) \cap D| = 1$. Since $A$ is complete to $\bigcup_{i=1}^r B_i$, since every vertex of $H$ is of degree two in $H$, and since $V(H) \cap A \neq \emptyset$, we see that $|V(H) \cap (\bigcup_{i=1}^r B_i)| \leq 2$; similarly, $|V(H) \cap (\bigcup_{i=1}^r C_i)| \leq 2$. Thus, $k = |V(H)| = |V(H) \cap A|+|V(H) \cap D|+|V(H) \cap (\bigcup_{i=1}^r B_i)|+|V(H) \cap (\bigcup_{i=1}^r C_i)| \leq 1+1+2+2 = 6$. But we already showed that $k \geq 6$, and so $k = 6$. This proves (4).
\end{proof}

We are now done by (1), (2), and (4).
\end{proof}

A {\em thickened $\Theta_3^r$} is any graph obtained from $\Theta_3^r$ by blowing up each vertex to a nonempty clique of arbitrary size. In other words, a {\em thickened $\Theta_3^r$} ($r \geq 3$) is a graph $T$ whose vertex set can be partitioned into nonempty cliques $A,B_1,\dots,B_r,C_1,\dots,C_r,D$ such that all the following hold:
\begin{itemize}
\item $A$ is anticomplete to $D$;
\item $A$ is complete to $\bigcup_{i=1}^r B_i$ and anticomplete to $\bigcup_{i=1}^r C_i$;
\item $D$ is complete to $\bigcup_{i=1}^r C_i$ and anticomplete to $\bigcup_{i=1}^r B_i$;
\item for all $i \in \{1,\dots,r\}$, $B_i$ is complete to $C_i$;
\item for all distinct $i,j \in \{1,\dots,r\}$, $B_i \cup C_i$ is anticomplete to $B_j \cup C_j$.
\end{itemize}
Under these circumstances, we say that $(A,B_1,\dots,B_r,C_1,\dots,C_r,D)$ is a {\em good partition} for the thickened $\Theta_3^r$ $T$ (see Figure~\ref{fig:ThickenedTheta}).

Note that for an integer $r \geq 3$, $\Theta_3^r$ is a thickened $\Theta_3^r$, and every thickened $\Theta_3^r$ contains an induced $\Theta_3^r$. Furthermore, every thickened $\Theta_3^r$ is an $r$-lantern, and every $r$-lantern contains an induced $\Theta_3^r$.

\begin{figure}
\begin{center}
\includegraphics[scale=0.6]{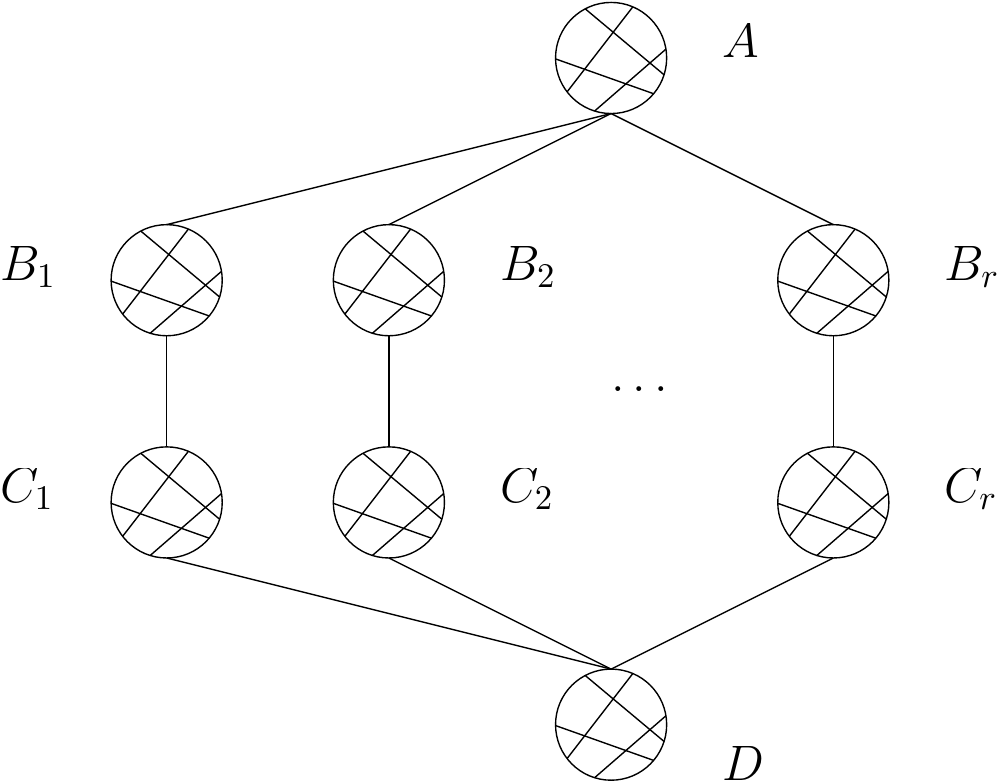}
\end{center}
\caption{Thickened $\Theta_3^r$, $r \geq 3$, with good partition $(A,B_1,\dots,B_r,C_1,\dots,C_r,D)$. A shaded disk represents a nonempty clique. A straight line between two cliques indicates that the two cliques are complete to each other. The absence of a line between two cliques indicates that the two cliques are anticomplete to each other.} \label{fig:ThickenedTheta}
\end{figure}

For a graph $G$, an integer $r \geq 3$, an induced subgraph $T$ of $G$ such that $T$ is a thickened $\Theta_3^r$ with good partition $\mathcal{P} = (A,B_1,\dots,B_r,C_1,\dots,C_r,D)$:
\begin{itemize}
\item $U^{G,T,\mathcal{P}}$ is the set of all vetices in $V(G) \setminus V(T)$ that are complete to $V(T)$;
\item $X^{G,T,\mathcal{P}}$ is the set of all vertices in $V(G) \setminus V(T)$ that are mixed on $V(T)$;
\item $Z^{G,T,\mathcal{P}}$ is the set of all vertices in $V(G) \setminus V(T)$ that are anticomplete to $V(T)$;
\item for all $i \in \{1,\dots,r\}$,
\begin{itemize}
\item $T_{B_i}^{G,T,\mathcal{P}}$ is the set of all vertices in $X^{G,T,\mathcal{P}}$ that are complete to $A \cup B_i$, mixed on $C_i$, and anticomplete to $V(T) \setminus (A \cup B_i \cup C_i)$,
\item $T_{C_i}^{G,T,\mathcal{P}}$ is the set of all vertices in $X^{G,T,\mathcal{P}}$ that are complete to $D \cup C_i$, mixed on $B_i$, and anticomplete to $V(T) \setminus (D \cup C_i \cup B_i)$,
\item $R_{B_i}^{G,T,\mathcal{P}}$ is the set of all vertices in $X^{G,T,\mathcal{P}}$ that are complete to $A$, have a neighbor in $B_i$, and are anticomplete to $V(T) \setminus (A \cup B_i)$,
\item $R_{C_i}^{G,T,\mathcal{P}}$ is the set of all vertices in $X^{G,T,\mathcal{P}}$ that are complete to $D$, have a neighbor in $C_i$, and are anticomplete to $V(T) \setminus (D \cup C_i)$.
\end{itemize}
\item $L_A^{G,T,\mathcal{P}}$ is the set of all vertices in $X$ that have a neighbor in $A$ and are anticomplete to $V(T) \setminus A$;
\item $L_D^{G,T,\mathcal{P}}$ is the set of all vertices in $X$ that have a neighbor in $D$ and are anticomplete to $V(T) \setminus D$.
\end{itemize}
When $G$, $T$, and $\mathcal{P}$ are clear from the context, we suppress the superscripts, and we write simply $U,X,Z,T_{B_i},T_{C_i},R_{B_i},R_{C_i},L_A,L_D$ instead of $U^{G,T,\mathcal{P}},X^{G,T,\mathcal{P}},Z^{G,T,\mathcal{P}},T_{B_i}^{G,T,\mathcal{P}},T_{C_i}^{G,T,\mathcal{P}},R_{B_i}^{G,T,\mathcal{P}},R_{C_i}^{G,T,\mathcal{P}},L_A^{G,T,\mathcal{P}},L_D^{G,T,\mathcal{P}}$, respectively. Clearly, $V(G) = V(T) \cup U \cup X \cup Z$, with $V(T)$, $U$, $X$, and $Z$ pairwise disjoint. Furthermore, it is clear that the sets $L_A,L_D,T_{B_1},\dots,T_{B_r},R_{B_1},\dots,R_{B_r},T_{C_1},\dots,T_{C_r},R_{C_1},\dots,R_{C_r}$ are pairwise disjoint. In what follows, we will repeatedly use the following two facts (which follow immediately from the construction):
\begin{itemize}
\item for all $i \in \{1,\dots,r\}$, every vertex in $T_{B_i} \cup R_{B_i}$ is complete to $A$, has a neighbor in $B_i$ and a nonneighbor in $C_i$, and is anticomplete to $V(T) \setminus (A \cup B_i \cup C_i)$;
\item for all $i \in \{1,\dots,r\}$, every vertex in $T_{C_i} \cup R_{C_i}$ is complete to $D$, has a neighbor in $C_i$ and a nonneighbor in $B_i$, and is anticomplete to $V(T) \setminus (D \cup C_i \cup B_i)$.
\end{itemize}
By construction, we have that $L_A \cup L_D \cup \bigcup_{i=1}^r (T_{B_i} \cup R_{B_i} \cup T_{C_i} \cup R_{C_i}) \subseteq X$. In Lemma~\ref{lemma-attachment-to-thickened-theta}, we show that if $G$ is $(P_7,C_4,C_5,C_7)$-free, and $T$ is an inclusion-wise maximal induced thickened $\Theta_3^r$ in $G$, then equality holds. In Lemma~\ref{lemma-anticomp-in-X}, we describe some additional properties of the sets defined above in the case when $G$ is $(P_7,C_4,C_5,C_7)$-free. Finally, we use Lemmas~\ref{lemma-one-anticomp},~\ref{lemma-lantern-P7C4C5-free},~\ref{lemma-attachment-to-thickened-theta}, and~\ref{lemma-anticomp-in-X} to prove Theorem~\ref{thm-P7C4C5C7-free-contains-Theta-decomp}.

\begin{lemma} \label{lemma-attachment-to-thickened-theta} Let $G$ be a $(P_7,C_4,C_5,C_7)$-free graph, and let $r \geq 3$ be an integer such that $G$ contains an induced $\Theta_3^r$. Let $T$ be an inclusion-wise maximal induced thickened $\Theta_3^r$ in $G$, and let $(A,B_1,\dots,B_r,C_1,\dots,C_r,D)$ be a good partition for $G$. Then $X = L_A \cup L_D \cup \bigcup_{i=1}^r (T_{B_i} \cup R_{B_i} \cup T_{C_i} \cup R_{C_i})$.
\end{lemma}
\begin{proof}
The fact that $L_A \cup L_D \cup \bigcup_{i=1}^r (T_{B_i} \cup R_{B_i} \cup T_{C_i} \cup R_{C_i}) \subseteq X$ is immediate from the construction. For the reverse inclusion, we fix some $x \in X$, and we show that $x \in L_A \cup L_D \cup \bigcup_{i=1}^r (T_{B_i} \cup R_{B_i} \cup T_{C_i} \cup R_{C_i})$.

\begin{quote}
{\em (1) Vertex $x$ is anticomplete to at least one of $A$ and $D$.}
\end{quote}
\begin{proof}[Proof of (1)]
Suppose otherwise, and fix $a \in A$ and $d \in D$ such that $xa,xd \in E(G)$. Our goal is to show that $x \in U$, contrary to the fact that $x \in X$.

First, we claim that $x$ is complete to $\bigcup_{i=1}^r (B_i \cup C_i)$. Suppose otherwise. By symmetry, we may assume that $x$ is nonadjacent to some $b_1' \in B_1$. Fix any $c_1 \in C_1$. But now if $xc_1 \in E(G)$, then $x,a,b_1',c_1,x$ is a 4-hole in $G$, and if $xc_1 \notin E(G)$, then $x,a,b_1',c_1,d,x$ is a 5-hole in $G$, a contradiction in either case. This proves that $x$ is indeed complete to $\bigcup_{i=1}^r (B_i \cup C_i)$.

Next, we claim that $x$ is complete to $A \cup D$. Suppose otherwise. By symmetry, we may assume that $x$ is nonadjacent to some $a' \in A$. Fix $b_1 \in B_1$ and $b_2 \in B_2$. By what we just showed, we have that $xb_1,xb_2 \in E(G)$, and it follows that $x,b_1,a',b_2,x$ is a 4-hole in $G$, a contradiction. Thus, $x$ is indeed complete to $A \cup D$.

We now have that $x$ is complete to $A \cup D \cup \bigcup_{i=1}^r (B_i \cup C_i) = V(T)$. Thus, $x \in U$, contrary to the fact that $x \in X$. This proves (1).
\end{proof}

\begin{quote}
{\em (2) Vertex $x$ has a neighbor in at most one of the sets $B_1,\dots,B_r$. Similarly, $x$ has a neighbor in at most one of the sets $C_1,\dots,C_r$.}
\end{quote}
\begin{proof}[Proof of (2)]
By symmetry, it suffices to prove the first statement. Suppose otherwise, that is, suppose that $x$ has a neighbor in at least two of $B_1,\dots,B_r$.

First, we claim that $x$ is complete to $A$. Suppose otherwise, and fix $a' \in A$ such that $xa' \notin E(G)$. Further, since $x$ has a neighbor in at least two of $B_1,\dots,B_r$, we may assume by symmetry that $x$ is adjacent to some $b_1 \in B_1$ and to some $b_2 \in B_2$. But now $x,b_1,a',b_2,x$ is a 4-hole in $G$, a contradiction. Thus, $x$ is complete to $A$, and by (1), it follows that $x$ is anticomplete to $D$.

Next, we claim that $x$ is anticomplete to $\bigcup_{i=1}^r C_i$. Suppose otherwise. By symmetry, we may assume that $x$ is adjacent to some $c_1 \in C_1$. By supposition, $x$ has a neighbor in at least two of $B_1,\dots,B_r$; consequently, $x$ has a neighbor in $\bigcup_{i=2}^r B_i$. By symmetry, we may now assume that $x$ is adjacent to some $b_2 \in B_2$. Fix $c_2 \in C_2$ and $d \in D$. Since $x$ is anticomplete to $D$, we know that $xd \notin E(G)$. But now if $xc_2 \in E(G)$, then $x,c_1,d,c_2,x$ is a 4-hole in $G$, and if $xc_2 \notin E(G)$, then $x,c_1,d,c_2,b_2,x$ is a 5-hole in $G$, a contradiction in either case. This proves that $x$ is anticomplete to $\bigcup_{i=1}^r C_i$.

We now have that $x$ is complete to $A$, has a neighbor in at least two of $B_1,\dots,B_r$, and is anticomplete to $D \cup \bigcup_{i=1}^r C_i$. If $x$ is complete to $\bigcup_{i=1}^r B_i$, then $G[V(T) \cup \{x\}]$ is a thickened $\Theta_3^r$ (we simply ``add'' $x$ to $A$), contrary to the maximality of $T$. Thus, $x$ is not complete to $\bigcup_{i=1}^r B_i$. But now we may assume by symmetry that $x$ is adjacent to some $b_1 \in B_1$ and $b_2 \in B_2$, and is nonadjacent to some $b_3' \in B_3$. Fix $c_2 \in C_2$, $c_3 \in C_3$, and $d \in D$. Then $b_1,x,b_2,c_2,d,c_3,b_3'$ is an induced $P_7$ in $G$, a contradiction. This proves (2).
\end{proof}

\begin{quote}
{\em (3) Vertex $x$ has a neighbor in at most one of the sets $B_1 \cup C_1,\dots,B_r \cup C_r$.}
\end{quote}
\begin{proof}[Proof of (3)]
Suppose otherwise. By (2), and by symmetry, we may assume that $x$ is adjacent to some $b_1 \in B_1$ and to some $c_2 \in C_2$, and that $x$ is anticomplete to $V(T) \setminus (B_1 \cup C_2)$. By (1), and by symmetry, we may assume that $x$ is anticomplete to $A$. Fix $a \in A$ and $b_2 \in B_2$. Then $x,b_1,a,b_2,c_2,x$ is a 5-hole in $G$, a contradiction. This proves (3).
\end{proof}

\begin{quote}
{\em (4) For all $i \in \{1,\dots,r\}$, if $x$ has a neighbor in $B_i$ and a nonneighbor in $C_i$, then $x$ is complete to $A$ and anticomplete to $V(T) \setminus (A \cup B_i \cup C_i)$. Similarly, for all $i \in \{1,\dots,r\}$, if $x$ has a neighbor in $C_i$ and a nonneighbor in $B_i$, then $x$ is complete to $D$ and anticomplete to $V(T) \setminus (D \cup C_i \cup B_i)$.}
\end{quote}
\begin{proof}[Proof of (4)]
By symmetry, it suffices to prove the first statement. In fact, by symmetry, it suffices to show that if $x$ has a neighbor in $B_1$ and a nonneighbor in $C_1$, then $x$ is complete to $A$ and anticomplete to $V(T) \setminus (A \cup B_1 \cup C_1)$. So assume that $x$ is adjacent to some $b_1 \in B_1$ and nonadjacent to some $c_1' \in C_1$. By (3), $x$ is anticomplete to $\bigcup_{i=2}^r (B_i \cup C_i)$. Further, $x$ is anticomplete to $D$, for otherwise, we fix some $d \in D$ such that $xd \in E(G)$, and we observe that $x,b_1,c_1',d,x$ is a 4-hole in $G$, a contradiction. We have now shown that $x$ is anticomplete to $D \cup \bigcup_{i=2}^r (B_i \cup C_i) = V(T) \setminus (A \cup B_1 \cup C_1)$. It remains to show that $x$ is complete to $A$. Suppose otherwise, and fix some $a' \in A$ such that $xa' \notin E(G)$. Further, fix $b_2 \in B_2$, $c_2 \in C_2$, $d \in D$, $c_3 \in C_3$. Now $x,b_1,a',b_2,c_2,d,c_3$ is an induced $P_7$ in $G$, a contradiction. This proves (4).
\end{proof}

\begin{quote}
{\em (5) For all $i \in \{1,\dots,r\}$, if $x$ has a neighbor both in $B_i$ and in $C_i$, then $x \in T_{B_i} \cup T_{C_i}$.}
\end{quote}
\begin{proof}[Proof of (5)]
By symmetry, it suffices to show that if $x$ has a neighbor both in $B_1$ and in $C_1$, then $x \in T_{B_1} \cup T_{C_1}$. So assume that $x$ has a neighbor both in $B_1$ and in $C_1$. By (3), $x$ is anticomplete to $\bigcup_{i=2}^r (B_i \cup C_i)$.

Suppose first that $x$ is complete to $B_1 \cup C_1$. By (1), $x$ is anticomplete to at least one of $A$ and $D$; by symmetry, we may assume that $x$ is anticomplete to $D$. Suppose first that $x$ has a nonneighbor $a' \in A$. Now, fix $b_1 \in B_1$, $b_2 \in B_2$, $c_2 \in C_2$, $c_3 \in C_3$, and $d \in D$. Then $x,b_1,a',b_2,c_2,d,c_3$ is an induced $P_7$ in $G$, a contradiction. Thus, $x$ is complete to $A$. But now $G[V(T) \cup \{x\}]$ contradicts the maximality of $T$ (we simply ``add'' $x$ to $B_1$).

Thus, $x$ is not complete to $B_1 \cup C_1$. If $x$ is mixed on both $B_1$ and $C_1$, then (4) implies that $x$ is complete to both $A$ and $D$, contrary to (1). We now have that $c$ has a neighbor both in $B_1$ and in $C_1$, that $x$ is not complete to $B_1 \cup C_1$, and that $x$ is mixed on at most one of $B_1$ and $C_1$. This implies that $x$ is complete to one of $B_1,C_1$ and is mixed on the other; by symmetry, we may assume that $x$ is complete to $B_1$ and mixed on $C_1$. Now (4) implies that $x$ is complete to $A$ and anticomplete to $V(T) \setminus (A \cup B_1 \cup C_1)$, and we deduce that $x \in T_{B_1}$. This proves (5).
\end{proof}

\begin{quote}
{\em (6) $x \in L_A \cup L_D \cup \bigcup_{i=1}^r (T_{B_i} \cup R_{B_i} \cup T_{C_i} \cup R_{C_i})$.}
\end{quote}
\begin{proof}[Proof of (6)]
Suppose first that $x$ is anticomplete to $\bigcup_{i=1}^r (B_i \cup C_i)$. By (1), $x$ is anticomplete to at least one of $A$ and $D$; by symmetry, we may assume that $x$ is anticomplete to $D$. Thus, $x$ is anticomplete to $V(T) \setminus A$, that is, all neighbors of $x$ in $V(T)$ belong to $A$. Since $x \in X$, we know that $x$ has at least one neighbor in $V(T)$, and we deduce that $x \in L_A$.

From now on, we assume that $x$ has a neighbor in $\bigcup_{i=1}^r (B_i \cup C_i)$. By symmetry, we may assume that $x$ has a neighbor in $B_1 \cup C_1$. If $x$ has a neighbor both in $B_1$ and in $C_1$, then by (5), $x \in T_{B_1} \cup T_{C_1}$, and we are done. By symmetry, we may now assume that $x$ has a neighbor $b_1 \in B_1$ and is anticomplete to $C_1$. We claim that $x \in R_{B_1}$.

By (3), $x$ is anticomplete to $\bigcup_{i=2}^r (B_i \cup C_i)$. Next, if $x$ is adjacent to some $d \in D$, then we fix some $c_1 \in C_1$, and we observe that $x,b_1,c_1,d,x$ is a 4-hole, a contradiction. Thus, $x$ is anticomplete to $D$.

We now have that $x$ has a neighbor in $B_1$ and is anticomplete to $V(T) \setminus (A \cup B_1)$. It remains to show that $x$ is complete to $A$. Suppose otherwise, and fix some $a' \in A$ such that $xa' \notin E(G)$. Fix some $b_2 \in B_2$, $c_2 \in C_2$, $d \in D$, and $c_3 \in C_3$. But now $x,b_1,a',b_2,c_2,d,c_3$ is an induced $P_7$ in $G$, a contradiction. Thus, $x$ is complete to $A$, and it follows that $x \in R_{B_1}$. This proves (6).
\end{proof}

By (6), we have that $X \subseteq L_A \cup L_D \cup \bigcup_{i=1}^r (T_{B_i} \cup R_{B_i} \cup T_{C_i} \cup R_{C_i})$, and we are done.
\end{proof}

\begin{lemma} \label{lemma-anticomp-in-X} Let $G$ be a $(P_7,C_4,C_5,C_7)$-free graph, and let $r \geq 3$ be an integer such that $G$ contains an induced $\Theta_3^r$. Let $T$ be an induced thickened $\Theta_3^r$ in $G$, and let $(A,B_1,\dots,B_r,C_1,\dots,C_r,D)$ be a good partition for $T$. Then all the following hold:
\begin{enumerate}[(a)]
\item \label{ref-anticomp-in-X-LA-anticomp-LD-or-Theta-r+1} either $L_A$ is anticomplete to $L_D$, or $G$ contains an induced $\Theta_3^{r+1}$.
\item \label{ref-anticomp-in-X-LA-anticomp-toTRCi} $L_A$ is anticomplete to $\bigcup_{i=1}^r (T_{C_i} \cup R_{C_i})$, and $L_D$ is anticomplete to $\bigcup_{i=1}^r (T_{B_i} \cup R_{B_i})$;
\item \label{ref-anticomp-in-X-TRBi-anticomp-TRCi} $\bigcup_{i=1}^r (T_{B_i} \cup R_{B_i})$ is anticomplete to $\bigcup_{i=1}^r (T_{C_i} \cup R_{C_i})$;
\item \label{ref-anticomp-in-X-at-most-one-TRBi} at most one of the sets $T_{B_1} \cup R_{B_1},\dots,T_{B_r} \cup R_{B_r}$ is nonempty, and at most one of the sets $T_{C_1} \cup R_{C_1},\dots,T_{C_r} \cup R_{C_r}$ is nonempty;
\item \label{ref-anticomp-in-X-at-most-one-TBiCi} at most one of the sets $T_{B_1} \cup T_{C_1},\dots,T_{B_r} \cup T_{C_r}$ is nonempty;
\item \label{ref-anticomp-in-X-TBi-clique} $T_{B_1},\dots,T_{B_r},T_{C_1},\dots,T_{C_r}$ are (possibly empty) cliques;
\item \label{ref-anticomp-in-X-U-clique-comp-to-TBi} $U$ is a (possibly empty) clique, complete to $\bigcup_{i=1}^r (T_{B_i} \cup T_{C_i})$;
\item \label{ref-anticomp-in-X-Z-anticomp-X} $Z$ is anticomplete to $L_A \cup L_D \cup \bigcup_{i=1}^r (T_{B_i} \cup R_{B_i} \cup T_{C_i} \cup R_{C_i})$.
\end{enumerate}
\end{lemma}
\begin{proof}
We first prove (\ref{ref-anticomp-in-X-LA-anticomp-LD-or-Theta-r+1}). Suppose that $L_A$ is not anticomplete to $L_D$, and fix adjacent vertices $x \in L_A$ and $y \in L_D$. Fix $a \in A$ and $d \in D$ such that $xa,yd \in E(G)$. Further, for all $i \in \{1,\dots,r\}$, fix $b_i \in B_i$ and $c_i \in C_i$. Then $G[a,b_1,\dots,b_r,x,c_1,\dots,c_r,y,d]$ is a $\Theta_3^{r+1}$. This proves (\ref{ref-anticomp-in-X-LA-anticomp-LD-or-Theta-r+1}).

Next, we prove (\ref{ref-anticomp-in-X-LA-anticomp-toTRCi}). By symmetry, it suffices to show that $L_A$ is anticomplete to $T_{C_1} \cup R_{C_1}$. Suppose otheriwse, and fix adjacent vertices $a' \in L_A$ and $x \in T_{C_1} \cup R_{C_1}$. By construction, $a'$ has a neighbor $a \in A$ and is anticomplete to $V(T) \setminus A$, and $x$ is complete to $D$, and has a neighbor $c_1 \in C_1$ and a nonneighbor $b_1' \in B_1$. But now $a',a,b_1',c_1,x,a'$ is a 5-hole in $G$, a contradiction. This proves (\ref{ref-anticomp-in-X-LA-anticomp-toTRCi}).

We next prove (\ref{ref-anticomp-in-X-TRBi-anticomp-TRCi}). By symmetry, it suffices to show that $T_{B_1} \cup R_{B_1}$ is anticomplete to $T_{C_1} \cup R_{C_1} \cup T_{C_2} \cup R_{C_2}$.

We first show that $T_{B_1} \cup R_{B_1}$ is anticomplete to $T_{C_1} \cup R_{C_1}$. Suppose otherwise, and fix adjacent vertices $x \in T_{B_1} \cup R_{B_1}$ and $y \in T_{C_1} \cup R_{C_1}$. We claim that one of $N_G[x] \cap (B_1 \cup C_1)$ and $N_G[y] \cap (B_1 \cup C_1)$ includes the other. Suppose otherwise, and fix $v_x,v_y \in B_1 \cup C_1$ such that $xv_x,yv_y \in E(G)$ and $xv_y,yv_x \notin E(G)$. Since $B_1 \cup C_1$ is a clique, we now have that $x,v_x,v_y,y,x$ is a 4-hole in $G$, a contradiction. By symmetry, we may now assume that $N_G[y] \cap (B_1 \cup C_1) \subseteq N_G[x] \cap (B_1 \cup C_1)$. Since $x \in T_{B_1} \cup R_{B_1}$, we know that $x$ is adjacent to some $b_1 \in B_1$ and nonadjacent to some $c_1 \in C_1$. Since $N_G[y] \cap (B_1 \cup C_1) \subseteq N_G[x] \cap (B_1 \cup C_1)$, it follows that $yc_1 \notin E(G)$, and consequently, $y \notin T_{C_1}$. Thus, $y \in R_{C_1}$, and in particular, $yb_1 \notin E(G)$. We now fix $d \in D$, and we observe that $x,b_1,c_1,d,y,x$ is a 5-hole in $G$, a contradiction. This proves that $T_{B_1} \cup R_{B_1}$ is anticomplete to $T_{C_1} \cup R_{C_1}$.

We now show that $T_{B_1} \cup R_{B_1}$ is anticomplete to $T_{C_2} \cup R_{C_2}$. Suppose otherwise, and fix adjacent vertices $x \in T_{B_1} \cup R_{B_1}$ and $y \in T_{C_2} \cup R_{C_2}$. Fix $b_1 \in B_1$ and $c_1 \in C_1$ such that $xb_1 \in E(G)$ and $xc_1 \notin E(G)$. Fix $d \in D$. Now $x,b_1,c_1,d,y,x$ is a 5-hole in $G$, a contradiction. This proves (\ref{ref-anticomp-in-X-TRBi-anticomp-TRCi}).

Next, suppose that (\ref{ref-anticomp-in-X-at-most-one-TRBi}) is false. By symmetry, we may assume $T_{B_1} \cup R_{B_1}$ and $T_{B_2} \cup R_{B_2}$ are both nonempty. Fix $x_1 \in T_{B_1} \cup R_{B_1}$ and $x_2 \in T_{B_2} \cup R_{B_2}$. For each $i \in \{1,2\}$, fix $b_i \in B_i$ and $c_i \in C_i$ such that $x$ is adjacent to $b_i$ and nonadjacent to $c_i$. Further, fix $d \in D$; then $x_1,x_2$ are nonadjacent to $d$. But now if $x_1x_2 \in E(G)$, then $x_1,b_1,c_1,d,c_2,b_2,x_2,x_1$ is a 7-hole in $G$, and otherwise, $x_1,b_1,c_1,d,c_2,b_2,x_2$ is an induced $P_7$ in $G$, a contradiction in either case. This proves (\ref{ref-anticomp-in-X-at-most-one-TRBi}).

Next, suppose that (\ref{ref-anticomp-in-X-at-most-one-TBiCi}) is false. By (\ref{ref-anticomp-in-X-at-most-one-TRBi}), and by symmetry, we may assume that $T_{B_1}$ and $T_{C_2}$ are both nonempty. Fix $x \in T_{B_1}$ and $y \in T_{C_2}$. By (\ref{ref-anticomp-in-X-TRBi-anticomp-TRCi}), $xy \notin E(G)$. Since $x \in T_{B_1}$, we know that $x$ has a neighbor $b_1 \in B_1$ and a nonneighbor $c_1 \in C_1$. Since $y \in T_{C_2}$, we know that $y$ is mixed on $B_2$; fix $b_2,b_2' \in B_2$ such that $yb_2 \in E(G)$ and $yb_2' \notin E(G)$. Finally, fix $d \in D$. Now $x,b_1,c_1,d,y,b_2,b_2'$ is an induced $P_7$ in $G$, a contradiction. This proves (\ref{ref-anticomp-in-X-at-most-one-TBiCi}).

Next, we prove (\ref{ref-anticomp-in-X-TBi-clique}). By symmetry, it suffices to show that $T_{B_1}$ is a clique. Suppose otherwise, and fix nonadjacent $x,y \in T_{B_1}$. Suppose that neither of $N_G[x] \cap C_1$ and $N_G[y] \cap C_1$ is included in the other. Fix $c_x,c_y \in C_1$ such that $xc_x,yc_y \in E(G)$ and $xc_y,yc_x \notin E(G)$. Fix $a \in A$. Then $a,x,c_x,c_y,y,a$ is a 5-hole in $G$, a contradiction. By symmetry, we may now assume that $N_G[y] \cap C_1 \subseteq N_G[x] \cap C_1$. Since $y \in T_{B_1}$, we know that $y$ has a neighbor $c_1 \in C_1$; since $N_G[y] \cap C_1 \subseteq N_G[x] \cap C_1$, it follows that $xc_1 \in E(G)$. Fix $a \in A$. Now $a,x,c_1,y,a$ is a 4-hole in $G$, a contradiction. This proves (\ref{ref-anticomp-in-X-TBi-clique}).

We now prove (\ref{ref-anticomp-in-X-U-clique-comp-to-TBi}). Suppose that $U$ is not a clique, and fix nonadjacent vertices $u_1,u_2 \in U$. Fix $a \in A$ and $d \in D$. Then $a,u_1,d,u_2,a$ is a 4-hole in $G$, a contradiction. Thus, $U$ is a clique. It remains to show that $U$ is complete to $\bigcup_{i=1}^r (T_{B_i} \cup T_{C_i})$. Suppose otherwise. By symmetry, we may assume that some $u \in U$ and $x \in T_{B_1}$ are nonadjacent. Fix $a \in A$, and fix $c_1 \in C_1$ such that $xc_1 \in E(G)$. Then $a,x,c_1,u,a$ is a 4-hole in $G$, a contradiction.

It remains to prove (\ref{ref-anticomp-in-X-Z-anticomp-X}). Suppose that some $z \in Z$ has a neighbor in $L_A \cup L_D \cup \bigcup_{i=1}^r (T_{B_i} \cup R_{B_i} \cup T_{C_i} \cup R_{C_i})$. By symmetry, we may assume that $z$ is adjacent to some $x \in L_A \cup T_{B_1} \cup R_{B_1}$. Then $x$ has a neighbor $a \in A$, has a nonneighbor $c_1 \in C_1$, and is anticomplete to $V(T) \setminus (A \cup B_1 \cup C_1)$. Fix $b_2 \in B_2$, $c_2 \in C_2$, and $d \in D$. Now $z,x,a,b_2,c_2,d,c_1$ is an induced $P_7$ in $G$, a contradiction. This proves (\ref{ref-anticomp-in-X-Z-anticomp-X}).
\end{proof}

We are now ready to prove Theorem~\ref{thm-P7C4C5C7-free-contains-Theta-decomp}, restated below for the reader's convenience.

\begin{thm-P7C4C5C7-free-contains-Theta-decomp} Let $G$ be a graph. Then the following two statements are equivalent:
\begin{enumerate}[(i)]
\item $G$ is a $(P_7,C_4,C_5,C_7)$-free graph that contains an induced $\Theta_3^3$ and does not admit a clique-cutset;
\item $G$ contains exactly one nontrivial anticomponent, and this anticomponent is a lantern.
\end{enumerate}
\end{thm-P7C4C5C7-free-contains-Theta-decomp}
\begin{proof}
The fact that (ii) implies (i) follows from Lemmas~\ref{lemma-one-anticomp} and~\ref{lemma-lantern-P7C4C5-free}, and from the fact that every lantern contains an induced $\Theta_3^3$. It remains to show that (i) implies (ii). For this, we assume that $G$ satisfies (i), that is, that $(P_7,C_4,C_5,C_7)$-free, contains an induced $\Theta_3^3$, and does not admit a clique-cutset; we must show that $G$ contains exactly one nontrivial anticomponent, and that this anticomponent is a lantern.

Let $r \geq 3$ be maximal with the property that $G$ contains an induced $\Theta_3^r$. Let $T$ be an inclusion-wise maximal thickened $\Theta_3^r$ in $G$, and let $(A,B_1,\dots,B_r,C_1,\dots,C_r,D)$ be a good partition for $T$. By construction, we have that $V(G) = V(T) \cup U \cup X \cup Z$, and by Lemma~\ref{lemma-attachment-to-thickened-theta}, we have that $X = L_A \cup L_D \cup \bigcup_{i=1}^r (T_{B_i} \cup R_{B_i} \cup T_{C_i} \cup R_{C_i})$.

\begin{quote}
{\em (1) $Z \cup L_A \cup L_D \cup \bigcup_{i=1}^r (R_{B_i} \cup R_{C_i}) = \emptyset$.}
\end{quote}
\begin{proof}[Proof of (1)]
Suppose otherwise; our goal is to show that $G$ admits a clique-cutset, contrary to the fact that $G$ satisfies (i). Since $Z \cup L_A \cup L_D \cup \bigcup_{i=1}^r (R_{B_i} \cup R_{C_i}) \neq \emptyset$, we see that at least one of $Z \cup L_A \cup \bigcup_{i=1}^r R_{B_i}$ and $Z \cup L_D \cup \bigcup_{i=1}^r R_{C_i}$ is nonempty. By symmetry, we may assume that $Z \cup L_A \cup \bigcup_{i=1}^r R_{B_i} \neq \emptyset$. By Lemma~\ref{lemma-anticomp-in-X}~(\ref{ref-anticomp-in-X-at-most-one-TRBi}), we know that at most one of the sets $T_{B_1} \cup R_{B_1},\dots,T_{B_r} \cup R_{B_r}$ is nonempty; by symmetry, we may assume that $\bigcup_{i=2}^r (T_{B_r} \cup R_{B_r}) = \emptyset$ (the set $T_{B_1} \cup R_{B_1}$ might or might not be empty). Our goal is to show that $U \cup A \cup B_1 \cup T_{B_1}$ is a clique-cutset of $G$. The fact that $U \cup A \cup B_1 \cup T_{B_1}$ is a clique follows from the construction, and from Lemma~\ref{lemma-anticomp-in-X}~(\ref{ref-anticomp-in-X-TBi-clique}) and~(\ref{ref-anticomp-in-X-U-clique-comp-to-TBi}). It remains to show that $U \cup A \cup B_1 \cup T_{B_1}$ is a cutset of $G$.

Clearly, it suffices to show that $Z \cup L_A \cup R_{B_1}$ is anticomplete to $V(T) \setminus (A \cup B_1)$ and to $X \setminus (L_A \cup T_{B_1} \cup R_{B_1})$. The former is immediate from the construction. For the latter, recall that $X = L_A \cup L_D \cup \bigcup_{i=1}^r (T_{B_i} \cup R_{B_i} \cup T_{C_i} \cup R_{C_i})$ and that $\bigcup_{i=2}^r (T_{B_i} \cup R_{B_i}) = \emptyset$. Thus, it suffices to show that $Z \cup L_A \cup R_{B_1}$ is anticomplete to $L_D \cup \bigcup_{i=1}^r (T_{C_i} \cup R_{C_i})$. The fact that $L_A$ is anticomplete to $L_D$ follows from Lemma~\ref{lemma-anticomp-in-X}~(\ref{ref-anticomp-in-X-LA-anticomp-LD-or-Theta-r+1}), and from the maximality of $r$. The rest follows from Lemma~\ref{lemma-anticomp-in-X}~(\ref{ref-anticomp-in-X-LA-anticomp-toTRCi}),~(\ref{ref-anticomp-in-X-TRBi-anticomp-TRCi}), and~(\ref{ref-anticomp-in-X-Z-anticomp-X}). This proves (1).
\end{proof}

\begin{quote}
{\em (2) There exists some $i \in \{1,\dots,r\}$ such that $V(G) = V(T) \cup U \cup T_{B_i} \cup T_{C_i}$.}
\end{quote}
\begin{proof}[Proof of (2)]
Recall that $V(G) = V(T) \cup U \cup X \cup Z$ and $X = L_A \cup L_D \cup \bigcup_{i=1}^r (T_{B_i} \cup R_{B_i} \cup T_{C_i} \cup R_{C_i})$. It then follows from (1) that $V(G) = V(T) \cup U \cup \bigcup_{i=1}^r (T_{B_i} \cup T_{C_i})$. By Lemma~\ref{lemma-anticomp-in-X}~(\ref{ref-anticomp-in-X-at-most-one-TBiCi}), and by symmetry, we may assume that $\bigcup_{i=2}^r (T_{B_i} \cup T_{C_i}) = \emptyset$. But now $V(G) = V(T) \cup U \cup T_{B_1} \cup T_{C_1}$. This proves (2).
\end{proof}

By (2), and by symmetry, we may now assume that $V(G) = V(T) \cup U \cup T_{B_1} \cup T_{C_1}$. Set $H = G \setminus U$. By construction, and by Lemma~\ref{lemma-anticomp-in-X}~(\ref{ref-anticomp-in-X-U-clique-comp-to-TBi}), $U$ is a (possibly empty) clique, complete to $V(H)$. Furthermore, since $V(H) = V(T) \cup T_{B_1} \cup T_{C_1}$, it is easy to see that $H$ is anticonnected. Thus, $H$ is the only anticomponent of $G$. It remains to show that $H$ is a lantern. Set $B_1' = B_1 \cup T_{B_1}$ and $C_1' = C_1 \cup T_{C_1}$. Further, for all $i \in \{2,\dots,r\}$, set $B_i' = B_i$ and $C_i' = C_i$. We claim that $H$ is an $r$-lantern with good partition $(A,B_1',\dots,B_r',C_1',\dots,C_r',D)$. Lemma~\ref{lemma-anticomp-in-X}~(\ref{ref-anticomp-in-X-TBi-clique}) implies that $B_1'$ and $C_1'$ are cliques. It now suffices to show that $B_1'$ and $C_1'$ can be ordered as $B_1' = \{b_1^1,\dots,b_{|B_1'|}^1\}$ and $C_1' = \{c_1^1,\dots,c_{|C_1'|}^1\}$ so that $N_H[b_{|B_1'|}^1] \cap C_1' \subseteq \dots \subseteq N_H[b_1^1] \cap C_1' = C_1'$ and $N_H[c_{|C_1'|}^1] \cap B_1' \subseteq \dots \subseteq N_H[c_1^1] \cap B_1' = B_1'$ (all other axioms from the definition of an $r$-lantern are clearly satisfied).

\begin{quote}
{\em (3) For all distinct $x,y \in B_1'$, one of $N_H[x] \cap C_1'$ and $N_H[y] \cap C_1'$ includes the other. Similarly, for all distinct $x,y \in C_1'$, one of $N_H[x] \cap B_1'$ and $N_H[y] \cap B_1'$ includes the other.}
\end{quote}
\begin{proof}[Proof of (3)]
By symmetry, it suffices to prove the first statement. Suppose otherwise, and fix distinct $x,y \in B_1'$ such that neither one of $N_H[x] \cap C_1'$ and $N_H[y] \cap C_1'$ includes the other. Then there exist distinct vertices $c_x,c_y \in C_1'$ such that $xc_x,yc_y \in E(G)$ and $xc_y,yc_x \notin E(G)$. Since $B_1'$ and $C_1'$ are cliques, we know that $xy,c_xc_y \in E(G)$. But now $x,c_x,c_y,y,x$ is a 4-hole in $G$, a contradiction. This proves (3).
\end{proof}

It follows from (3) that $B_1'$ and $C_1'$ can be ordered as $B_1' = \{b_1^1,\dots,b_{|B_1'|}^1\}$ and $C_1' = \{c_1^1,\dots,c_{|C_1'|}^1\}$ so that $N_H[b_{|B_1'|}^1] \cap C_1' \subseteq \dots \subseteq N_H[b_1^1] \cap C_1'$ and $N_H[c_{|C_1'|}^1] \cap B_1' \subseteq \dots \subseteq N_H[c_1^1] \cap B_1'$. It remains to show that $b_1^1$ is complete to $C_1'$, and that $c_1^1$ is complete to $B_1'$. By symmetry, it suffices to prove the former. Suppose otherwise, and fix an index $j \in \{1,\dots,|C_1'|\}$ such that $b_1^1c_j^1 \notin E(G)$. Since $N_H[b_{|B_1'|}^1] \cap C_1' \subseteq \dots \subseteq N_H[b_1^1] \cap C_1'$, it follows that $c_j^1$ is anticomplete to $B_1'$. Since $B_1 \subseteq B_1'$, it follows that $c_j^1$ is anticomplete to $B_1$. But this is impossible since $c_j^1 \in C_1' = C_1 \cup T_{C_1}$, $C_1$ is complete to $B_1$, every vertex in $T_{C_1}$ has a neighbor in $B_1$, and $B_1 \neq \emptyset$. This completes the argument.
\end{proof}

\subsection{Proof of Theorem~\ref{thm-P7C4C5C7Theta33-free-decomp}} \label{subsec:thm-P7C4C5C7Theta33-free-decomp}

A {\em theta} is any subdivision of the complete bipartite graph $K_{2,3}$; in particular, $K_{2,3}$ is a theta. A {\em pyramid} is any subdivision of the complete graph $K_4$ in which one triangle remains unsubdivided, and of the remaining three edges, at least two edges are subdivided at least once. A {\em prism} is any subdivision of $\overline{C_6}$ in which the two triangles remain unsubdivided; in particular, $\overline{C_6}$ is a prism. A {\em three-path-configuration}, or {\em 3PC} for short, is any theta, pyramid, or prism (see Figure~\ref{fig:3PC}).

\begin{figure}
\begin{center}
\includegraphics[scale=0.6]{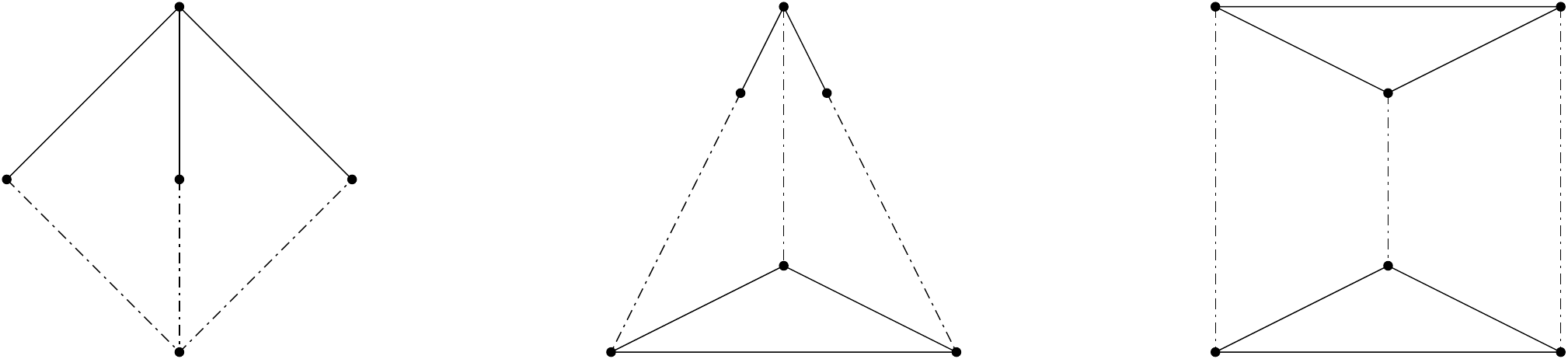}
\end{center}
\caption{Three types of 3PC: theta (left), pyramid (center), and prism (right). A solid line represents an edge, and a dashed line represents a path that has at least one edge.} \label{fig:3PC}
\end{figure}

Note that all holes in a $(P_7,C_4,C_5,C_7)$-free graph are of length six. It is easy to see that every theta other than $\Theta_3^3$ contains at least one hole of length other than six. On the other hand, $\Theta_3^3$ is a lantern, and so Lemma~\ref{lemma-lantern-P7C4C5-free} implies that $\Theta_3^3$ is $(P_7,C_4,C_5,C_7)$-free. Thus, $\Theta_3^3$ is the only $(P_7,C_4,C_5,C_7)$-free theta. It is also easy to see that the only prism in which all holes are of length six is the one obtained from $\overline{C_6}$ by subdividing each of the three edges that do not belong to any triangle exactly once; however, this prism contains an induced $P_7$ (see Figure~\ref{fig:PrismP7}). Finally, it is clear that every pyramid contains an odd hole, and consequently, no pyramid is $(P_7,C_4,C_5,C_7)$-free. It follows that $\Theta_3^3$ is the only $(P_7,C_4,C_5,C_7)$-free 3PC.

\begin{figure}
\begin{center}
\includegraphics[scale=0.6]{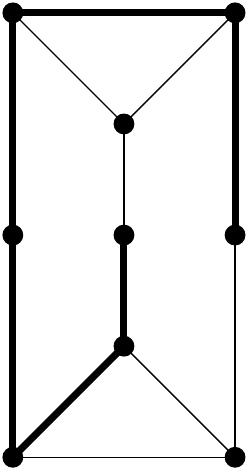}
\end{center}
\caption{The only prism in which all holes are of length six. An induced $P_7$ is shown with a bold line.} \label{fig:PrismP7}
\end{figure}

A {\em wheel} is a graph that consists of a hole and an additional vertex that has at least three neighbors in the hole. If this additional vertex is adjacent to all vertices of the hole, then the wheel is said to be a {\em universal wheel}; if the additional vertex is adjacent to three consecutive vertices of the hole, and to no other vertices of the hole, then the wheel is said to be a {\em twin wheel}. For $k \geq 4$, the universal wheel on $k+1$ vertices is denoted by $W_k$, and the twin wheel on $k+1$ vertices is denoted by $W_k^{\text{t}}$. A {\em proper wheel} is a wheel that is neither a universal wheel nor a twin wheel.

It is easy to see that the only wheels in which every hole is of length six are the wheels $W_6$ and $W_6^{\text{t}}$. Clearly, $W_6$ and $W_6^{\text{t}}$ are $P_7$-free, and we deduce that $W_6$ and $W_6^{\text{t}}$ are the only $(P_7,C_4,C_5,C_7)$-free wheels.

Let $\mathcal{G}_{\text{UT}}$ be the class of all (3PC, proper wheel)-free graphs.

\begin{lemma} \label{lemma-P6C4C5C7Theta33-free-in-GUT} Every $(P_7,C_4,C_5,C_7,\Theta_3^3)$-free graph is (3PC, proper wheel)-free. In other words, all (3PC, proper wheel)-free graphs belong to $\mathcal{G}_{\text{UT}}$.
\end{lemma}
\begin{proof}
This follows from the fact that $\Theta_3^3$ is the only $(P_7,C_4,C_5,C_7)$-free 3PC, and from the fact that the universal wheel $W_6$ and the twin wheel $W_6^{\text{t}}$ are the only $(P_7,C_4,C_5,C_7)$-free wheels.
\end{proof}

A {\em ring} is a graph $R$ whose vertex set can be partitioned into $k \geq 4$ nonempty sets, say $X_0,\dots,X_{k-1}$ (with indices understood to be in $\mathbb{Z}_k$), such that for all $i \in \mathbb{Z}_k$, $X_i$ can be ordered as $X_i = \{u_1^i,\dots,u_{|X_i|}^i\}$ so that $X_i \subseteq N_R[u_{|X_i|}^i] \subseteq \dots \subseteq N_R[u_1^i] = X_{i-1} \cup X_i \cup X_{i+1}$.\footnote{Note that this implies that $X_0,\dots,X_{k-1}$ are cliques, as well as that $u_1^0,u_1^1,\dots,u_1^{k-1},u_1^0$ is a $k$-hole in $R$. It also implies that for all $i \in \mathbb{Z}_k$, $X_i$ is anticomplete to $V(R) \setminus (X_{i-1} \cup X_i \cup X_{i+1})$.} Under these circumstances, we say that the ring $R$ is of {\em length} $k$, as well as that $R$ is a {\em $k$-ring}. Furthermore, $(X_0,\dots,X_{k-1})$ is said to be a {\em good partition} of the ring $R$. A ring is {\em long} if it is of length at least five. In what follows, we will mostly need 6-rings (see Figure~\ref{fig:6Ring-partition}).

\begin{figure}
\begin{center}
\includegraphics[scale=0.6]{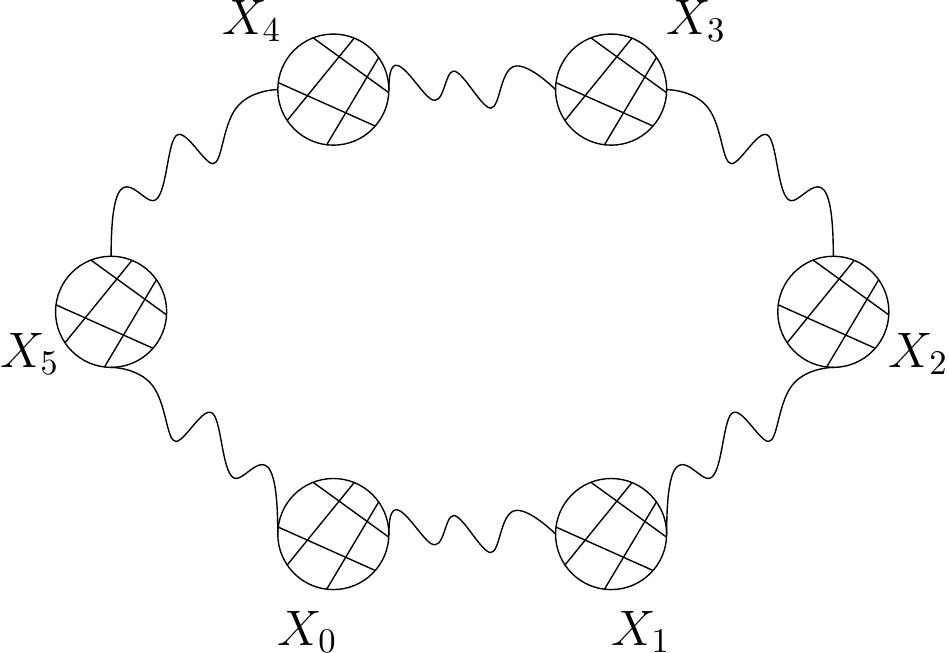}
\end{center}
\caption{6-Ring with good partition $(X_0,\dots,X_5)$. A shaded disk represents a nonempty clique. A wavy line between two cliques indicates that there are edges between the two cliques (furthermore, such edges must obey the axioms from the definition of a ring). The absence of a line between two cliques indicates that the two cliques are anticomplete to each other.} \label{fig:6Ring-partition}
\end{figure}

\begin{lemma} \label{lemma-6-ring-6-wreath-char} Let $R$ be a graph, and let $(X_0,\dots,X_5)$, with indices in $\mathbb{Z}_6$, be a partition of $V(R)$ into nonempty sets. Then the following are equivalent:
\begin{itemize}
\item $R$ is a 6-ring with good partition $(X_0,\dots,X_5)$, $X_0$ is complete to $X_1$, $X_2$ is complete to $X_3$, and $X_4$ is complete to $X_5$;
\item $R$ is a 6-wreath with good partition $(X_0,\dots,X_5)$.
\end{itemize}
\end{lemma}
\begin{proof}
This is immediate from the appropriate definitions.
\end{proof}

For an integer $k \geq 4$, a {\em $k$-hyperhole} is any graph obtained from $C_k$ by blowing up each vertex to a nonempty clique. A {\em long hole} is a hole of length at least five. Note that for all $k \geq 4$, every $k$-hyperhole is a $k$-ring.

Let $\mathcal{B}_{\text{UT}}$ be the class of all graphs $G$ that satisfy at least one of the following:
\begin{itemize}
\item $G$ has exactly one nontrivial anticomponent, and this anticomponent is a long ring;
\item $G$ is (long hole, $K_{2,3}$, $\overline{C_6}$)-free;
\item $\alpha(G) = 2$, and every anticomponent of $G$ is either a 5-hyperhole or a $(C_5,\overline{C_6})$-free graph.
\end{itemize}

The following two results were proven in~\cite{CliqueCutsetsBeyondChordalGraphs} (Theorem~\ref{decomp-thm-GUT} below corresponds to Theorem~1.6 from~\cite{CliqueCutsetsBeyondChordalGraphs}, and Theorem~\ref{ring-in-GT} follows immediately from Lemma~2.4 from~\cite{CliqueCutsetsBeyondChordalGraphs}).

\begin{theorem} \label{decomp-thm-GUT}~\cite{CliqueCutsetsBeyondChordalGraphs} Every graph in $\mathcal{G}_{\text{UT}}$ either belongs to $\mathcal{B}_{\text{UT}}$ or admits a clique-cutset.
\end{theorem}

\begin{theorem} \label{ring-in-GT}~\cite{CliqueCutsetsBeyondChordalGraphs} Let $k \geq 4$. Then every $k$-ring is (3PC, proper wheel, universal wheel)-free, and every hole in a $k$-ring is of length $k$.
\end{theorem}

\begin{lemma} \label{lemma-ring-P7C4C5Theta33-free} Let $R$ be a 6-ring. Then $R$ is $(C_4,C_5,C_7,\Theta_3^3)$-free, and $R$ does not admit a clique-cutset.
\end{lemma}
\begin{proof}
Since $\Theta_3^3$ is a theta, and every theta is a 3PC, Theorem~\ref{ring-in-GT} immediately implies that $R$ is $(C_4,C_5,C_7,\Theta_3^3)$-free.

It remains to show that $R$ does not admit a clique-cutset. Let $(X_0,\dots,X_5)$, with indices in $\mathbb{Z}_6$, be a good partition of the 6-ring $R$. Let $S$ be a clique in $R$; we must show that $R \setminus S$ is connected. By the definition of a ring, we see that there exists some $i \in \mathbb{Z}_6$ such that $S \subseteq X_i \cup X_{i+1}$. Clearly, $R \setminus (X_i \cup X_{i+1})$ is connected, and every vertex in $X_i \cup X_{i+1}$ has a neighbor in $V(R) \setminus (X_i \cup X_{i+1})$; thus, $R \setminus S$ is indeed connected. This completes the argument.
\end{proof}

\begin{lemma} \label{lemma-char-P7free6ring} Let $R$ be a graph. Then the following are equivalent:
\begin{enumerate}[(a)]
\item $R$ is a $(P_7,C_4,C_5,C_7)$-free ring;
\item $R$ is a $P_7$-free 6-ring;
\item $R$ is either a 6-wreath or a 6-crown.
\end{enumerate}
\end{lemma}
\begin{proof}
We show that (a) and (b) are equivalent, and that (b) and (c) are equivalent.

By Lemma~\ref{lemma-ring-P7C4C5Theta33-free}, (b) implies (a). Let us show that (a) implies (b). The definition of a ring implies that, for every integer $k \geq 4$, every $k$-ring contains a $k$-hole. Furthermore, it is clear that every hole of length greater than seven contains an induced $P_7$. Thus, every $(P_7,C_4,C_5,C_7)$-free ring is of length six, and it follows that (a) implies (b). This proves that (a) and (b) are equivalent.

We next show that (c) implies (b).

\begin{quote}
{\em (1) Every 6-wreath is a $P_7$-free 6-ring.}
\end{quote}
\begin{proof}[Proof of (1)]
It is clear that every 6-wreath is a 6-ring. Further, note that if $(X_0,\dots,X_5)$ is a good partition of a 6-wreath, then $X_0 \cup X_1$, $X_2 \cup X_3$, and $X_4 \cup X_5$ are cliques; consequently, the vertex set of any 6-wreath can be partitioned into three cliques. Since $\alpha(P_7) = 4$, it follows that every 6-wreath is $P_7$-free. This proves (1).
\end{proof}

\begin{quote}
{\em (2) Every 6-crown is a $P_7$-free 6-ring.}
\end{quote}
\begin{proof}[Proof of (2)]
We use the notation from Figure~\ref{fig:C63C64}.

It is clear that $C_6^3$ and $C_6^4$ are 6-rings; it is also clear that any graph obtained by blowing up each vertex of a 6-ring to a nonempty clique is a 6-ring. Thus, every 6-crown is a 6-ring. It remains to show that 6-crowns are $P_7$-free. Since $P_7$ does not admit a proper homogeneous set, it is clear that any graph obtained by blowing up each vertex of a $P_7$-free graph to a nonempty clique is $P_7$-free. Thus, we need only show that $C_6^3$ and $C_6^4$ are $P_7$-free. Since $C_6^3$ is an induced subgraph of $C_6^4$, it in fact suffices to show that $C_6^4$ is $P_7$-free.

Suppose that $C_6^4$ is not $P_7$-free, and let $P$ be an induced $P_7$ in $C_6^4$. Clearly, $\{c_1,c_2,d_2\},\{d_3\},\{c_3,c_4,d_4\},\{c_5,d_6,c_0\}$ is a partition of $V(C_6^4)$ into four cliques; since $\alpha(P) = \alpha(P_7) = 4$, it follows that $V(P)$ intersects each of these four cliques, and in particular, $d_3 \in V(P)$. Similarly, $d_4 \in V(P)$. Let us show that $V(P) \cap \{c_3,c_4\} \neq \emptyset$. Suppose otherwise. Then $P$ is an induced $P_7$ in the 8-vertex graph $C_6^4 \setminus \{c_3,c_4\}$, and it follows that $P$ contains all but one of the vertices $c_0,c_1,c_2,c_5,d_2,d_3,d_4,d_5$. But this implies that either $c_0,c_5,d_5 \in V(P)$ or $c_1,c_2,d_2 \in V(P)$, neither of which is possible since $\{c_0,c_5,d_5\}$ and $\{c_1,c_2,d_2\}$ are both cliques of size three in $C_6^4$, and $P_7$ is triangle-free. Thus, $V(P) \cap \{c_3,c_4\} \neq \emptyset$.

By symmetry, we may assume that $c_3 \in V(P)$. We now have that $c_3,d_3,d_4 \in V(P)$. Since $c_3$ is complete to $\{d_3,d_4\}$, and since no vertex of $P$ is of degree greater than two in $P$, we see that $N_R[c_3] \cap V(P) = \{c_3,d_3,d_4\}$. But this implies that $c_2,d_2,c_4 \notin V(P)$. Since $C_6^4$ has ten vertices, while $P$ has seven, it follows that $V(P) = V(C_6^4) \setminus \{c_2,d_2,c_4\}$. But this is impossible since (for instance) $\{c_0,c_5,d_5\}$ is now a triangle in $P$, contrary to the fact that $P$ is triangle-free. This proves (2).
\end{proof}

By (1) and (2), we have that (c) implies (b). It remains to show that (b) implies (c). From now on, we assume that $R$ satisfies (b), that is, that $R$ is a $P_7$-free 6-ring. Let $(X_0,\dots,X_5)$ be a good partition of the 6-ring $R$. For all $i \in \mathbb{Z}_6$, set $t_i = |X_i|$ and $X_i = \{u_1^i,\dots,u_{t_i}^i\}$ so that $X_i \subseteq N_R[u_{t_i}^i] \subseteq \dots \subseteq N_R[u_1^i] = X_{i-1} \cup X_i \cup X_{i+1}$, as in the definition of a ring.

\begin{quote}
{\em (3) For all $i \in \mathbb{Z}_6$, either $X_i$ is complete to $X_{i+1}$, or $X_{i+3}$ is complete to $X_{i+4}$.}
\end{quote}
\begin{proof}[Proof of (3)]
Suppose otherwise. By symmetry, we may assume that $X_0$ is not complete to $X_1$, and that $X_3$ is not complete to $X_4$. Then $u_{t_0}^0u_{t_1}^1,u_{t_3}^3u_{t_4}^4 \notin E(R)$; note that this implies that $t_0,t_1,t_3,t_4 \geq 2$. But now $u_{t_0}^0,u_1^0,u_{t_1}^1,u_1^2,u_{t_3}^3,u_1^4,u_{t_4}^4$ is an induced $P_7$ in $R$, a contradiction. This proves (3).
\end{proof}

\begin{quote}
{\em (4) If $X_i$ is complete to at least one of $X_{i-1},X_{i+1}$ for all $i \in \mathbb{Z}_6$, then $R$ is a 6-wreath.}
\end{quote}
\begin{proof}[Proof of (4)]
Assume that $X_i$ is complete to at least one of $X_{i-1},X_{i+1}$ for all $i \in \mathbb{Z}_6$. If $X_i$ is complete to $X_{i-1} \cup X_{i+1}$ for all $i \in \mathbb{Z}_6$, then $R$ is a 6-hyperhole and therefore a 6-wreath. So assume that this is not the case. We may assume that $X_1$ is not complete to $X_2$ (we cyclically permute the $X_i$'s if necessary). Then by assumption, $X_0$ is complete to $X_1$, and $X_2$ is complete to $X_3$. Furthermore, by (3), $X_4$ is complete to $X_5$. Lemma~\ref{lemma-6-ring-6-wreath-char} now implies that $R$ is a 6-wreath with good partition $(X_0,\dots,X_5)$. This proves (4).
\end{proof}

Recall that $X_i \subseteq N_R[u_{t_i}^i] \subseteq \dots \subseteq N_R[u_1^i] = X_{i-1} \cup X_i \cup X_{i+1}$ for all $i \in \mathbb{Z}_6$. For all $i \in \mathbb{Z}_6$, let $s_i \in \{1,\dots,t_i\}$ be maximal with the property that $N_R[u_{s_i}^i] = X_{i-1} \cup X_i \cup X_{i+1}$. Further, for all $i \in \mathbb{Z}_6$, let $s_i^+ \in \{1,\dots,t_i\}$ be maximal with the property that $u_{s_i}^+$ is complete to $X_{i+1}$, and let $s_i^- \in \{1,\dots,t_i\}$ be maximal with the property that $u_{s_i}^-$ is complete to $X_{i-1}$. Clearly, $s_i = \min\{s_i^+,s_i^-\}$ for all $i \in \mathbb{Z}_6$. For all $i \in \mathbb{Z}_6$, set $C_i = \{u_1^i,\dots,u_{s_i}^i\}$ and $D_i = X_i \setminus C_i$. By construction, we have that for all $i \in \mathbb{Z}_6$, $u_1^i \in C_i$ (consequently, $C_i \neq \emptyset$), $C_i$ is complete to $X_{i-1} \cup X_{i+1}$, and every vertex in $D_i$ has a nonneighbor in $X_{i-1} \cup X_{i+1}$.\footnote{If $D_i = \emptyset$, then it is vacuously true that every vertex in $D_i$ has a nonneighbor in $X_{i-1} \cup X_{i+1}$.}

\begin{quote}
{\em (5) For all $i \in \mathbb{Z}_6$, if $X_i$ is complete neither to $X_{i-1}$ nor to $X_{i+1}$, then $s_i^+ = s_i^- = s_i < t_i$.}
\end{quote}
\begin{proof}[Proof of (5)]
By symmetry, it suffices to prove the statement for $i = 1$. So assume that $X_1$ is complete neither to $X_0$ nor to $X_2$. It is then clear that $s_1^+,s_1^- < t_1$, as well as that $u_{t_1}^1$ is anticomplete to $\{u_{t_0}^0,u_{t_2}^2\}$. Since $s_1 = \min\{s_1^+,s_1^-\}$, it only remains to show that $s_1^+ = s_1^-$. Suppose otherwise. By symmetry, we may assume that $s_1^+ > s_1^-$. Set $\ell = s_1^-+1$. Then $s_1^- < \ell \leq s_1^+$, and it follows that $u_{\ell}^1$ is complete to $X_2$, but is not complete to $X_0$ (consequently, $u_{\ell}^1u_{t_0}^0 \notin E(G)$). But now $u_{t_1}^1,u_{\ell}^1,u_{t_2}^2,u_1^3,u_1^4,u_1^5,u_{t_0}^0$ is an induced $P_7$ in $R$, a contradiction. This proves (5).
\end{proof}

\begin{quote}
{\em (6) For all $i \in \mathbb{Z}_6$, if $X_i$ is complete neither to $X_{i-1}$ nor to $X_{i+1}$, then $s_i < t_i$, $D_i \neq \emptyset$, and $D_i$ is a (not necessarily proper) homogeneous set in $R$.}
\end{quote}
\begin{proof}[Proof of (6)]
By symmetry, it suffices to prove the claim for $i = 1$. So assume that $X_1$ is complete neither to $X_0$ nor to $X_2$. By (5), we have that $s_1^+ = s_1^- = s_1 < t_1$. Note that this implies that $D_1 \neq \emptyset$, and that $D_1$ is anticomplete to $\{u_{t_0}^0,u_{t_2}^2\}$. Now, clearly, $\{u_1^1,\dots,u_{s_1}^1\} \cup \{u_1^0,u_1^2\}$ is complete to $D_1$, and $X_3 \cup X_4 \cup X_5 \{u_{t_0}^0,u_{t_2}^2\}$ is anticomplete to $D_1$. Thus, in order to show that $D_1$ is a homogeneous set in $R$, it suffices to show that no vertex in $(X_0 \setminus \{u_1^0,u_{t_0}^0\}) \cup (X_2 \setminus \{u_1^2,u_{t_2}^2\})$ is mixed on $D_1$. Suppose otherwise. By symmetry, we may assume that there exists some $\ell_2 \in \{2,\dots,t_2-1\}$ such that $u_{\ell_2}^2$ is mixed on $D_1$. Fix $\ell_1,\ell_1' \in \{s_1+1,\dots,t_1\}$ such that $u_{\ell_2}^2$ is adjacent to $u_{\ell_1}^1$ and nonadjacent to $u_{\ell_1'}^1$; clearly, $\ell_1 < \ell_1'$. But now $u_{\ell_1'}^1,u_{\ell_1}^1,u_{\ell_2}^2,u_1^3,u_1^4,u_1^5,u_{t_0}^0$ is an induced $P_7$ in $R$, a contradiction. This proves (6).
\end{proof}

\begin{quote}
{\em (7) For all $i \in \mathbb{Z}_6$, if $X_i$ is complete neither to $X_{i-1}$ nor to $X_{i+1}$, then $D_{i-1},D_i,D_{i+1} \neq \emptyset$, and $D_i$ is anticomplete to $D_{i-1} \cup D_{i+1}$.}
\end{quote}
\begin{proof}[Proof of (7)]
Fix $i \in \mathbb{Z}_6$, and assume that $X_i$ is complete neither to $X_{i-1}$ nor to $X_{i+1}$. Clearly then, $D_{i-1},D_i,D_{i+1} \neq \emptyset$. It remains to show that $D_i$ is anticomplete to $D_{i-1} \cup D_{i+1}$. Suppose otherwise. By symmetry, we may assume that some $d_i \in D_i$ and $d_{i+1} \in D_{i+1}$ are adjacent. By (6), $D_i$ is a homogeneous set in $R$; since $d_{i+1}$ has a neighbor in $D_i$, we now deduce that $d_{i+1}$ is complete to $D_i$. By construction, $C_i$ is complete to $X_{i+1}$; since $d_{i+1} \in X_{i+1}$, we see that $d_{i+1}$ is complete to $C_i$. Thus, $d_{i+1}$ is complete to $C_i \cup D_i = X_i$. Since $d_{i+1} \in D_{i+1}$, we deduce that $d_{i+1}$ is not complete to $X_{i+2}$. But now we have that $X_{i+1}$ is complete neither to $X_i$ nor to $X_{i+2}$, and so (6) implies that $D_{i+1}$ is a homogeneous set in $R$. Since $D_i$ and $D_{i+1}$ are disjoint homogeneous sets in $R$, and since $d_i \in D_i$ and $d_{i+1} \in D_{i+1}$ are adjacent, we deduce that $D_i$ is complete to $D_{i+1}$. But since $C_i$ is complete to $X_{i+1}$, and $C_{i+1}$ is complete to $X_i$, we now readily deduce that $X_i$ is complete to $X_{i+1}$, a contradiction. This proves (7).
\end{proof}

\begin{quote}
{\em (8) If there exists some $i \in \mathbb{Z}_6$ such that $X_i$ is complete neither to $X_{i-1}$ nor to $X_{i+1}$, then $R$ is a 6-crown.}
\end{quote}
\begin{proof}[Proof (8)]
By symmetry, we may assume that $X_4$ is complete neither to $X_3$ nor to $X_5$. By (7), we have that $D_3,D_4,D_5$ are nonempty and pairwise anticomplete to each other. By (3), $X_1$ is complete to $X_0 \cup X_2$. Further, by (3), either $X_2$ is complete to $X_3$, or $X_5$ is complete to $X_0$; by symmetry, we may assume that $X_5$ is complete to $X_0$ ($X_2$ may or may not be complete to $X_3$). Note that this implies that $D_0 = D_1 = \emptyset$. Our goal is to show that $R$ is a 6-crown with good triple $(\{C_i\}_{i \in \mathbb{Z}_6},\{D_i\}_{i \in \mathbb{Z}_6},2)$.

So far, we have shown that $D_3,D_4,D_5$ are nonempty and anticomplete to each other, and that $D_0 = D_1 = \emptyset$. We claim that $D_0,\dots,D_5$ are pairwise anticomplete to each other. Clearly, we need only show that $D_2$ is anticomplete to $D_i$ for all $i \in \mathbb{Z}_6 \setminus \{2\}$. Since $X_2$ is anticomplete to $X_0 \cup X_4 \cup X_5$ (and therefore to $D_0 \cup D_4 \cup D_5$ as well), and since $D_1 = \emptyset$, it only remains to show that $D_2$ is anticomplete to $D_3$. If $D_2 = \emptyset$, then this is immediate. So suppose that $D_2 \neq \emptyset$. Recall that $X_1$ is complete to $X_2$; since $D_2 \neq \emptyset$, we deduce that $X_2$ is not complete to $X_3$. But now $X_3$ is complete neither to $X_2$ nor to $X_4$, and so (7) implies that $D_3$ is anticomplete to $D_2$, which is what we needed. We have now established that $D_0,\dots,D_5$ are pairwise anticomplete to each other.

Next, by the definition of a 6-ring, $X_0,\dots,X_5$ are cliques; consequently, $D_i$ is complete to $C_i$ for all $i \in \mathbb{Z}_6$. Further, by construction, $C_i$ is complete to $X_{i-1} \cup X_{i+1}$ for all $i \in \mathbb{Z}_6$; consequently, $D_i$ is complete to $C_{i-1} \cup C_{i+1}$ for all $i \in \mathbb{Z}_6$. By the definition of a 6-ring, we have that $X_i$ is anticomplete to $X_{i+2} \cup X_{i+3} \cup X_{i+4}$ for all $i \in \mathbb{Z}_6$; consequently, $D_i$ is anticomplete to $C_{i+2} \cup C_{i+3} \cup C_{i+4}$ for all $i \in \mathbb{Z}_6$.

To show that $R$ is a 6-crown with good partition $(\{C_i\}_{i \in \mathbb{Z}_6},\{D_i\}_{i \in \mathbb{Z}_6},2)$, it now remains to show that $C_0,\dots,C_5$ are nonempty cliques, and that $C_i$ is complete to $C_{i-1} \cup C_{i+1}$ and anticomplete to $C_{i+2} \cup C_{i+3} \cup C_{i+4}$ for all $i \in \mathbb{Z}_6$. The fact that $C_0,\dots,C_5$ are cliques follows from the fact that, by the definition of a 6-ring, $X_0,\dots,X_5$ are cliques. Furthermore, it is clear that $u_1^i \in C_i$ for all $i \in \mathbb{Z}_6$, and so $C_0,\dots,C_5$ are all nonempty. By construction, $C_i$ is complete to $X_{i-1} \cup X_{i+1}$ (and consequently, to $C_{i-1} \cup C_{i+1}$) for all $i \in \mathbb{Z}_6$. Further, by the definition of a 6-ring, $X_i$ is anticomplete to $X_{i+2} \cup X_{i+3} \cup X_{i+4}$ for all $i \in \mathbb{Z}_6$; consequently, $C_i$ is anticomplete to $C_{i+2} \cup C_{i+3} \cup C_{i+4}$ for all $i \in \mathbb{Z}_6$. It now follows that $R$ is a 6-crown with good triple $(\{C_i\}_{i \in \mathbb{Z}_6},\{D_i\}_{i \in \mathbb{Z}_6},2)$. This proves (8).
\end{proof}

Clearly, (4) and (8) together imply that $R$ is either a 6-wreath or a 6-crown. Thus, (b) implies (c). This completes the argument.
\end{proof}

We are now ready to prove Theorem~\ref{thm-P7C4C5C7Theta33-free-decomp}, restated below for the reader's convenience.

\begin{thm-P7C4C5C7Theta33-free-decomp} Let $G$ be a graph. Then the following two statements are equivalent:
\begin{enumerate}[(i)]
\item $G$ is a $(P_7,C_4,C_5,C_7,\Theta_3^3)$-free graph that does not admit a clique-cutset;
\item either $G$ is a complete graph, or $G$ contains exactly one nontrivial anticomponent, and this anticomponent is either a 6-wreath or a 6-crown.
\end{enumerate}
\end{thm-P7C4C5C7Theta33-free-decomp}
\begin{proof}
Suppose first that $G$ satisfies (ii); we must show that $G$ satisfies (i). If $G$ is a complete graph, then this is immediate. So assume that $G$ contains exactly one nontrivial anticomponent, call it $R$, and that $R$ is either a 6-wreath or a 6-crown. By Lemma~\ref{lemma-char-P7free6ring}, $R$ is a $P_7$-free 6-ring. Since $R$ is a 6-ring, Lemma~\ref{lemma-ring-P7C4C5Theta33-free} implies that $R$ is $(C_4,C_5,C_7,\Theta_3^3)$-free and does not admit a clique-cutset. Thus, $R$ is $(P_7,C_4,C_5,C_7,\Theta_3^3)$-free and does not admit a clique-cutset. The fact that $G$ satisfies (i) now follows from Lemma~\ref{lemma-one-anticomp}.

Suppose now that $G$ satisfies (i); we must show that $G$ satisfies (ii). By Lemma~\ref{lemma-P6C4C5C7Theta33-free-in-GUT}, $G \in \mathcal{G}_{\text{UT}}$. Since $G$ satisfies (i), $G$ does not admit a clique-cutset. Theorem~\ref{decomp-thm-GUT} now implies that $G \in \mathcal{B}_{\text{UT}}$. Then by the definition of $\mathcal{B}_{\text{UT}}$, one of the following holds:
\begin{itemize}
\item[(a)] $G$ has exactly one nontrivial anticomponent, and this anticomponent is a long ring;
\item[(b)] $G$ is (long hole, $K_{2,3}$, $\overline{C_6}$)-free;
\item[(c)] $\alpha(G) = 2$, and every anticomponent of $G$ is either a 5-hyperhole or a $(C_5,\overline{C_6})$-free graph.
\end{itemize}
Suppose first that $G$ satisfies (a). Let $R$ be the unique nontrivial anticomponent of $G$; then $R$ is a long ring. Since $G$ (and therefore $R$ as well) is $(P_7,C_4,C_5,C_7)$-free, Lemma~\ref{lemma-char-P7free6ring} implies that $R$ is either a 6-wreath or a 6-crown, and in either case, $G$ satisfies (ii).

Suppose now that $G$ satisfies (b) or (c). Then either $G$ is long-hole-free, or $\alpha(G) = 2$. Since every hole of length greater than five contains a stable set of size three, we deduce that $G$ contains no holes of length greater than five. But by (i), $G$ is also $(C_4,C_5)$-free. Thus, $G$ contains no holes, that is, $G$ is chordal. By~\cite{Dirac61}, it follows that either $G$ is complete, or $G$ admits a clique-cutset. Since $G$ satisfies (i), the latter is impossible; thus, $G$ is complete, and it follows that $G$ satisfies (ii).
\end{proof}

\section{Algorithms and $\chi$-boundedness} \label{sec:alg}

In this section, we give polynomial-time algorithms for solving the following three discrete optimization problems for the class of $(P_7,C_4,C_5)$-free graphs: the minimum coloring problem, the maximum weight stable set problem, and the maximum weight clique problem. We also show that the class of $(P_7,C_4,C_5)$-free graphs is $\chi$-bounded by a linear function. 

Our algorithms will take the following general approach.

\begin{framed}
1. Decompose the input graph $G$ via clique-cutsets into subgraphs that do not admit clique-cutsets.
These subgraphs are called \emph{atoms}. \\

2. Find the solution for each atom. \\

3. Combine solutions for atoms along the clique-cutsets to obtain a solution for $G$.
\end{framed}

\subsection{Clique-cutset decomposition}

Let $G=(V,E)$ be a graph and $K\subseteq V$ a clique-cutset such that
$G\setminus K$ is a disjoint union of two subgraphs $H_1$ and $H_2$ of $G$.
We let $G_i=G[H_i\cup K]$ for $i=1,2$.
We say that $G$ is \emph{decomposed into} $G_1$ and $G_2$ \emph{via} $K$,
and call this a \emph{decomposition step}.
We then recursively decompose $G_1$ and $G_2$ via clique-cutsets until no clique-cutset exists.
This procedure can be represented by a rooted binary tree $T(G)$ where $G$ is the root and
the leaves are induced subgraphs of $G$ that do not admit clique-cutsets. These subgraphs are called
\emph{atoms} of $G$.  Tarjan~\cite{Ta85} showed that for any graph $G$, $T(G)$ can be found in
$O(nm)$ time. Moreover, in each decomposition step Tarjan's algorithm produces an atom,
and consequently $T(G)$ has at most $n-1$ leaves (or equivalently atoms).

Let $k\ge 1$ be a fixed integer.
Tarjan~\cite{Ta85} observed that  $G$ is $k$-colorable if and only if each atom of $G$ is $k$-colorable.
This implies that if one can solve chromatic number for atoms, then one can also solve the problem
for $G$. It is straightforward to check that once a $k$-coloring of each atom is found, then it takes
$O(n^2)$ time to combine these colorings to obtain a $k$-coloring of $G$.

In a slightly more complicated fashion,  Tarjan~\cite{Ta85} showed that once the maximum weight
stable set problem is solved for atoms, one can solve the problem for $G$.
Let $G=(V,E)$ be a graph with a weight function $w:V\rightarrow \mathbb{R}$. For a given subset $S\subseteq  V$,
we let $w(S)=\sum_{v\in S}w(v)$, and denote the maximum weight of a stable set of $G$ by $\alpha_w(G)$.
Suppose that $G$ is decomposed into $A$ and $B$  via a clique-cutset $S$, where $A$ is an atom.
We explain Tarjan's approach as follows. To compute a stable set of weight $\alpha_w(G)$, we do the following.
\begin{enumerate}[\bfseries (1)]
\item Compute a maximum weight  stable set $I'$ of $A\setminus S$.\label{itm:first}
\item For each vertex $v\in S$, compute a maximum weight  stable set $I_v$ of $A\setminus N[v]$.
\item Re-define the weight of $v\in S$ as $w'(v)=w(v)+w(I_v)-w(I')$.
\item Compute the maximum weight  stable set $I''$ of $B$ with respect to the new weight $w'$.
If $I''\cap S=\{v\}$, then let $I= I_v\cup I''$; otherwise let $I=I'\cup I''$.\label{itm:last}
\end{enumerate}
It is easy to see that $\alpha_w(G)=w(I)$. This divide-and-conquer approach can be applied top-down
on $T(G)$ to obtain a solution for $G$ by solving $O(n^2)$ subproblems on induced subgraphs of atoms,
since there are  $O(n)$ decomposition steps and  each step amounts to solving $O(n)$ subproblems as explained in
(\ref{itm:first})-(\ref{itm:last}).

Therefore, it suffices to explain below how to solve coloring and maximum weight stable set for atoms
of $(P_7,C_4,C_5)$-free graphs. In the following, we assume that $A$ is an atom.
The {\em join} of  two given graphs $G_1$ and $G_2$ is the graph obtained from the disjoint union of $G_1$
and $G_2$ by adding every possible edge between vertices in $G_1$ and vertices in $G_2$.
We only need the following  two decomposition theorems, which follow from
Theorem~\ref{thm-decomp-P7C4C5-free-with-C7}, and from Theorems~\ref{thm-P7C4C5C7-free-contains-Theta-decomp}
and~\ref{thm-P7C4C5C7Theta33-free-decomp}, respectively.
\begin{theorem}[Decomposition $C_7$]\label{thm:decomposition}
Let $G$ be a  $(P_7,C_4,C_5)$-free atom that contains a $C_7$.
Then $G$ is the join of a (possibly null) complete graph with either a thickened emerald (see Figure \ref{fig:emerald}) or
a $7$-bracelet (see Figure \ref{fig:7bracelet}).
\end{theorem}

\begin{theorem}[Decomposition $C_7$-free]\label{thm:decompositionC7-free}
Let $G$ be a $(P_7,C_4,C_5,C_7)$-free atom. Then
$G$ is either a complete graph, or the join of a (possibly null) complete graph with either a lantern, or a $6$-ring.
\end{theorem}

\subsection{Solving the maximum weight stable set problem} 

Let $A$ be a $(P_7,C_4,C_5)$-free atom. It then follows from
Theorems~\ref{thm:decomposition} and~\ref{thm:decompositionC7-free}
that $A$ is either a complete graph, or the join of a (possibly empty) complete graph $U$ with either a thickened emerald, or a $7$-bracelet, or a lantern or a $6$-ring.
Note that any vertex in $U$ is a universal vertex of $A$. Therefore, any stable set containing a vertex $u$ in $U$
must be $\{u\}$. It thus suffices to solve the problem for thickened emeralds, or $7$-bracelets, or lanterns or  $6$-rings.
We do so by the following unified approach developed in~\cite{BH07}.
Let $G$ be a graph with a weight function $w:V(G)\rightarrow \mathbb{R}$. 
Then the following holds:
\[\alpha_w(G)=\max(\{0\} \cup \{w(v)+\alpha_w(G-N[v]) \mid v\in V(G)\})\]
(Note that $\alpha_w(G) = 0$ if and only if $w(v) \leq 0$ for all $v \in V(G)$. In this case, $\emptyset$ is a maximum weight stable set.) Therefore, computing the maximum weight of a stable set in $G$ reduces to $n$ subproblems on certain induced subgraphs of $G$.
Luckily, the structure of the induced subgraphs $G-N[v]$ of the four special kinds of graphs under consideration
is simple and it allows the maximum weight stable set problem to be solved efficiently.
\begin{lemma}\label{lem:nonneighborhood}
Let $G$ be a thickened emerald, or a $7$-bracelet, or a lantern or a $6$-ring. For any vertex $v\in V(G)$,
$G-N[v]$ is a chordal graph.
\end{lemma}
\begin{proof}
This is routine to verify.
\end{proof}

It is well known that the maximum weight stable set problem can be solved in $O(m+n)$ time
for a chordal graph with $n$ vertices and $m$ edges~\cite{Ga72}.
Moreover, since the property in Lemma \ref{lem:nonneighborhood} also holds for all induced subgraphs
of the four kinds of graphs, it then follows from Lemma \ref{lem:nonneighborhood}
that we can solve the problem for an induced subgraph of a $(P_7,C_4,C_5)$-free atom
in $O(n_A+m_A)O(n_A)=O(m_An_A)$ time.
Note that for two different atoms $A$ and $B$, the subgraphs of $A$ for which the subproblems
need to be solved are vertex-disjoint from the subgraphs of $B$ for which the subproblems need to be solved.
This implies that it takes $O(n)\sum_{A}O(m_An_A)=O(n)O(mn)=O(n^2m)$ time to
solve all these subproblems, where the summation goes over all atoms of $G$.
So, the total running time is dominated by solving subproblems for atoms, that is, $O(n^2m)$.
\begin{theorem}
The maximum weight stable set problem can be solved in $O(n^2m)$ time for $(P_7,C_4,C_5)$-free graphs.
\end{theorem}

\subsection{Solving the minimum coloring and maximum weight clique problems}

Let $A$ be a $(P_7,C_4,C_5)$-free atom. It then follows from
Theorems~\ref{thm:decomposition} and~\ref{thm:decompositionC7-free}
that $A$ is either a complete graph, or the join of a (possibly empty) complete graph $U$ with either a thickened emerald, a $7$-bracelet, a lantern, or a $6$-ring.
Note that any vertex in $U$ is a universal vertex of $A$ and thus requires its own color in $A$.
It thus suffices to solve the problem for thickened emeralds, $7$-bracelets, lanterns, and $6$-rings.

We first deal with lanterns and $6$-rings.
\begin{lemma}\label{lem:colorlantern}
Let $R$ be an $r$-lantern with a good partition $(A, B_1,\ldots, B_r, C_1,\ldots,C_r,D)$ 
and clique number $\omega$. 
Then $\chi(R) = \omega$, and furthermore, there is an algorithm to optimally color $R$ in $O(|R|)$ time.
\end{lemma}
\begin{proof}
Recall that the vertices in $B_1$ and $C_1$ can be ordered as
$B_1 = \{b_1^1,\dots,b_{|B_1|}^1\}$ and $C_1 = \{c_1^1,\dots,c_{|C_1|}^1\}$
so that $N_R[b_{|B_1|}^1] \cap C_1 \subseteq \dots \subseteq N_R[b_1^1] \cap C_1 = C_1$
and $N_R[c_{|C_1|}^1] \cap B_1 \subseteq \dots \subseteq N_R[c_1^1] \cap B_1 = B_1$.
We greedily color the vertices of $R$ using colors $1,2,\ldots,\omega$ as follows.
\begin{itemize}
\item Assign colors $1,2,\ldots,|B_1|$ to vertices $b_1^1,b_2^1,\ldots, b_{|B_1|}^1$, respectively.
\item Assign colors $\omega,\omega-1,\ldots,\omega-|C_1|+1$ to vertices $c_1^1,c_2^1\ldots, c_{|C_1|}^1$, respectively.
\item Assign colors $\omega,\omega-1,\ldots$ and $1,2,\ldots$ to vertices in $A$ and $D$, respectively.
\item For each $i\ge 2$, assign colors  $1,2,\ldots$ and $\omega,\omega-1,\ldots$ to vertices in $B_i$ and $C_i$, respectively.
\end{itemize}
We show that this is a proper coloring of $R$. Suppose not. Then there exist two adjacent vertices $x$ and $y$
that have the same color, say $i$, under this coloring. Clearly, $x$ and $y$ are in different sets in the good partition.
Suppose first that $x\in B_1$ and $y\in C_1$.  Then $x=b_i^1$ and $y=c_{\omega-i+1}^1$ due to the definition of our coloring.
It follows from the definition of $B_1$ and $C_1$ that $\{b_1^1,\ldots,b_i^1,c_1^1,\ldots,c_{\omega-i+1}^1\}$ is a clique
of size $\omega+1$, a contradiction. The remaining cases can be proved in a similar way.
Thus, this indeed is a proper coloring of $R$. Clearly, it takes $O(|R|)$ time to obtain this coloring.
This completes our proof.
\end{proof}

\begin{lemma}\label{lem:colorring}
Let $R$ be a $6$-ring with a good partition $(X_0,X_1,\ldots,X_5)$ 
and clique number $\omega$.
Then $\chi(R) = \omega$, and there is an algorithm to optimally color $R$ in $O(|R|)$ time.
\end{lemma}
\begin{proof}
We greedily color the vertices of $R$ using colors $1,2,\ldots,\omega$ as follow.
\begin{itemize}
\item Assign colors $1,2,\ldots$ to vertices $u_1^i, u_2^i,\ldots$ for even $i$
and assign colors $\omega,\omega-1,\ldots$ to vertices $u_1^i, u_2^i,\ldots$ for odd $i$.
\end{itemize}
We show that this is a proper coloring of $R$. Suppose not.
Then there exist two adjacent vertices $x$ and $y$
that have the same color, say $i$, under this coloring. Clearly, $x$ and $y$ are in adjacent sets in the good partition,
say $x\in X_0$ and $y\in X_1$.
Then $x=u_i^0$ and $y=u_{\omega-i+1}^1$ due to the definition of our coloring.
It follows from the definition of  $R$ that $\{u_1^0,\ldots,u_i^0,u_1^1,\ldots,u_{\omega-i+1}^1\}$ is a clique
of size $\omega+1$, a contradiction.
Clearly, it takes $O(|R|)$ time to obtain this coloring.
This completes our proof.
\end{proof}

To apply the above greedy algorithm, however, we need to first obtain a good partition of a lantern and of a $6$-ring, and to find the clique number of such graphs. 

To recognize $6$-rings and find a good partition, we use the algorithm in~\cite{CliqueCutsetsBeyondChordalGraphs}, which
takes $O(n^2)$ time to  find a good partition of  a graph if the graph is a $6$-ring or determine that the graph is not a $6$-ring.
Now we show how to recognize lanterns.
For any graph $G$, we say that a pair $\{u,v\}$ of vertices are {\em twins} if $N_G[v]=N_G[u]$.
Note that being twins defines an equivalence relation on $V(G)$, and so $V(G)$ can be partitioned into equivalence classes
of twins $T_1,\ldots,T_r$. The {\em skeleton} $G'$ of $G$ is the subgraph induced by $\{t_1,\ldots,t_r\}$ where $t_i\in T_i$.
We first determine the equivalence classes $T_1,\ldots, T_r$ of twins and
the skeleton $G'$ in $O(m+n)$ time~\cite{CliqueCutsetsBeyondChordalGraphs}.
Note that if $G$ is a $r$-lantern with a good partition $(A,B_1,\ldots,B_r,C_1,\ldots,C_r,D)$,
then $A$, $D$, $B_i$ and $C_i$ for $2\le i\le r$ become singletons $\{a\}$, $\{d\}$, $\{b_i\}$ and $\{c_i\}$ in $G'$.
Moreover, $b_i$ and $c_i$ have degree 2 in $G'$ and $a$ and $d$ have degree $r\ge 3$ in $G'$.
So, we can scan the adjacency list of each vertex $v\in G'$ to separate vertices of low degree
from those of high degree: if $d_G'(v)=2$, then we put $v$ in $L$; otherwise we put $v$ in $H$.
If we cannot find a vertex $u\in L$ such that one of its two neighbors has degree 2 and the other neighbor has degree $r$,
then $G$ is not a lantern. Otherwise, let $u\in L$ such that one of its neighbors $w$ has degree 2 and the other neighbor $v$ has degree $r$.
Then scanning the adjacency list of each vertex in $G'$ starting from $v$ (this vertex must be $a$ or $d$ if $G$ is a lantern)
will either find the skeleton of a $r$-lantern or show that
$G$ is not a lantern. All we do in the algorithm is to scan the adjacency lists twice and therefore the running time is $O(m+n)$. 

\begin{lemma}\label{lem:recognize lantern}
Given a graph $G$ with $n$ vertices and $m$ edges, there exists an algorithm that
determines whether or not $G$ is a lantern, and if it is a lantern, finds a good partition of $G$
in $O(n^2)$ time.
\end{lemma}

We also need to compute the clique number of 6-rings and lanterns before we can use the greedy algorithms.
\begin{lemma}\label{lem:clique number}
Given a 6-ring or a lantern $G$ with $n$ vertices and $m$ edges
and a good partition of $G$, there exists an $O(n+m)$ algorithm to compute $\omega(G)$.
\end{lemma}

\begin{proof}
Suppose first that $G$ is a 6-ring and $(X_0,X_1,\ldots,X_5)$ is a good partition of $G$.
We observe that any maximal clique of $G$ must be contained in $X_i\cup X_{i+1}$ for some $i$.
Therefore, $\omega(G)=\max_{i\in \mathbb{Z}_6}\omega(G[X_i\cup X_{i+1}])$.
Since $G[X_i\cup X_{i+1}]$ is $C_4$-free cobipartite and thus chordal, $\omega(G[X_i\cup X_{i+1}])$
can be found in time linear in the size of $G[X_i\cup X_{i+1}]$~\cite{Ga72}. Hence,
the total running time is $O(m+n)$, since each vertex or edge of $G$ is counted at most twice when summing up the running
time over all $i$.

Now assume that $G$ is an $r$-lantern and $(A,B_1,\ldots,B_r,C_1,\ldots,C_r,D)$ is a good partition of $G$.
Clearly, any maximal clique of $G$ must be contained in $A\cup B_i$, $B_i\cup C_i$ or $C_i\cup D$ for some $i=1,2,\ldots,r$.
Let $a=\max_{1\le i\le r}|A\cup B_i|$,  $d=\max_{1\le i\le r}|D\cup C_i|$
and $q=\max_{2\le i\le r}|C_i\cup B_i|$.  Then $\omega(G)=\max\{a,q,d,\omega(G[B_1\cup C_1])\}$.
Since $(A,B_1,\ldots,B_r,C_1,\ldots,C_r,D)$ is given, all of $a$, $q$, $d$, and $\omega(G[B_1\cup C_1])$ can be found
in linear time (we are using the fact that $G[B_1\cup C_1]$ is $C_4$-free cobipartite, and consequently chordal). This completes the proof.
\end{proof}

We now turn to thickened emeralds and $7$-bracelets.
A graph $G$ is called a {\em circular-arc graph} if it is the intersection graph of arcs of a circle, i.e.,
each vertex $v\in V(G)$ corresponds to an arc $A_v$ on a circle in such a way that
$v$ and $u$ are adjacent if and only if $A_v$ and $A_u$ intersect on the circle.
We say that the family $\{A_v\}_{v\in V(G)}$ is a {\em circular-arc representation} of the graph $G$.
A circular-arc representation is {\em proper} if no arc  is properly contained in another.
If a graph admits a proper circular-arc representation, then it is called a {\em proper circular-arc graph}.
It is known that coloring proper circular-arc graphs can be done in $O(n^{1.5})$ time~\cite{SH89}.
We show that thickened emeralds and 7-bracelets are subfamilies of the class of proper circular-arc graphs and hence
can be colored in $O(n^{1.5})$ time.

We first observe that the emerald is a proper circular-arc graph.
It is routine to verity that the representation given in Figure \ref{fig:pcaemerald} is a proper circular-arc representation of the emerald.
Since creating a twin vertex preserves being a proper circular-arc graph (we can use the same arc to represent a pair of twin vertices),
it follows that thickened emeralds are proper circular-arc graphs. We state this below for future reference. 

\begin{lemma}\label{lem:thickenedemeraldispca}
Every thickened emerald is a proper circular-arc graph.
\end{lemma}

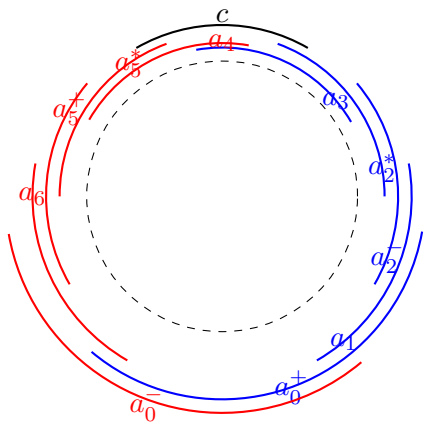
\begin{figure}
\centering
\begin{tikzpicture}[scale=.6]
\draw[dashed] (0,0) circle (3);

\draw[thick] (60:3.8) arc (60:120:3.8);
\node at (0,4) {$c$};

\def \dist{0.3}
\foreach \angle/ \i in {30/1, 0/2, -30/3, -60/4}
{
	\draw[blue, thick] (\angle: 3+\dist*\i) arc (\angle:\angle+70:3+\dist*\i);
}

\node[blue] at (30+10:3+\dist*1) {$a_3$};
\node[blue] at (0+10:3+\dist*2) {$a^*_2$};
\node[blue] at (-30+10:3+\dist*3) {$a^-_2$};
\node[blue] at (-60+10:3+\dist*4) {$a_1$};

\draw[blue, thick] (-130:4.5) arc (-130:-10:4.5);
\node[blue] at (-70:4.5) {$a^+_0$};

\foreach \angle/ \i in {110/2, 140/3, 170/4}
{
	\draw[red, thick] (\angle: 3+\dist*\i) arc (\angle:\angle+70:3+\dist*\i);
}

\draw[red, thick] (80:3.4) arc (80:150:3.4);
\node[red] at (90:3.4) {$a_4$};

\node[red] at (110+15:3+\dist*2) {$a^*_5$};
\node[red] at (140+10:3+\dist*3) {$a^+_5$};
\node[red] at (170+10:3+\dist*4) {$a_6$};

\draw[red, thick] (190:4.8) arc (190:310:4.8);
\node[red] at (250:4.9) {$a^-_0$};

\end{tikzpicture}
\caption{A proper circular-arc representation of the emerald in Figure \ref{fig:emerald}. The dashed circle is the underlying circle.}
\label{fig:pcaemerald}
\end{figure}

We now show that $7$-bracelets are also proper circular-arc graphs.

\noindent {\bf Reduction to canonical $7$-bracelets.}
A cobipartite graph $G=(X,Y)$ is {\em canonical}
if $|X|=|Y|=t$ and vertices in $X$ and $Y$ can be ordered as $x_1,\ldots,x_t$ and $y_1,\ldots,y_t$
such that
\begin{equation*}\label{equ:equ1}
N_Y(x_i)=\{y_1,\ldots,y_i\} \text{ and  } N_X(y_i)=\{x_i,\ldots,x_t\} \text{ for each } i=1,\ldots,t.
\end{equation*}
We refer to $x_i$ and $y_i$ as the {\em $i$th vertex} in $X$ and $Y$, respectively.
Here we view $(X,Y)$ as an ordered pair so that under the neighborhood containment relation
the $i$th vertex in $X$ is the $i$th smallest vertex in $X$ and the $i$th vertex in $Y$ is the $i$th largest vertex in $Y$.
The quantity $t$ is called the {\em order} of  $G$. A canonical cobipartite graph
of order $4$ is drawn in Figure \ref{fig:canonical}.

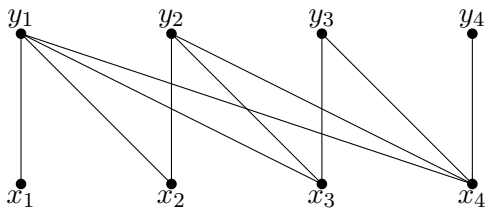
\begin{figure}
\centering
\begin{tikzpicture}[scale=2]
\foreach \i in {1,...,4}
{
	\path (\i,0) coordinate (X\i);
	\fill (X\i) circle (1pt);
	\node at (\i,-.1) {$x_{\i}$};
	
}

\foreach \j in {1,...,4}
{
	\path (\j,1) coordinate (Y\j);
	\fill (Y\j) circle (1pt);
	\node at (\j,1.1) {$y_{\j}$};
}

\foreach \i in {1,...,4}
{
	\foreach \j in {1,...,\i}
	{
		\draw (X\i) -- (Y\j);
	}
}
\end{tikzpicture}
\caption{A canonical cobipartite graph $(X,Y)$ of order $4$. For clarity, the edges with both ends in
$X$ or $Y$ are not drawn.}\label{fig:canonical}
\end{figure}

Let $B$ be a $7$-bracelet with good pair $(\{A_i\}_{i\in \mathbb{Z}_7},i^*)$ as shown in Figure \ref{fig:7bracelet-details}.
We may assume by symmetry that $i^*=1$.
We say that $B$ is a {\em canonical} $7$-bracelet if it satisfies the following two conditions.

\begin{enumerate}[\bfseries (1)]
\item Each of the sets $A_3$, $A_4$, $A_0^*$, $A_1^*$, $A_2^*$, $A_5^*$ and $A_6^*$ is a singleton.
\item The pairs $(A_5^+, A_0^-)$ , $(A_0^+,A_2^-)$ and $(A_6^+,A_1^-)$
all induce canonical cobipartite graphs of the same order $t$ with some integer $t\ge 1$.
The quantity $t$ is called the {\em order} of  $B$.
\end{enumerate}

\begin{lemma}\label{lem:reduce to canonical}
Every $7$-bracelet with no twins is an induced subgraph of a canonical $7$-bracelet.
\end{lemma}
\begin{proof}
Let $B$ be a $7$-bracelet with good pair $(\{A_i\}_{i\in \mathbb{Z}_7},0)$.
We note that by the definition all of the sets $A_3$, $A_4$, $A_0^*$, $A_1^*$, $A_2^*$, $A_5^*$ and $A_6^*$
are homogeneous cliques in $B$, and $(A_0^+,A_2^-)$, $(A_0^-,A_5^+)$ and $(A_6^+,A_1^-)$ are cobipartite graphs.
Since $B$ has no twins, it follows that each of $A_3$, $A_4$, $A_0^*$, $A_1^*$, $A_2^*$, $A_5^*$ and $A_6^*$
has size at most one, and each of $(A_0^+,A_2^-)$, $(A_0^-,A_5^+)$ and $(A_6^+,A_1^-)$ induces a canonical cobipartite graph.
If we let $t$ be the maximum order of the three canonical cobipartite graphs,
then $B$ is an induced subgraph of a canonical $7$-bracelet of  order $t$.
\end{proof}

It follows from Lemma \ref{lem:reduce to canonical} that it suffices to show that canonical $7$-bracelets are proper circular-arc graphs.

\noindent {\bf Reduction to interval representations.}
Let $B$ be a canonical $7$-bracelet with good pair $(\{A_i\}_{i\in \mathbb{Z}_7},0)$ of order $t$.
We can assume that $A_3=\{a_3\}$, $A_4=\{a_4\}$, $A_0^*=\{a_0^*\}$, $A_1^*=\{a_1^*\}$,
$A_2^*=\{a_2^*\}$, $A_5^*=\{a_5^*\}$ and $A_6^*=\{a_6^*\}$,
and each of the pairs $(A_5^+, A_0^-)$ , $(A_0^+,A_2^-)$ and $(A_6^+,A_1^-)$ induces a
canonical cobipartite graph of the same order $t$.
We now show that it is possible to reduce the problem of constructing a proper circular-arc representation
for $B$ to that of constructing a special proper interval representation for $B-a_3a_4$.
\begin{proposition}\label{prop:reduce to interval}
If there is a proper interval representation for $B-a_3a_4$ such that the interval $I_{a_4}$ is the leftmost
and the interval $I_{a_3}$ is the rightmost, then there is a proper circular-arc representation for $B$.
\end{proposition}
\begin{proof}
Let $\{I_v\}_{v\in B}$ be a proper interval representation for $B-a_3a_4$ as described in the hypothesis.
We then obtain a proper circular-arc representation for $B$ by identifying the left endpoint of $I_{a_4}$
with the right endpoint of $I_{a_3}$.
\end{proof}

\noindent {\bf Interval representation for $B-a_3a_4$.}
It follows from Proposition \ref{prop:reduce to interval} that
it suffices to construct a proper interval representation for $B-a_3a_4$ such that the interval $I_{a_4}$ is the leftmost
and the interval $I_{a_3}$ is the rightmost.
Recall that $t$ is the order of $B$.  Choose a rational number $s$ such that $0<st<1$.

Recall that each of $(A_5^+, A_0^-)$ , $(A_0^+,A_2^-)$ and $(A_6^+,A_1^-)$ induces
a canonical cobipartite graph of order $t$.  For $1\le i\le t$, let $x_i^{A_5^+}$ and
$x_i^{A_0^-}$ be the $i$th vertex in  $A_5^+$ and $A_0^-$, respectively.
The vertices $x_i^{A_0^+}$, $x_i^{A_2^-}$, $x_i^{A_6^+}$, $x_i^{A_1^-}$ are defined analogously.
We give an interval representation $\{I_v\}_{v\in V(B)}$ for $B-a_3a_4$ as follows.

\begin{enumerate}[$\bullet$]
\item $I_{a_4}=[1,4]$, $I_{a_5^*}=[3,6]$,  $I_{a_6^*}=[5,8]$, $I_{a_0^*}=[7,10]$, $I_{a_1^*}=[9,12]$, $I_{a_2^*}=[11,14]$,
$I_{a_3}=[13,16]$.
\item For each $1\le i\le t$, $I_{x_i^{A_5^+}}=[3+is,6+is]$ and $I_{x_i^{A_0^-}}=[6+is,9+is]$.
\item For each $1\le i\le t$, $I_{x_i^{A_0^+}}=[7+is,10+is]$ and $I_{x_i^{A_2^-}}=[10+is,13+is]$.
\item For each $1\le i\le t$, $I_{x_i^{A_6^+}}=[5+is,8+is]$ and $I_{x_i^{A_1^-}}=[8+is,11+is]$.
\end{enumerate}

It can be readily checked that $\{I_v\}_{v\in V(B)}$ is indeed an interval representation of $B-a_3a_4$.
Moreover, each interval has equal length $3$ and therefore the representation is proper. Thus:
\begin{lemma}\label{lem:7braceletispca}
Every $7$-bracelet is a proper circular-arc graph.
\end{lemma}

We are now ready to present our main result in this section.
\begin{theorem}
There exists an $O(n^3)$ algorithm to find a minimum coloring for any $(P_7,C_4,C_5)$-free graph
with $n$ vertices and $m$ edges.
\end{theorem}

\begin{proof}
Let $G$ be a $(P_7,C_4,C_5)$-free graph with $n$ vertices and $m$ edges, and $A$ be an atom of $G$. It suffices to show that
we can color $A$ in $O(n^2)$ time.

We first use the linear-time algorithm in~\cite{DHH96} to test if $A$ is a proper circular-arc graph.
If the answer is yes, then we use the algorithm in ~\cite{SH89} to find an optimal
coloring of $A$ in $O(n^{1.5})$ time. Otherwise $A$ is not a proper circular-arc graph.
Then we use the $O(n^2)$ time algorithm in~\cite{CliqueCutsetsBeyondChordalGraphs} to test if $A$ is a $6$-ring.
If the answer is yes, we then greedily color $A$ in $O(n+m)$ time by Lemmas \ref{lem:colorring} and \ref{lem:clique number}.
Otherwise, $G$ is not a $6$-ring and hence must be a lantern. We
find a good partition of $A$ in $O(n^2)$ time by  Lemma \ref{lem:recognize lantern}
and then color it in $O(n+m)$ time by  Lemmas \ref{lem:colorlantern} and \ref{lem:clique number}.
In any case, we can color $A$ in $O(n^2)$.
\end{proof}

We note that all of our algorithms are robust in the sense that we do
not need to assume that the input graph is in the class of graphs for which the algorithm is guaranteed to work; 
that is, given any graph as input,
each of the above algorithms either finds a correct solution, or correctly determines that the graph is not in the class.

As a byproduct of Lemma~\ref{lem:clique number},
we can also solve the maximum weight clique problem for $(P_7,C_4,C_5)$-free graphs in $O(n^3)$ time,
since the problem can be solved in linear time~\cite{BHH96} for proper circular-arc graphs.

\subsection{$\chi$-Boundedness} 

\begin{theorem} Every $(P_7,C_4,C_5)$-free graph $G$ satisfies $\chi(G) \leq \lfloor \frac{3}{2}\omega(G) \rfloor$. 
\end{theorem} 
\begin{proof} 
Let $G$ be a $(P_7,C_4,C_5)$-free graph, and assume inductively that every $(P_7,C_4,C_5)$-free graph $G'$ such that $|V(G')| < |V(G)|$ satisfies $\chi(G') \leq \lfloor \frac{3}{2}\omega(G') \rfloor$. We must show that $\chi(G) \leq \lfloor \frac{3}{2}\omega(G) \rfloor$. 

Suppose first that $G$ admits a clique-cutset, and let $(A,B,C)$ be a clique-cut-partition of $G$. Set $G_A = G[A \cup C]$ and $G_B = G[B \cup C]$. By the induction hypothesis, $\chi(G_A) \leq \lfloor \frac{3}{2}\omega(G_A) \rfloor$ and $\chi(G_B) \leq \lfloor \frac{3}{2}\omega(G_B) \rfloor$. Clearly, $\omega(G_A),\omega(G_B) \leq \omega(G)$, and we deduce that $\chi(G_A),\chi(G_B) \leq \lfloor \frac{3}{2}\omega(G) \rfloor$. We properly color $G_A$ and $G_B$ with colors $1,\dots,\lfloor \frac{3}{2}\omega(G) \rfloor$; after possibly permuting colors, we may assume that the two colorings agree on the clique $C$. The union of these two colorings is a proper coloring of $G$ that uses at most $\lfloor \frac{3}{2}\omega(G) \rfloor$ colors, and so $\chi(G) \leq \lfloor \frac{3}{2}\omega(G) \rfloor$. From now on, we assume that $G$ does not admit a clique-cutset. 

If $G$ is complete, then $\chi(G) = \omega(G) \leq \lfloor \frac{3}{2}\omega(G) \rfloor$, and we are done. So assume that $G$ is not complete. 

Next, suppose that $G$ is not anticonnected, and let $K_1,\dots,K_t$ (with $t \geq 2$) be the anticomponents of $G$. Clearly, $\chi(G) = \sum\limits_{i=1}^t \chi(K_i)$ and $\omega(G) = \sum\limits_{i=1}^t \omega(K_i)$. By the induction hypothesis, $\chi(K_i) \leq \lfloor \frac{3}{2} \omega(K_i) \rfloor$ for all $i \in \{1,\dots,t\}$, and we readily deduce that $\chi(G) \leq \lfloor \frac{3}{2}\omega(G) \rfloor$. From now on, we assume that $G$ is anticonnected. 

Suppose that $G$ contains an induced $C_7$. Since $G$ is anticonnected and does not admit a clique-cutset, Theorem~\ref{thm-decomp-P7C4C5-free-with-C7} implies that $G$ is either a 7-bracelet or a thickened emerald. It now follows from Lemmas~\ref{lem:thickenedemeraldispca} and~\ref{lem:7braceletispca} that $G$ is a circular arc graph, and so by~\cite{Karapetyan}, we have that $\chi(G) \leq \lfloor \frac{3}{2}\omega(G) \rfloor$. 

From now on, we assume that $G$ is $C_7$-free; thus, $G$ is $(P_7,C_4,C_5,C_7)$-free. Since $G$ is anticonected, is not complete, and does not admit a clique-cutset, Theorems~\ref{thm-P7C4C5C7-free-contains-Theta-decomp} and~\ref{thm-P7C4C5C7Theta33-free-decomp} and Lemma~\ref{lemma-char-P7free6ring} imply that $G$ is either a lantern or a 6-ring. But now Lemmas~\ref{lem:colorlantern} and~\ref{lem:colorlantern} imply that $\chi(G) = \omega(G) \leq \lfloor \frac{3}{2} \omega(G) \rfloor$, and we are done. 
\end{proof}

\end{document}